\documentclass[11pt]{amsart}

\title
[The Weil representation 
restricted to open compact subgroups]
{An irreducible decomposition of the Weil representation 
restricted to open compact subgroups}
\author{Yuta Takahashi}
\date{}

\usepackage{amsmath,amsfonts,amssymb,latexsym,physics}
\usepackage{mathtools,mathrsfs,nccmath}
\usepackage{amsthm}
\usepackage{dsfont}
\usepackage{cleveref}
\usepackage{thmtools, thm-restate}
\usepackage{enumitem}
\usepackage{tikz-cd}
\usepackage{arydshln}
\usepackage{indentfirst}
\usepackage{bm}
\usepackage{url}
\usepackage{yhmath}
 \usepackage{anysize}
 \marginsize{3cm}{3cm}{3cm}{3cm}

 \usepackage[
  backend=biber,
  style=alphabetic,
  sorting=nyt,
  giveninits=true,
  doi=false,
  url=false,
  isbn=false,
  eprint=false
]{biblatex}
\DeclareMathOperator{\GL}{\mathrm{GL}} 

\newcommand{\GLGL}{\GL_{r}(\mathcal{O})\times\GL_{n-r}(\mathcal{O})}
\DeclareMathOperator{\Sp}{\mathrm{Sp}}
\newcommand{\Mp}{\widetilde{\mathrm{Sp}}}
\DeclareMathOperator{\Sym}{\mathrm{Sym}} 
\DeclareMathOperator{\Mat}{\mathrm{Mat}}
\newcommand{\M}{\mathrm{Sym}_{n}^{r}(\varpi^{-1}\mathcal{O},\mathcal{O})}
\DeclareMathOperator{\val}{\mathrm{val}}
\DeclareMathOperator{\vol}{\mathrm{vol}}
\DeclareMathOperator{\modd}{\, \text{mod} \,}
\DeclareMathOperator{\supp}{\mathrm{supp}} 
\DeclareMathOperator{\smp}{\mathfrak{sp}} 
\DeclareMathOperator{\all}{\,\mathrm{for\,\,all}\,}
\DeclareMathOperator{\diag}{\mathrm{diag}}
\newcommand{\Qp}{\mathbb{Q}_p}
\newcommand{\X}{\mathcal{X}}
\newcommand{\Y}{\mathcal{Y}}
\newcommand{\E}{\mathcal{E}}
\newcommand{\F}{\mathcal{F}}

\let\SS\relax
\newcommand{\SS}{\mathrm{S}}

\newcommand{\oo}{\mathcal{O}}

\newcommand{\tp}{{}^\intercal\kern -1.2pt}
\newcommand{\vect}[2]{\left(\begin{smallmatrix}#1\\#2\end{smallmatrix}\right)}
\newcommand{\one}{\mathbf{1}}
\newcommand{\threevect}[3]
{\left(\begin{smallmatrix}#1\\#2\\#3\end{smallmatrix}\right)}
\newcommand{\xlow}{x^\circ}

\newcommand{\WC}{W_{\mathbb{C}}}
\newcommand{\gC}{\mathfrak{g}{_\mathbb{C}}}
\newcommand{\hC}{\mathfrak{h}_{\mathbb{C}}}

\newcommand{\phiaff}{\Phi^{\mathrm{aff}}}
\newcommand{\delaff}{\Delta^{\mathrm{aff}}}

\newcommand{\chit}[1]{\tilde{x}(#1)}
\newcommand{\wt}{\tilde{w}}
\newcommand{\hht}[1]{\tilde{h}(#1)}
\newcommand{\etat}{\tilde{\eta}}
\newcommand{\mphi}{\varphi^-}
\newcommand{\pphi}{\varphi^+}
\newcommand{\pmphi}{\varphi^\pm}

\newcommand{\phihat}{\hat{\varphi}}

\newcommand{\Lt}{\Lambda^{-m}_m}
\newcommand{\Ls}{\Lambda^{-m}_{m+1}}

\newcommand{\lat}[2]{\Lambda_{#1}^{#2}}

\newcommand{\tuni}[1]{\tau_{#1}}
\newcommand{\suni}[1]{\sigma_{#1}}

\newcommand{\sunihat}[1]{\hat{\sigma}_{#1}}

\newcommand{\Slow}[2]{\mathrm{S}_{#1}[#2]}

\newcommand{\Shigh}[2]{\mathrm{S}^{#1}[#2]}
\newcommand{\azero}{a_0}

\newcommand{\eps}{\varepsilon}
\let\braket\relax
\newcommand{\braket}[1]{\langle#1\rangle}
\newcommand{\addpi}[1]{\overline{#1} }

\newcommand{\Kzero}{\tilde{K}_0}
\newcommand{\Kr}{\tilde{K}_r}
\newcommand{\Kn}{\tilde{K}_n}

\newcommand{\even}{\sigma^+}
\newcommand{\odd}{\sigma^-}
\newcommand{\teven}{\sigma^+_\mathrm{ind}}
\newcommand{\todd}{\sigma^-_\mathrm{ind}}

\newtheorem{theorem}{Theorem}
\newtheorem{lemma}[theorem]{Lemma}
\newtheorem{proposition}[theorem]{Proposition}
\newtheorem{corollary}[theorem]{Corollary}

\theoremstyle{definition}

\newtheorem{remark}{Remark}

\crefname{theorem}{Theorem}{Theorems}
\crefname{lemma}{Lemma}{Lemmas}
\crefname{definition}{Definition}{Definitions}
\crefname{proposition}{Proposition}{Propositions}
\crefname{corollary}{Corollary}{Corollaries}
\crefname{section}{Section}{Sections}
\crefname{remark}{Remark}{Remarks}

\numberwithin{theorem}{section}
\numberwithin{lemma}{section}
\numberwithin{proposition}{section}
\numberwithin{corollary}{section}
\numberwithin{definition}{section}
\numberwithin{example}{section}
\numberwithin{equation}{section}

\begin{document}


\begin{abstract}
For a $p$-adic field of any residual characteristic, we determine irreducible decompositions of the Weil representation restricted to the maximal open compact subgroups and the Iwahori subgroup by using the Schr\"{o}dinger model. In particular, these decompositions show that the restrictions are multiplicity free. 
\end{abstract}

\maketitle

\section{Introduction}
Let $F$ be a non-Archimedian local field of characteristic zero, and let $\oo$ be the ring of integers. Let $W$ be a symplectic space over $F$ of dimension $2n$ and $\Sp(W)=\Sp_{2n}(F)$ the corresponding symplectic group. Fix a nontrivial additive character $\psi$ of $F$ with a conductor $4\oo$, and let $(\omega_\psi, S)$ be the Weil representation of the metaplectic group $\Mp(W)$, that is, the two fold central extension of $\Sp(W)$.
The symplectic group $\Sp_{2n}(F)$ contains $n+1$ maximal open compact subgroups up to conjugacy, and we denote each of them by $K_r$ with $0\leq r\leq n$. 
The maximal compact subgroups $\Kr$ of $\Mp(W)$ are obtained as the full-inverse image of $K_r$ with respect to the projection $\Mp(W)\rightarrow\Sp(W)$.

In the case where $p>2$, D.\ Prasad \cite{prasad1991} determined the brunching rule of $\omega_\psi$ for the hyperspecial maximal compact subgroup $\Kzero$ by using the lattice model, which was later extended to all conjugacy classes $\Kr$ by K.\ Maktouf and P.\ Torasso \cite{KM2016}.
Meanwhile, in \cite{savin2015}, G.\ Savin and A.\ Wood introduced a framework for the Weil representation in arbitrary residual characteristic by using Schr\"{o}dinger models and proved the irreducibility of a certain space. 
Based on their ideas, one of the main results in this paper is to determine the branching rule for the restriction of $\omega_\psi$ to the maximal compact subgroups $\Kr$ with arbitrary residual characteristic.

Fix a symplectic basis $\{e_1,\dots,e_n,f_1,\dots,f_n\}$ of $W$, and let $Y$ be the subspace with basis $\{f_1,\dots,f_n\}$.
For each integer $0\leq r\leq n$, define the lattice $\mathcal{L}_r$ as the $\oo$-span of the set
$\{e_1,\dots, e_n,\varpi f_1,\dots,\varpi f_r,f_{r+1},\dots,f_n\}$, and its dual lattice $\mathcal{L}^*_r$
to be the $\oo$-span of the set $\{\varpi^{-1} e_1,\dots, \varpi^{-1} e_r,e_{r+1},\dots,e_n,f_1,\dots,f_n\}$. 
The maximal compact subgroups $K_r$ of $\Sp(W)$ are characterized as the subgroup preserving the lattice $\mathcal{L}_r$ (or equivalently $\mathcal{L}_r^*$). 

Let $S=\SS(Y)$ be the space of Schwartz functions on $Y$, that is, locally constant functions with a compact support. 
We define subspaces of $\SS(Y)$ as follows.
For each integer $m\geq0$, define two subspaces
\begin{equation*}
\Slow{r}{m}=\SS(\varpi^{-m}L / 2\varpi^{m}L_r) \text{\qquad and\qquad}
\Shigh{r}{m}=\SS(\varpi^{-m}L_r / 2\varpi^{m}L),
\end{equation*}
where $L_r=\mathcal{L}_r\cap Y\cong\varpi\oo^r\times\oo^{n-r}$ and $L=\mathcal{L}^*_r\cap Y\cong\oo^n$. Also, for a function space $\mathcal{U}$, denote by $\mathcal{U}^+$ (resp.\ $\mathcal{U}^-$)
the set of even (resp.\ odd) functions in $\mathcal{U}$. In this setting, Savin and Wood \cite{savin2015} showed the following theorem.
\begin{theorem}[Savin-Wood]
    For $0\leq r\leq n$, the maximal compact subgroup $\Kr$ preserves the filtration 
\begin{equation}\label{filter}
        \Slow{r}{0}\subset\Shigh{r}{1}\subset\Slow{r}{1}\subset\Shigh{r}{2}\subset\Slow{r}{2}\subset\cdots,
\end{equation}
and acts irreducibly on the components $\Slow{r}{0}^+$ and $\Slow{r}{0}^-$.
\end{theorem}
Building upon their ideas, we establish an irreducible decomposition of the representation restricted 
to a subgroup $\Kr$.
Since $\SS(Y)=\lim_{m\rightarrow\infty}\Slow{r}{m}=\lim_{m\rightarrow\infty}\Shigh{r}{m}$, the filtration (\ref{filter}) yields
\begin{align*}
    \SS(Y)&\cong\Slow{r}{0}\oplus\bigoplus_{m\geq1}\left(\Shigh{r}{m}/\Slow{r}{m-1}\oplus\Slow{r}{m}/\Shigh{r}{m}\right)\\
    &=\Slow{r}{0}^+\oplus\Slow{r}{0}^-\oplus\bigoplus_{\substack{m\geq1\\\eps\in\{+,-\}}}
     \left(\Shigh{r}{m}^\eps/\Slow{r}{m-1}^\eps\oplus\Slow{r}{m}^\eps/\Shigh{r}{m}^\eps\right).
\end{align*}
We will show that all these components are irreducible, which provides an explicit irreducible decomposition of $\omega_\psi|_{\Kr}$ as follows.
\begin{theorem}
    The irreducible decomposition of $\omega_\psi|_{\Kr}$ is given by 
    \begin{equation*}
        \SS(Y)\cong\Slow{r}{0}^+\oplus\Slow{r}{0}^-
        \oplus\bigoplus_{m\geq1}(\E_{r,m}^+\oplus\E_{r,m}^-\oplus
        \F_{r,m}^+\oplus\F_{r,m}^-),
    \end{equation*}
    where
    \begin{align*}
        \E_{r,m}=\Shigh{r}{m}/\Slow{r}{m-1},\quad
        \F_{r,m}=\Slow{r}{m}/\Shigh{r}{m}.
    \end{align*}
    The dimension of each component is written explicitly.
    Moreover, the representation $\omega_\psi|_{\Kr}$ is multiplicity-free.
\end{theorem}
Furthermore, pushing this approach, we also determine the branching rule for restriction to the Iwahori subgroup $\tilde{I}$. The Iwahori subgroup $\tilde{I}$ coincides with the intersection of all maximal compact subgroups $\Kr$. Hence, we see that $\tilde{I}$ preserves subspaces 
$\Slow{r}{m}$ and $\Shigh{r}{m}$ for all $0\leq r\leq n$ and $m\geq 0$, namely, 
\begin{equation}
\begin{array}{ccccccccccc}
\omega_\psi|_{\Kzero}:&&\Slow{0}{0} &\subset& \Shigh{0}{1} &=& \Slow{0}{1} &\subset&\cdots&&\\
&&\cap && \cup && \cap &&\\
\omega_\psi|_{\tilde{K}_1}:&&\Slow{1}{0} &\subset& \Shigh{1}{1} &\subset& \Slow{1}{1}&\subset&\cdots&&\\
&&\cap && \cup && \cap &&\\
&&\vdots && \vdots && \vdots &&\vdots\\
&&\cap && \cup && \cap &&\cup\\
\omega_\psi|_{\tilde{K}_{n-1}}:&&\Slow{n-1}{0} &\subset& \Shigh{n-1}{1} &\subset& \Slow{n-1}{1}&\subset&\Shigh{n-1}{2}&\subset&\cdots\\
&&\cap && \cup && \cap &&\cup\\
\omega_\psi|_{\Kn}:&&\Slow{n}{0} &=& \Shigh{n}{1} &\subset& \Slow{n}{1}&=&\Shigh{n}{2}&\subset&\cdots\rlap{\quad.}\\
\end{array}
\end{equation}
From this, we can make the finer filtration from vertical inclusions:
\begin{align*}
    \Slow{0}{0}\subset\Slow{1}{0}\subset\cdots\subset\Slow{n}{0}=\Shigh{n}{1}\subset\Shigh{n-1}{1}\subset\cdots\subset\Shigh{0}{1}
    =\Slow{0}{1}\subset\Slow{1}{1}\subset\cdots.
\end{align*}
Therefore, we consider the representation on the quotient space of two adjacent components. We define $\tuni{m}=\SS(\varpi^{-m}\oo/2\varpi^m\oo)$ and $\suni{m}=\SS(\varpi^{-m}\oo/2\varpi^{m+1}\oo)$. Then, the quotient spaces are written as follows:
\begin{alignat*}{2}
    \Slow{r}{m}/\Slow{r-1}{m}
    &\cong
    \suni{m}^{\otimes r-1}\otimes\suni{m}/\tuni{m}\otimes\tuni{m}^{\otimes n-r}
    &&\cong(\tuni{m}\oplus\suni{m}/\tuni{m})^{\otimes r-1}\otimes\suni{m}/\tuni{m}\otimes\tuni{m}^{\otimes n-r};
    \\
    \Shigh{r}{m}/\Shigh{r+1}{m}
    &\cong\suni{m}^{\otimes r}\otimes\tuni{m+1}/\suni{m}\otimes\tuni{m+1}^{\otimes n-r-1}
    &&\cong\suni{m}^{\otimes r}\otimes\tuni{m+1}/\suni{m}\otimes(\suni{m}\oplus\tuni{m+1}/\suni{m})^{\otimes n-r-1}.
\end{alignat*}
We show that, after we expand the tensor product, each summand  decomposes into odd and even spaces, giving irreducible components. Moreover, these components are multiplicity-free. Namely we have 

\begin{theorem}
The irreducible decomposition of $\omega_\psi|_{\tilde I}$ is given by
\[
    \SS(Y)\cong\tuni{0}^{\otimes n}
    \oplus\bigoplus_{m\geq0}
    \left(
        \X_m^+\oplus\X_m^-
        \oplus\Y_m^+\oplus\Y_m^-
    \right),
\]
where $\X_m$ and $\Y_m$ are defined by
\begin{align*}
    \X_m\oplus\tuni{m}^{\otimes n}
    &=
    \left(\tuni{m}\oplus\suni{m}/\tuni{m}\right)^{\otimes n},\\
    \Y_m\oplus\suni{m}^{\otimes n}
    &=
    \left(\suni{m}\oplus\tuni{m+1}/\suni{m}\right)^{\otimes n}.
\end{align*}
The dimension of each component is written explicitly.
Moreover, $\omega_\psi|_{\tilde I}$ is multiplicity-free.
\end{theorem}

The layout of this paper is as follows. In section 2, we construct maximal compact subgroups of $\Sp(W)$ and the Iwahori subgroup, and we state a structure of these subgroups. In section 3, we review Schwartz functions, Fourier transform, and even/odd functions. In section 4, we define the Weil representation of metaplectic group $\Mp(W)$ and its Schir\"{o}dinger model. In section 5 we establish some properties for additive character $\psi$.
In section 6, we determine the branching rules for the Weil representation $\omega_\psi$ restricted to maximal open compact subgroups $\Kr$. Finally, in section 7, we show the irreducible decomposition of the Weil representation restricted to the Iwahori subgroup.

\subsection*{Acknowledgments}
This paper is an expanded version of the author's master's thesis at the University of Osaka. The author would like to express sincere gratitude to his thesis advisor, Shuichiro Takeda, for suggesting this problem along with many helpful discussions, valuable comments, and constant encouragement throughout the course of this work.
\subsection*{Notation}
Throughout this paper, we denote by $F$ a non-Archimedian local field of characteristic 
zero, that is, a finite extension of the $p$-adic field $\Qp$. 
Let $\oo$ be the ring of integers of $F$ and $\varpi$ be the uniformizer of $\oo$. So $\varpi\oo$ is the maximal ideal of $\oo$.
We denote by $q$ the order of the residue field $\oo/\varpi\oo$, and let $e$ be the valuation of $2$ in $F$. That is, $e$ is the ramification index of $2$ if $p=2$, and $e=0$ otherwise.

Let $W$ be the $2n$-dimensional symplectic vector space over $F$ which is equipped with a symplectic basis $\{e_1,\dots,e_n,f_1,\dots,f_n\}$.
Let $\Sp(W)$ be the symplectic group associated with $W$.
We denote by $\Sym_n(\oo)$ the set of $n\times n$ symmetric matrices with entries in $\oo$. Let $E_{ij}$ be the $n \times n$ matrix with a $1$ in the $ij$-position and $0$ elsewhere.

We write $\tp X$ for the transpose of a matrix $X$, and we regard vectors as column vectors. For a vector $x$, we often write $x=\vect{x^{(1)}}{x^{(2)}}$, separating the first and second components. Where to split the vector depends on the context. For example, in the representation of $\Kr$, 
we set
$$
x^{(1)}=\threevect{x_1}{\vdots}{x_r}
\textup{ and }
x^{(2)}=\threevect{x_r}{\vdots}{x_n}.
$$
For a vector $x=(x_i)_{i=1}^n\in F^n$, we define the valuation of $x$ by
$\val(x)=\min\{\val(x_1),\ldots,\val(x_n)\}$.

We fix a nontrivial, smooth, additive character $\psi$ of $F$. Assume that the conductor
of $\psi$ is $2e$: equivalently, $\psi(tx)=1$ for all $t\in\oo$ if and only if $x\in\varpi^{2e}\oo=4\oo$.

For a set $A$, we denote the characteristic function of $A$ by $\one _A$.

\section{Compact subgroups of $\Sp(W)$}\label{Compact-subgroups}
Let $W$ be the symplectic vector space over $F$ with symplectic basis 
$
\{e_1,\dots,e_n,f_1,\dots,f_n\}.
$
Denote by $X$ and $Y$ the subspaces of $W$ which is spanned by 
$\{f_1,\dots,f_n\}$ 
and $\{e_1,\dots,e_n\}$, respectively. Then, we have the polarization $W=X\oplus Y$.
The symplectic group $\Sp(W)$ is the set of automorphisms that 
preserve a symplectic form $\Omega$.
With respect to the fixed symplectic basis, the group $\Sp(W)=\Sp_{2n}(F)$ is represented in $\GL_{2n}(F)$ as
\begin{equation*}
  \Sp(W)=\left\{
  \begin{bmatrix}
    a&b\\
    c&d
  \end{bmatrix}
  :\begin{aligned}
    \tp ad-\tp cb=1\\
    \tp ac-\tp ca=0\\
    \tp bd-\tp db=0
  \end{aligned}
  \right\}.
\end{equation*}
To investigate the structure of maximal compact subgroups of $\Sp(W)$,
we consider their construction using the symplectic Lie algebra $\smp(\WC)$. From this, we obtain
the generators of these compact subgroups.
\subsection{Root system from Lie algebra}
We review the construction of a root system from a complex Lie algebra.
Let $\gC$ be a semisimple Lie algebra over $\mathbb{C}$ and $\hC$ be its Cartan subalgebra.
Define the root spaces
\begin{equation*}
  \mathfrak{g}_\alpha=\{X\in\gC:[H,X]=\alpha(H)X\all H\in\hC\}.
\end{equation*}
This yields the Cartan decomposition 
\begin{equation}\label{cartan-decomposition}
  \gC=\hC\oplus\bigoplus_{\alpha\in\Phi}\mathfrak{g}_\alpha,
\end{equation}
where $\Phi=\{\alpha\in\hC:\mathfrak{g}_\alpha\ne0\}$.
The set $\Phi$, called the root system of $\gC$ relative to $\hC$, is finite. 
We note that the space $\mathfrak{g}_\alpha$ is one-dimensional, spanned by an element denoted by $X_\alpha$. 
In addition, the Cartan subalgebra $\hC$ has a basis $\{Z_\delta\}$ indexed by $\delta\in\Delta$, the set of simple roots.
Let $\phiaff$ be the set of affine roots given by 
\begin{equation*}
    \phiaff=\{\alpha+m\colon\alpha\in\Phi,\, m\in\mathbb{Z}\},
\end{equation*}
where each affine functional $\alpha+m\colon\hC\rightarrow\mathbb{C}$ is defined by
$(\alpha+m)(x)=\alpha(x)+m$. 
Affine roots can be regarded as transformations of an affine space. Let $\mathcal{A}$ be a Euclidean space containing the vectors $\Check{\alpha}\in\Check{\Phi}$. Then, each
$\alpha+m\in\phiaff$ determines a hyperplane:
$$
H_{\alpha+m}=\{x\in \mathcal{A}:(\alpha+m)(x)=0\}.
$$
From these hyperplanes, we obtain the maximal compact subgroups of the Chevalley group. Hyperplanes associated with simple affine
roots determine certain subgroups. Another choice of simple affine roots gives another set of maximal compact subgroups, up to conjugacy.

\subsection{Chevalley groups}
To construct a group over $F$ from the symplectic Lie algebra over $\mathbb{C}$, we employ the exponential map.
Recall that the Cartan decomposition ($\ref{cartan-decomposition}$) provides 
elements $X_{\alpha}$ for $\alpha\in\Phi$ and $Z_{\delta}$ for $\delta\in\Delta$.
Moreover, we can choose $X_{\alpha}$ and $Z_{\delta}$ such that
their entries are in $\mathbb{Z}$.
Then, we define the $\mathbb{Z}$-form (or integer span) 

\begin{equation*}
    \mathfrak{g}=\Biggl(\bigoplus_{\alpha\in\Phi}\mathbb{Z}\cdot X_{\alpha}\Biggr)
    \oplus\Biggl(\bigoplus_{\delta\in\Delta}\mathbb{Z}\cdot H_{\delta}\Biggr),
\end{equation*} 
and we also define the exponential map $\exp\in\text{End}(\mathfrak{g}\otimes F)$ as
\begin{equation*}
    \exp(tX_\alpha)=1+tX_\alpha+\frac{(tX_\alpha)^2}{2}+\cdots.
\end{equation*}
Since there exists $n\in\mathbb{N}$ such that $X_{\alpha}^n=0$, this series is well-defined. We write $ x_{\alpha}(t)=\exp(tX_{\alpha})$.
The inverse of $ x_{\alpha}(t)$ is given by $ x_{\alpha}(-t)$, and the Chevalley group $G$ is defined as 
\begin{equation*}
    \langle x_{\alpha}(t)\colon t\in F,\, \alpha\in\Phi\rangle.
\end{equation*}
We note that $ x_{-\alpha}(t)$ agrees with the transpose of $ x_{\alpha}(t)$.
Moreover, define $w_{\alpha}(t)$ and
$h_{\alpha}(t)$ as follows:
\begin{align*}
    w_{\alpha}(t)&= x_{\alpha}(t) x_{-\alpha}(-t^{-1}) x_{\alpha}(t),\\
    h_{\alpha}(t)&=w_{\alpha}(t)w_{\alpha}(-1).
\end{align*}

\begin{proposition}\label{chevalley_relation}
    For any roots $\alpha, \beta\in\Phi$, one has the following relations:
\begin{enumerate}[label={\textup{(R\arabic*)}}]
  \item $x_\alpha(t+u)=x_\alpha(t)x_\alpha(u)$;
  \item $[x_\alpha(t),x_\beta(u)]=\prod x_{i\alpha+j\beta}(c_{ij}t^iu^j)$, for $\alpha+\beta\ne0$;
  \item $w_\alpha(t)x_\alpha(u)w_\alpha(-t)=x_{-\alpha}(-t^2u)$;
  \item $h_\alpha(t)x_\beta(u)h_\alpha(t)^{-1}=x_{\beta}(t^{\langle \beta,\alpha \rangle}u)$;
  \item $h_\alpha(t)h_\alpha(u)=h_\alpha(tu)$.
\end{enumerate}
\end{proposition}
In (R2), the symbol $[\ ,\ ]$ is the commutator, and
the product on the right-hand side is taken over all roots of the form $i\alpha+j\beta\in\Phi$, and $c_{ij}$ are integers depending only on $\alpha$ and $\beta$.
Relations (R3) and (R4) are consequences of (R1) and (R2); 
details can be found in \cite{steinberg}, p.23.
Define the Cartan subgroup $T$ of $G$ to be the subgroup generated by 
$\{h_{\alpha}(s)\colon\alpha\in\Phi,\, s\in F^\times\}$. Note that $T$ is abelian by the relation (R5). 
Define the following subgroups of the Chevalley group $G$: 
\begin{align*}
    \mathfrak{X}_{\alpha}(\varpi^m\oo)&=\langle x_{\alpha}(\varpi^mt)\colon t\in\oo\rangle;\\
    T(\oo)&=\langle h_{\alpha}(s)\colon s\in\oo^\times,\, 
    \alpha\in\Phi\rangle.
\end{align*}
For an affine root $\psi=\alpha+m\in\phiaff$, define the root subgroup $\mathfrak{X}_{\psi}=\mathfrak{X}_{\alpha}(\varpi^m\oo)$.
For each point $x\in \mathcal{A}$, the parahoric subgroup
is defined as
\begin{equation*}\label{first_parahoric}
    G_x=\langle T(\oo),\, \mathfrak{X}_{\psi}\colon\psi\in\phiaff,\,
    \psi(x)\geq0\rangle,
\end{equation*}
or equivalently
\begin{equation*}\label{second_parahoric}
    G_x=\langle  x_{\alpha}(\varpi^{-\lfloor\alpha(x)\rfloor}t),\, h_{\alpha}(s)\,
    \colon t\in\oo,\, s\in\oo^{\times},\,\alpha\in\Phi\rangle.
\end{equation*}

If $x$ lies on the intersection of $n-1$ hyperplanes, then the subgroup $G_x$ is maximal parahoric, and this is a maximal open
compact subgroup of $G$. On the other hand, if $x$ does not lie on any hyperplane $H_{\psi}$ for $\psi\in\phiaff$, the group $G_x$ is minimal, called an Iwahori subgroup. 

From now on, we treat the case of the symplectic Lie algebra.
In this case the Chevalley group $G$ coincides with the symplectic group $\Sp(W)$ over $F$. 
Let $\WC=X_{\mathbb{C}}+Y_{\mathbb{C}}$ be the $2n$-dimensional symplectic vector space over $\mathbb{C}$.
The symplectic Lie algebra $\gC=\smp(\WC)$ with respect to a symplectic form $\Omega$ is defined by 
\begin{equation*}
  \smp(\WC)=\left\{T\in \mathrm{End}(\WC): \Omega(Tu,v)+\Omega(u,Tv)=0 \text{\, for all\,} u,v\in\WC \right\}
\end{equation*}
equipped with the Lie bracket $[T_1,T_2]=T_1T_2-T_2T_1$. Choosing $\Omega$ such that the associated matrix is $\smqty(0 & I \\ -I & 0)$, so that the symplectic Lie algebra is given explicitly by
\begin{equation*}
  \smp(\WC)=\left\{\left(\mqty{a&b\\c&{-}\tp a}\right)\colon b=\tp b,\quad c=\tp c\right\}.
\end{equation*}
In what follows, we compute matrix entries with respect to this standard form.
The Cartan subalgebra $\hC$ of $\gC$ consists of the diagonal matrices
\begin{equation*}
  \hC=\left\{
\mathrm{diag}\{a_1,\dots, a_n,-a_1,\dots, -a_n\}
:a_i\in\mathbb{C}\right\}
\cong\mathbb{C}^n.
\end{equation*}
The root space $\Phi$ is given by
\begin{equation*}
  \Phi=\{\pm(\eps_i\pm\eps_j):1\leq i<j\leq n\}\cup\{\pm2\eps_i:1\leq i \leq n\}
\end{equation*}
with $2n$-elements. Here, $\{\eps_i\}_{i=1}^n$ is the standard dual basis of $\hC$, that is, $\eps_i(a)=a_i$ for all $i$.
We choose the set of positive roots $\Phi^+$ and the negative roots $\Phi^-$ as
\begin{align*}
  \Phi^+&=\{\eps_i\pm\eps_j:1\leq i<j\leq n\}\cup\{2\eps_i:1\leq i \leq n\},\\
\Phi^-&=\{-(\eps_i\pm\eps_j):1\leq i<j\leq n\}\cup\{-2\eps_i:1\leq i \leq n\}
\end{align*}
(more simply, $\Phi^-=\{-\alpha:\alpha\in\Phi^+\}$). The set of simple roots $\Delta=\{\alpha_1,\dots,\alpha_n\}$ is given by 
\begin{align*}
  \alpha_i &=\eps_i-\eps_{i+1} \text{\quad for\quad} 1\leq i\leq n-1;\\
  \alpha_n&=2\eps_n.
\end{align*}
Next, we define the exponential map in our setting. Since  $X_\alpha^2=0$ for all $\alpha\in\Phi$, the exponential map is 
\begin{equation*}
    \exp(tX_\alpha)=1+tX_\alpha.
\end{equation*}
Accordingly, we define $ x_\alpha(t)=1+tX_\alpha$.
Explicitly, the matrices $ x_\alpha(t)$ are given by:

\begin{align*}
     x_{\eps_i-\eps_j}(t)&=
    \begin{bmatrix}
        1+tE_{ij}&0\\
        0&1-tE_{ji}
    \end{bmatrix},&
     x_{\eps_i+\eps_j}(t)=&
    \begin{bmatrix}
        1&t(E_{ij}+E_{ji})\\
        0&1
    \end{bmatrix},\\
     x_{2\eps_i}(t)&=
    \begin{bmatrix}
        1&tE_{ii}\\
        0&1
    \end{bmatrix},&
     x_{-2\eps_i}(t)=&
    \begin{bmatrix}
        1&0\\
        tE_{ii}&1
    \end{bmatrix},
\end{align*}
where $E_{ij}$ is the matrix with $1$ at the $(i,j)$-entry and $0$ otherwise. Recall that 
$ x_{-\alpha}(t)=\tp x_{\alpha}(t)$.
Thus, the Chevalley group $G=\langle x_{\alpha}(t)\colon t\in F,\, \alpha\in\Phi\rangle$
coincides with to the symplectic group $\Sp(W)$ over $F$.
Moreover, the elements 
$w_{\alpha}(t)= x_{\alpha}(t) x_{-\alpha}(-t^{-1}) x_{\alpha}(t)$
are given explicitly as follows:
\begin{align*}
w_{\eps_i-\eps_j}(t)&=
    \begin{bmatrix}
        D_{i,j}+tE_{ij}-t^{-1}E_{ji}&0\\
        0&D_{i,j}-tE_{ji}-t^{-1}E_{ij}
    \end{bmatrix},\\
w_{\eps_i+\eps_j}(t)&=
    \begin{bmatrix}
        D_{i,j}&t(E_{ji}+E_{ji})\\
        -t^{-1}(E_{ij}+E_{ji})&D_{i,j}
    \end{bmatrix},\\
w_{2\eps_i}(t)&=
\begin{bmatrix}
    D_i&tE_{ii}\\
    -t^{-1}E_{ii}&D_i
\end{bmatrix},
\end{align*}
and the elements $h_{\alpha}(t)=w_{\alpha}(t)w_{\alpha}(-1)$
are of the following form:
\begin{align*}
h_{\eps_i-\eps_j}(t)&=
    \begin{bmatrix}
        D_{i,j}+tE_{ii}+t^{-1}E_{jj}&0\\
        0&D_{i,j}-t^{-1}E_{ii}-tE_{jj}
    \end{bmatrix},\\
h_{\eps_i+\eps_j}(t)&=
\begin{bmatrix}
    D_{i,j}+t(E_{ii}+E_{jj})&0\\
    0&D_{i,j}+t^{-1}(E_{ii}+E_{jj})
\end{bmatrix},\\
h_{2\eps_i}(t)&=
\begin{bmatrix}
    D_i+tE_{ii}&0\\
    0&D_i+t^{-1}E_{ii}
\end{bmatrix}.
\end{align*}
Here $D_{i}$ (resp.\ $D_{i,j}$) denotes the matrix obtained from the identity matrix 
by replacing the $i$-th (resp.\ $i$ and $j$-th) diagonal entry with  $0$. 

The set of simple affine roots $\delaff$ is obtained by adding the affine root 
$\alpha_0=-2\eps_1+1$ to simple roots $\Delta$. In other words, one has $\delaff=\{\alpha_0,\alpha_1,\dots,\alpha_n\}$. Each of these roots defines a hyperplane in the affine space $\mathcal{A}$:
\begin{align*}
    H_{\alpha_0}&=\{a\colon -2a_1+1=0\},\\
    H_{\alpha_i}&=\{a\colon a_i-a_{i+1}=0\}
    \quad (1\leq i\leq n-1),\\
    H_{\alpha_n}&=\{a\colon 2a_n=0\},
\end{align*}
where $a=(a_1,\dots, a_n)\in\mathcal{A}$.
These hyperplanes are not parallel, and the space $\mathcal{A}$ is $n$-dimensional. So any $n-1$ hyperplanes intersect at a single point. Therefore we define the vertex $z_i$ such that $z_i$ is the all common intersection of hyperplanes except $H_{\alpha_i}$. 
Explicitly, the vertex can be written as follows:
\begin{equation*}
    z_0=\left(0,\dots,0\right),\quad z_1=\left(\frac{1}{2},0,\dots,0\right),\quad\dots,\quad
    z_n=\left(\frac{1}{2},\dots,\frac{1}{2}\right).
\end{equation*}
\subsection{Maximal compact subgroups}\label{cpt-subgp}
As discussed above, the group $G_{z_r}$ determines a maximal compact subgroup of $\Sp(W)$.
We recall that the maximal compact subgroups $K_r=G_{z_r}$ are defined by 
    \begin{equation*}
    G_{z_r}=\langle  x_{\alpha}(\varpi^{-\lfloor\alpha(z_r)\rfloor}t),\, h_{\alpha}(s)\,
    \colon t\in\oo,\, s\in\oo^{\times},\,\alpha\in\Phi\rangle.
\end{equation*}
To write these generators more explicitly, we record the values
of $\lfloor\alpha(z_r)\rfloor$:
\begin{align*}
    \lfloor\eps_i-\eps_j(z_r)\rfloor&=0,
    &\lfloor-(\eps_i-\eps_j(z_r))\rfloor&=
    \begin{cases}
        -1 &\textup{if }i\leq r<j\\
        0 &\textup{otherwise}
    \end{cases},\\
    \lfloor\eps_i+\eps_j(z_r)\rfloor&=
    \begin{cases}
        1 &\textup{if }i< r\leq j \\
        0 &\textup{otherwise}
    \end{cases},
    &\lfloor-(\eps_i+\eps_j(z_r))\rfloor&=
    \begin{cases}
        0 &\textup{if }r<i<j\\
        -1 &\textup{otherwise}
    \end{cases},\\
    \lfloor2\eps_i(z_r)\rfloor&=
    \begin{cases}
        1 &\textup{if }i\leq r\\
        0 &\textup{if }r<i
    \end{cases},
    &\lfloor-2\eps_i(z_r)\rfloor&=
    \begin{cases}
        -1 &\textup{if }i\leq r\\
        0 &\textup{if }r< i
    \end{cases}.
\end{align*}
We now show that the subgroup $K_r$ is generated by three elements:
$x(a)$, $h(b)$ and $\eta_rw$, which we now describe.
Let $\M$ be the set of $n\times n$ symmetric matrices of the form
\begin{equation*}
\left[\begin{array}{ccc:ccc}
     &&&&&  \\
     &\Sym_r(\varpi^{-1}\oo)&&&\oo&\\
     &&&&& \\
     \cdashline{1-6}
     &&&&&  \\
     &\oo&&&\oo&  \\
     &&&&&  \\
\end{array}  
\right],
\end{equation*}
and let $\GLGL\subset\GL_n(\oo)$ be the set of matrices embedded diagonally.
Let $\eta_r=\diag(\mathbf{m}_r(\varpi)^{-1},\mathbf{m}_r(\varpi))$ with $\mathbf{m}_r(\varpi)=\diag(\varpi,\ldots,\varpi,1,\ldots,1)\in\varpi\oo^r\times\oo^{n-r}$ and let $w=\begin{bmatrix}
    0&1\\
    -1&0
\end{bmatrix}$.
Then, the following lemma holds:

\begin{lemma}\label{max-cpt-subgp-generator}
    Let $K_r$ be the maximal compact subgroup of $\Sp(W)$ defined above.
    The following three elements generate $\Kr$:
    \begin{alignat*}{2}
    x(a)&=
    \begin{bmatrix}
        1&a\\
        0&1
    \end{bmatrix}\quad \textup{for}\quad a\in \M;\\
     h(b)&=
     \begin{bmatrix}
         b&0\\
         0&\tp b^{-1}
     \end{bmatrix}\quad \textup{for}\quad b\in \GLGL;\\
     \eta_rw&=
     \begin{bmatrix}
         0&\mathbf{m}_r(\varpi)^{-1}\\
         -\mathbf{m}_r(\varpi)&0         
     \end{bmatrix}.
\end{alignat*}
\end{lemma}
\begin{proof}
We first verify that three elements can be expressed as products of Chevalley generators.
The $(i,j)$-entry of a matrix $a\in\M$ corresponds to the root $\eps_i+\eps_j$. More precisely,
\begin{equation*}
    x(a)= \prod_{i<j}x_{\eps_i+\eps_j}(a_{ij})\prod_{i}x_{2\eps_i}(a_{ii}).
\end{equation*}
We will check that $a_{ij}$ can be taken from $\varpi^{-1}\oo$ if and only if $i\leq j\leq r$. Now, the value $\lfloor\alpha(z_r)\rfloor$ is equal to $\lfloor\eps_i(z_r)+\eps_j(z_r)\rfloor$. So this value is $1$ if and only if $\eps_i(z_r)=\eps_j(z_r)=1/2$.

Next, we give a construction of $h(b)$ (the cases $r=0,n$ being similar).
We set $b=\diag(b_1,b_2)$
for $b_1\in\GL_r(\oo)$ and $b_2\in\GL_{n-r}(\oo)$.
The matrix $b_1$ can be transformed into an upper triangle matrix via elementary row operations corresponding to multiplication by:
\begin{align*}
    1+E_{ij}+E_{ji}\quad &\textup{ for } i\ne j;\\
    1+tE_{ij}\quad&\textup{ for } i\ne j \textup{ and }t\in\oo;\\
    \diag(1,\ldots,1,s,1\ldots,1)\quad &\textup{ for } s\in\oo^\times.
\end{align*}
These matrices correspond to the roots $x_{\eps_i-\eps_j}(1)x_{\eps_j-\eps_i}(1)$, $x_{\eps_i-\eps_j}(t)$ and $h_{2\eps_i}(s)$, respectively.
We note that $x_{\eps_i-\eps_j}(1)$ is in $K_r$ (i.e. we can apply $t=1$) because we assume that $i,j\leq r$.
Moreover, an upper triangle matrix is decomposed as the multiple of a diagonal matrix and a unipotent matrix, which correspond to $h_{w_i}(s)$ and $x_{\eps_i-\eps_j}(t)$ respectively.

Finally, the Weyl element $\eta_rw$ is represented by 
long roots $w_{2\eps_i}(t)$.
The value of $2\eps_i(z_r)$ is $1$ when $1\leq i\leq r$ and $1/2$ when $r+1\leq i\leq n$. So the product $\eta_rw=w_{\eps_1}(\varpi^{-1})\cdots w_{\eps_r}(\varpi^{-1})w_{\eps_{r+1}}(1)\cdots w_{\eps_n}(1)$ lies in $K_r$.

Conversely, we have to check that each Chevalley generator can be represented by the three elements.
It is easy to see that the elements $x_{\alpha}(\cdot)$ with roots $\alpha=\eps_i+\eps_j,\,2\eps_i$ are written as the element $x(a)$ for some $a\in \M$. If $\alpha=-(\eps_i+\eps_j),\,-2\eps_i$, the matrix $x_{\alpha}(\cdot)$ is the transpose of $x_{-\alpha}(\cdot)$ and this equals to 
$wx_{-\alpha}(\cdot)w$ because upper left and lower right components of $x_{-\alpha}$ is identity matrix.
For $\alpha=\pm(\eps_i-\eps_j)$ and $h(s)$, we also see that 
these matrices are represented as $h(b)$.
Therefore, the subgroup $K_r$ is generated by $x(a)$, $h(b)$ and $\eta_rw$.
\end{proof}

\subsection{Iwahori subgroup}
Recall that the Iwahori subgroup is the minimal parahoric subgroup $G_x$ associated with a point $x$ that does not lie on any hyperplane $H_{\alpha_i}$.
One such point is given by $x=
    (1,\frac{1}{2},\ldots,\frac{1}{2^{n-1}})$.
Then, the values of $\alpha(x)$ for this $x$ are
\begin{align*}
    \lfloor\eps_i-\eps_j(x)\rfloor&=0,
    &\lfloor-(\eps_i-\eps_j(x))\rfloor&=-1,\\
    \lfloor\eps_i+\eps_j(x)\rfloor&=0,
    &\lfloor-(\eps_i+\eps_j(x))\rfloor&=-1,\\
    \lfloor2\eps_i(x)\rfloor&=0,
    &\lfloor-2\eps_i(x)\rfloor&=-1.
\end{align*}
Hence, the Iwahori subgroup $I$ is generated by
$x_{\eps_i-\eps_j}(t)$, $x_{-(\eps_i-\eps_j)}(\varpi t)$,
$x_{2\eps_i}(t)$, $x_{2\eps_i}(\varpi t)$,
$x_{\eps_i+\eps_j}(t)$ and $x_{-(\eps_i+\eps_j)}(\varpi t)$.
We note that the Iwahori subgroup $I$ is contained in every maximal compact subgroup $K_r$.
\begin{lemma}
    The Iwahori subgroup $I$ is generated by 
    $x(a)$ for $a\in\Sym_n(\oo)$, $h(b)$ for $b=(b_1,\ldots,b_n)\in(\oo^\times)^n$, $wx(\varpi a)w$, $x_{\eps_i-\eps_j}(t)$ and 
    $x_{\eps_j-\eps_i}(\varpi t)$. 
    Consequently, $I$ consists of the matrices of the form 
\begin{equation}\label{iwahori-coefficient}
\begin{bmatrix}
\oo^\times & \oo & \cdots & \oo\\
\varpi\oo & \oo^\times & \cdots & \oo\\
\vdots & \vdots & \ddots & \vdots\\
\varpi\oo & \varpi\oo & \cdots & \oo^\times
\end{bmatrix}.
\end{equation}
\end{lemma}
\begin{proof}
    First, we show that the above five elements 
    can be written as the products of Chevalley generators $x_\alpha(\cdot)$.
    The discussion of $x(a)$ is the same to that of \cref{max-cpt-subgp-generator}.
    The element $h(b)$ with $b=\diag(b_1,\ldots,b_n)$ is represented as $h_{2\eps_1}(b_1)\cdots h_{2\eps_n}(b_n)$.
    In the end, since one has $wx(\varpi a)w=\tp x(\varpi a)$,
    we have
    \begin{equation*}
    x(\varpi a)= \prod_{i<j}x_{-(\eps_i+\eps_j)}(\varpi a_{ji})\prod_{i}x_{-2\eps_i}(\varpi a_{ii}).
    \end{equation*}
    
    On the other hand, it is easy to see that each element $x_\alpha(\cdot)$ is generated by the above five elements.
    The two elements $x_{\eps_i+e_j}(t)$ and $x_{2\eps_i}(t)$ equals to $x(a)$ for some $a\in\M$, and one has $h_{2\eps_i}(s)=h(1,\ldots, s, 1,\ldots,1)$.
    Finally, the elements $x_{\eps_i+\eps_j}(\varpi t)$ and $x_{2\eps_i}(\varpi t)$ are represented by $wx(\varpi a)w=\tp x(\varpi a)$ for some $a$. The proof is complete.
\end{proof}
The Iwahori subgroup contains smaller subgroups isomorphic to Iwahori subgroups of lower-rank symplectic groups.
For $1\leq n'< n$, 
choose $n'$ numbers $k_1<\cdots<k_{n'}$ from $\{1,\ldots, n\}$ and let $N=\{k_1,\ldots,k_{n'}\}$.
Let $x'=(x_{k_i})_{k_i\in N}$ and $\Phi^N\subset\Phi$ be the root subspace  
\begin{align*}
  \Phi^N&=\{\pm(\eps_{k_i}\pm\eps_{k_j}):1\leq i<j\leq n'\}
  \cup
  \{
  \pm2\eps_{k_i}:1\leq i \leq n'
  \}.
\end{align*}
In this setting, we can regard $\Phi^N$ as 
a root system of rank $n'$.
We see that the point $x'$ does not lie on any hyperplane
$H_{\beta_i}$ with
$\beta_0=-2\eps_{k_1}+1$, $\beta_i=\eps_{k_i}-\eps_{k_{i+1}}$ and $\beta_{n'}=2\eps_{k_{n'}}$.
Therefore, the subgroup
$$
I^{(N)}=\langle  x_{\beta}(\varpi^{-\lfloor\beta(x')\rfloor}t),\, h_{\beta}(s)\,
    \colon\,\beta\in\Phi^N\rangle\subset I
$$
has the structure of an Iwahori subgroup. This subgroup is isomorphic to an Iwahori subgroup of $\Sp_{2n'}(W)$. We use this idea in \cref{iwahori3}.

We will see that $I$ contains two smaller Iwahori subgroups. We set $N=\{1,2,\ldots,n'\}$ and $\bar{N}=\{n'+1,\ldots,n\}$.
As we saw above, the group $I$ contains the Iwahori $I^{(N)}$ and $I^{(\bar{N})}$. Moreover, for $\beta\in\Phi^{N}$, 
we have seen that 
$ x_{\beta}(t)$ and $h_{\beta}(s)$ take the following form:
\begin{equation}\label{matrix-form}
\left[
\begin{array}{cccccc:cccccc}

&&\multicolumn{1}{c:}{}& && & &&& &&\\ 
&*&\multicolumn{1}{c:}{}& && & &&& &&\\ 
&&\multicolumn{1}{c:}{}& && & &&& &&\\ \cdashline{1-3}
&&& 1&& & &&& &&\\
&&& &\ddots& & &&& &&\\

&&& &&1 & &&& &&\\
\hdashline
& & & & & &&&\multicolumn{1}{c:}{} & &&\\
& & & & & &&*&\multicolumn{1}{c:}{} & &&\\
& & & & & &&&\multicolumn{1}{c:}{} & &&\\
\cdashline{7-9}

&&& && & &&& 1&&\\
&&& && & &&& &\ddots&\\
&&& && & &&& &&1\\

\end{array}
\right],
\end{equation}
where the blocks $[*]$ are of size $n'\times n'$.
On the other hand, for $\beta\in\Phi^{\bar{N}}$, each matrix $ x_{\beta}(t)$ and $h_{\beta}(s)$
have the same block structure as the matrix (\ref{matrix-form}), but its $[*]$-block is replaced by the identity matrix 
and the identity block replaced by $[*]$-block.
Therefore, the subgroup $I$ contains the group
given by the a direct product
$$
I\supset I^{(N)}\times I^{(\bar{N})}. 
$$
We note that  $I^{(N)}$ and $I^{(\bar{N})}$ are isomorphic to the Iwahori subgroup of $\Sp_{2n'}(W)$, and $\Sp_{2(n-n')}(W)$, respectively.
To emphasize that the Iwahori subgroup consists of $2n' \times 2n'$ matrices, we occasionally denote it by $I^{(n')}$.

\section{Schwartz functions and Fourier transform}\label{Schwartz-function}
Let $V$ be a finite dimensional vector space over $F$, equipped with a Haar measure $dv$.
Denote by $\SS(V)$ the space of Schwartz functions on $V$, that is, locally constant functions with compact support. 
Let $V^*$ be the dual vector space of $V$.
A Fourier transform is a map from $\SS(V)$ to $\SS(V^*)$ defined by
$f\mapsto \hat{f}$, where
\begin{equation*}
  \hat{f}(v^*)=\int_{V}f(v)\psi(2\langle v^*,v\rangle)dv.
\end{equation*}
Here, $\langle\quad,\quad\rangle$ is the canonical pairing. As we can identify $V^*$ as $V$,
we regard $\hat{f}\colon V^*\rightarrow \mathbb{C}$ as a map from $V$ to $\mathbb{C}$:
\begin{equation*}
  \hat{f}(u)=\int_{V}f(v)\psi(2\braket{u,v})dv.
\end{equation*}

We normalize the Haar measure so that $\hat{\hat{f}}(-v)=f(v)$.

The vector space $V$ is algebraically isomorphic to $F^n$, so we often denote $\SS(V)$ by $\SS(F^n)$. We consider the action on functions on some lattices.
\begin{lemma}\label{Fourier-transform}
For any integers $\ell_i,k_i$ with $\ell_i\leq e+k_i$, define the lattice $L^k$ in $F^n$ generated by $\{\varpi^{k_1},\ldots,\varpi^{k_n}\}$.
Then, the following statements hold:
\begin{enumerate}[label={\arabic*.}]
    \item The Fourier transform  maps
    \begin{equation*}
        \SS(\varpi^{\ell_i}\oo/2\varpi^{k_i}\oo)\longrightarrow
        \SS(\varpi^{-k_i}\oo/2\varpi^{-\ell_i}\oo).
    \end{equation*}
    \item Take a characteristic function $\mathbf{1}_{x_i+2\varpi^{k_i}\oo}\in\SS(\varpi^{\ell_i}\oo/2\varpi^{k_i}\oo)$.
    Then, the characteristic function $\mathbf{1}_{x+2L^k}$ of $\SS(L^\ell/2L^k)$ is compatible with respect to the tensor product and Fourier transform; that is, one has $\hat{\mathbf{1}}_{x+2L^k}=\otimes_i\hat{\mathbf{1}}_{x_i+2\varpi^{k_i}\oo}$. 
\end{enumerate}
 Therefore, for the lattice $L^k=\prod_{i=1}^n(\varpi^{k_i}\oo)$, the Fourier transform maps
    \begin{equation*}
        \SS(L^\ell/2L^k)\cong
        \bigotimes_{i=1}^n\SS(\varpi^{\ell_i}\oo/2\varpi^{k_i}\oo)\longrightarrow
        \bigotimes_{i=1}^n\SS(\varpi^{-k_i}\oo/2\varpi^{-\ell_i}\oo)
        \cong\SS(L^{-k}/2L^{-\ell}).
    \end{equation*}
 \end{lemma}
 \begin{proof}
     For integers $\ell_i,k_i$, the space $\SS(\varpi^{\ell_i}\oo/2\varpi^{k_i}\oo)$ is generated by the vectors $\mathbf{1}_{\varpi^{\ell_i} z+2\varpi^{k_i}\oo}$ for $z\in\oo$. Then the Fourier transform of this vector is 
     \begin{align*}
        \hat{\textbf{1}}_{\varpi^{\ell_i} z+2\varpi^{k_i}\oo}
        (u)&=\int_F\textbf{1}_{\varpi^{\ell_i} z+2\varpi^{k_i}\oo}
        (v)\psi(2uv)dv\\
        &=\psi(2\varpi^{\ell_izu})\int_{\oo}\psi(4\varpi^{k_i}t)dt\\
        &=
        \begin{cases}
            \vol(2\varpi^{k_i}\oo)\psi(2\varpi^{\ell_i}zu)\quad&u\in\varpi^{-k_i}\oo;\\
            0&\textup{otherwise}.
        \end{cases}\\
     \end{align*}
     Since $\psi(2\varpi^{\ell_i} zu)$ is $2\varpi^{-\ell_i}\oo$-invariant, the first assertion is proved.
     
     We now show the second statement. 
     By the definition of product measure, one has
     $\vol(2L^k)=\prod_{i=1}^n\vol(2\varpi^{k_i}\oo)$. Hence, we obtain    
     \begin{align*}
         \hat{\textbf{1}}_{x+2L^k}(u)
         &=\int_{x+2L^k}\psi(2\braket{u,v})dv\\
         &=\vol(2L^k)\psi(2\braket{x,u})\textbf{1}_{L^{-k}}(u)\\
         &=\prod_{i=1}^n\vol(2\varpi^{k_i}\oo)\psi(2x_iu_i)\textbf{1}_{\varpi^{-k_i}\oo}(u_i)\\
         &=\prod_{i=1}^n(\hat{\textbf{1}}_{x+\varpi^{-k_i}\oo}(u_i)).
     \end{align*}
     The proof is complete.
 \end{proof}

Next, let $A$ be a finite set. We consider the space $\SS(A)^+$ (resp.\ $\SS(A)^-$) of even (resp.\ odd) functions of $\SS(A)$.
For a function $f\in\SS(A)$, define functions $f^+$ and $f^-$ by 
\begin{align*}
  f^+(x)&=\frac{f(x)+f(-x)}{2}\\
  f^-(x)&=\frac{f(x)-f(-x)}{2}.
\end{align*}
Hence we can decompose $f$ as
\begin{equation}
  f=f^++f^-.
\end{equation}
Let us define $\even\colon f\mapsto f^+$ and $\odd\colon f\mapsto f^-$. These are surjections onto $\SS(A)^+$ and $\SS(A)^-$, respectively.
The space $\SS(A)$ is written as
\begin{align*}
  \SS(A)&=\even(\SS(A))\oplus\odd(\SS(A))\\
  &=\SS(A)^+\oplus\SS(A)^-.
\end{align*}

\begin{lemma}\label{plus-quotient}
    Let $A$ be a finite set and let $B$ be a subset of $A$.
    Then, one has
    \begin{equation}
        \SS(A)/\SS(B)
        \cong\SS(A)^+/\SS(B)^+\oplus\SS(A)^-/\SS(B)^-.
    \end{equation}
\end{lemma}
\begin{proof}
We can identify $\SS(B)$ with a vector subspace of $\SS(A)$:
\begin{equation*}
f(x)=
\begin{cases}
  f(x) \quad&\text{if }x\in B\\
  0 &\text{if }x\notin B.
\end{cases}
\end{equation*}
This injection induces a canonical map $\pi\colon\SS(A)\rightarrow\SS(A)/\SS(B)$.
Since the image of $\SS(B)$ under $\even$ is contained in $\SS(B)$, the map 
\begin{equation}
    \begin{array}{r@{\,\,}c@{\,\,}c@{\,\,}c}
        \teven\colon & \SS(A)/\SS(B) &\longrightarrow   & \SS(A)/\SS(B) \\
                     & [f]           &\longmapsto      & [\even f]
    \end{array}
\end{equation}
is well-defined.  
Then, there exists a map $T\colon\SS(A)^+\rightarrow\teven\left(\SS(A)/\SS(B)\right)$ such that the following diagram commutes:

$$
\begin{tikzcd}
    \SS(A) \arrow[r,  "\pi"] \arrow[d,"\even"] &\SS(A)/\SS(B) \arrow[d,"\teven"]\\
    \SS(A)^+ \arrow[r, dashed, "T"] & \teven\left(\SS(A)/\SS(B)\right) \rlap{\ .}
\end{tikzcd}
$$
As $\pi$, $\even$ and $\teven$ are all surjective, the map $T$ is surjective.
Furthermore, since the kernel of $T$ agrees with $\SS(B)^+$, one has the isomorphism
\begin{equation*}
  \teven\left(\SS(A)/\SS(B)\right)\cong \SS(A)^+/\SS(B)^+.
\end{equation*}
Similarly, $\todd\left(\SS(A)/\SS(B)\right)$ is isomorphic to $\SS(A)^-/\SS(B)^-$.
On the other hand, we can decompose $[f]\in\SS(A)/\SS(B)$ as follows:
\begin{align*}
  [f]=[f^++f^-]
  =[\even f+\odd f]
  =[\even f]+[\odd f ]
  =\teven[f]+\todd[f].
\end{align*}
Thus, we have a decomposition
\begin{equation}\label{sch-decomposition}
\begin{aligned}
\SS(A)/\SS(B)&=\teven\left(\SS(A)/\SS(B)\right)\oplus\todd\left(\SS(A)/\SS(B)\right)\\
  &\cong\SS(A)^+/\SS(B)^+\oplus\SS(A)^-/\SS(B)^-.
\end{aligned}    
\end{equation}
\end{proof}

\section{The definition of the Weil representation}

\subsection{Metaplectic group and Weil representation}
Let $H(W)$ be the Heisenberg group associated with $W$;
namely $H(W)=W\times k$ as a set and multiplication is given by 
\begin{equation*}
(u,s)\cdot(v,t)=\left(u+v,s+t+\frac{1}{2}\Omega(u,v)\right).
\end{equation*}
Let $(\rho,S)=(\rho_\psi,S)$ be a representation of $H(W)$ with central character $\psi$.
The center of $H(W)$ is $\{0\}\times F \cong F$. 
So the central character $\rho$ can be regarded as a character of $F$. 
For $g\in\Sp(W)$, consider the representation $\rho^g$ twisted by $g$, that is,
$\rho^g(v,t)=\rho(gv,t)$. This also has the central character $\psi$. 
By the Stone-von Neumann theorem, $\rho^g$ is isomorphic to $\rho$ as a representation,
so there exists an intertwining map such that 
$\rho^g\circ T(g)=T(g)\circ\rho$. The operator $T$ is determined uniquely up to scalar, so this defines a projective representation $(T,S)$ of $\Sp(W)$.
Moreover, $(T,S)$ lifts uniquely to a linear representation $(\omega_\psi,S)$ of the metaplectic group $\Mp(W)$, which is a central extension of $\Sp(W)$:

$$
\begin{tikzcd}
  1 \arrow[r,""] & \{ \pm1 \} \arrow[r,""] &\Mp(W) \arrow[r,""]&\Sp(W)\arrow[r,""]&1.
\end{tikzcd}
$$
We call the representation $(\omega_\psi,S)$ of $\Mp(W)$ the Weil representation. 

We note some relations of the metaplectic group.
By Theorem 10 of \cite{steinberg},
a universal central extension $E$ of $\Sp(W)$ exists, and is written by only relations (R1) and (R2).
So we denote by $\hat{x}_\alpha(t)$ the generator of $E$ as (R1) and (R2). This theorem implies that
there exists a unique map $\pi\colon E\rightarrow\Mp(W)$ such that 
the following diagram commutes:

\begin{equation*}
    \begin{tikzcd}
    E \arrow[r,"\pi"]\arrow[ rd] & \Mp(W) \arrow[d,]\\
    & \Sp(W) \rlap{\, .}
    \end{tikzcd}
\end{equation*}
Hence, the generator $x_\alpha(t)\in\Sp(W)$ lifts uniquely to $\tilde{x}_\alpha(t)\in\Mp(W)$ by taking
$\tilde{x}_\alpha(t)=\pi(\hat{x}_\alpha(t))$. The relations
(R1) through (R4) in \Cref{chevalley_relation} hold in $E$ or $\Mp(W)$ because (R3) and (R4) are consequences (R1) and (R2), but (R5) does not. More precisely, one has
$$
\textup{(R5')}\quad h_\alpha(t)h_\alpha(u)=(t,u)h_\alpha(tu)
$$
for some constant $(t,u)\in\{\pm1\}$.

The maximal open compact subgroups of $\Mp(W)$ are full-inverse images of 
those of $\Sp(W)$. So these maximal open compact subgroups are given by
$\Kr$, the full inverse image of $K_r$. 
As seen in \cref{max-cpt-subgp-generator}, $K_r$ is generated by the elements $x(a)$, $h(b)$ and $\eta_r w$. So the subgroup $\Kr$ is generated by a lift of each element, $\chit{a}$, $\hht{b}$ and $\etat_r\wt$. Similar facts hold for the Iwahori subgroup $\tilde{I}$.

\subsection{Schr\"{o}dinger model}\label{Schrodinger-model}
Let $\SS(Y)$ be the space of Schwartz-Bruhat functions on $Y$ with the polarization $W=X+Y$, so that 
the Weil representation is realized on $\SS(Y)$ as the Schr\"{o}dinger model. For generating elements of $\Mp(W)$, its action is written 
explicitly as follows:
\begin{align*}
    [\chit{a}\varphi](y)&=\psi(\tp yay)\varphi(y),\\
    [\hht{b}\varphi](y)&=\beta_b|\det b|^{1/2}\varphi(\tp by),\\
    [\wt\varphi](y)&=\gamma_1\hat{\varphi}(y).
\end{align*}
Here, 
\begin{align*}
    \chit{a}\text{ is a lift of }
    &\begin{bmatrix}
    1&a\\
    0&1
\end{bmatrix}
\text{ for } a\in \mathrm{M}_{n\times n}(F),\\
\hht{b}\text{ is a lift of }
&\begin{bmatrix}
    b&0\\
    0&\tp b^{-1}\\
\end{bmatrix}
\text{ for }b\in\GL_n(F),\\
\wt \text{ is a lift of }
&\begin{bmatrix}
0&1\\
-1&0
\end{bmatrix}.
\end{align*}
As the constants $\beta_b$ and $\gamma_1$ do not affect our result, we omit 
the details. We note that these constants are $8$-roots of unity.
The maximal compact subgroup $\Kr$ contain an element $\etat_r\wt$.
The action of $\etat_r$ is a particular case of $\hht{b}$ with $b=\mathbf{m}_r(\varpi)^{-1}=\diag(\varpi^{-1},\ldots,\varpi^{-1},1\ldots,1)$. Hence the action of $\etat_r\wt$ is given by  $\etat_r\wt\varphi(y)=\gamma_1\beta_{\mathbf{m}_r(\varpi)^{-1}}\phihat(\addpi{y})$ with 
\begin{equation}\label{addpi-def}
    \addpi{y}=\vect{\varpi^{-1}y^{(1)}}{y^{(2)}}=\tp(\varpi^{-1}y_1,\ldots,\varpi^{-1}y_r,y_{r+1},\ldots,y_n).
\end{equation}


\begin{remark}\label{iwahori-subrep}
We recall that the Iwahori subgroup $\tilde{I}=\tilde{I}^{(n)}$ contains a smaller Iwahori subgroup $\tilde{I}^{(N)}$, where $N\subsetneq\{1,2,\ldots, n\}$. By the above presentation of actions, we see that the subgroup $\tilde{I}^{(N)}$ acts only on the components with indices in $N$.
\end{remark}

\section{Some properties of additive character}
We review some properties of additive character $\psi$. These facts play a pivotal role in the proofs that follow. Recall that $\psi$ is nontrivial, smooth, additive character with conductor $4\oo$.
Let $\delta\in\{0,1\}$.
\begin{proposition}\label{n=1mod2o}
For fixed $x,y\in F$ , if $\psi(tx^2)=\psi(ty^2)$ for all $t\in \varpi^{-\delta} \oo$, 
then the following conditions hold.

\begin{enumerate}[label={\textup{\arabic*.}}]
  \item The element $x$ is congruent to $y$ or $-y$ mod $2\varpi^{\delta}\oo$.
  \item If $x,y\notin \oo$, then $\val(x)=\val(y)$ and
  $x$ is congruent to $y$ or $-y$ mod $2\varpi^{\ell+\delta}\oo$ with $\ell=-\val(r)=-\val(s)$.
\end{enumerate}
\end{proposition}

\begin{proof}
    Since $\psi(t(x^2-y^2))=1$ for all $t\varpi^{-\delta}\in\oo$, one has
  \begin{equation}\label{factor}
    (x+y)(x-y)\in4\varpi^{\delta}\oo.
  \end{equation}
  This implies that
  \begin{equation}\label{val-ineq}
     2\max \{x+y,x-y\}\geq \val(x+y)+\val(x-y)\geq2e+\delta.
  \end{equation}
  Thus, the first assertion is proved.
  
  We write $x$ and $y$ as 
  \begin{equation*}
    x=\varpi^{\val(x)}u_1,\quad y=\varpi^{\val(y)}u_2,\quad  \quad u_1,u_2\in  \oo^\times.
  \end{equation*}
  We can assume that $\val(x)\geq \val(y)$. Hence (\ref{factor}) implies that 
  \begin{equation*}
    (u_1+\varpi^{\val(x)-\val(y)}u_2)(u_1-\varpi^{\val(x)-\val(y)}u_2)\in 4\varpi^{-2\val(y)+\delta}\oo.
  \end{equation*}
  Since the left-hand side does not lie in $\oo^\times$ due to $\val(y)<0$, we must have $\val(x)=\val(y)$. 
  Let $\ell =-\val(y)=-\val(y)$. Then, we obtain 
  \begin{equation*}
    (u_1+u_2)(u_1-u_2)\in 4\varpi^{2\ell+\delta}\oo.
  \end{equation*}
  We see that $u_1+u_2\in 2\oo$ and $u_1-u_2\in 2\oo$. If not,
\begin{equation*}
\left\{ \,
    \begin{aligned}
    & u_1\pm u_2 =2\varpi^{-k_1} v_1  \\
    & u_1\mp u_2 =4\varpi^{k_2} v_2
    \end{aligned}
\right.
\end{equation*}
for some $v_1,v_2\in \oo^{\times}, 1\leq k_1\leq e$ and $k_2\geq 1$.
Then one has 
$2u_1=2\varpi^{-k_1}v_1+4\varpi^{k_2}v_2$.
So the unit $u_1=\varpi^{-k_1}v_1+2\varpi^{k_2}v_2$ lies in $\varpi^{-k_1}\oo^{\times}$, which is a contradiction.
Moreover, we also see that $u_1+ u_2$ or $u_1- u_2$ lies in $2\oo^{\times}$.
If not, we can write
\begin{equation*}
\left\{ \,
    \begin{aligned}
    & u_1+ u_2 =2\varpi z_1  \\
    & u_1- u_2 =2\varpi z_2
    \end{aligned}
\right.
\end{equation*}
for some $z_1,z_2\in \oo$. So we get $u_1=\varpi(z_1+z_2)$, which is a contradiction.
Therefore, either $u_1+u_2$ or $u_1- u_2$ lies in $2\oo^{\times}$, 
which means that either $u_1+u_2$ or $u_1- u_2$ lies in $2\varpi^{2\ell+\delta}\oo$.
Hence $x+y$ or $x-y$ is in $2\varpi^{\ell+\delta}\oo$. 
\end{proof}
\begin{corollary}\label{n=1mod2pimo}
  Fix $x, y\in \varpi^{-m} \oo \smallsetminus \varpi^{-m+1} \oo$ with $m\geq1$.
  If $\psi(tx^2)=\psi(ty^2)$ for all $t\in \varpi^{-\delta}\oo$, 
  then $\val(x)=\val(y)$, and $x\equiv y$ or $x\equiv -y$ mod $2\varpi^{m+\delta}\oo$.
\end{corollary}
\begin{proof}
      By \Cref{n=1mod2o}, one has $\val(x)=\val(y)$,
  and $x\equiv y$ or $x\equiv -y$ mod $2\varpi^{-\val(x)+\delta}\oo$.
  Since $\val(x)<-m+1$ and $\val(y)<-m+1$, we know
  $x\equiv y$ or $x\equiv -y$ mod $2\varpi^{m+\delta}\oo$.
\end{proof}

\begin{remark}\label{ext-delta}
    The above two results can be expanded to the case of $\delta=-1$, which is important for representations of the Iwahori subgroup.
    Moreover, the first assertion of \cref{n=1mod2o} asserts that $x\equiv \pm y$ mod $2\oo$ even though $\delta=-1$. This will be described in \cref{iwahori-separate-bottom-pi}.
\end{remark}

\begin{lemma}\label{non-trivial-1}
    Let $\psi$ be a nontrivial, smooth, additive character of $F$ with conductor $4\oo$. For each integer $\ell\leq e-1$, there exists an element $x\in F$ such that
    $\val(x)=\ell$ and $\psi(x)\ne1$.
\end{lemma}

\begin{proof}
We prove this by induction. By the definition of conductor, there exists an element $y$ such that $\val(y)=2e-1$ and $\psi(y)\ne1$, so it satisfies when $\ell=2e-1$.
     Suppose that $\psi(y)=-1$ for all $y$ with $\val(y)=\ell-1$. By hypothesis, we can find an element $x$ with $\val(x)=\ell$ such that  $\psi(y)\ne1$.
     As $\val(x+y)=\ell-1$, our hypothesis implies that $\psi(x+y)=1$. But $\psi(x+y)=\psi(x)\psi(y)\ne1$, which is a contradiction.
\end{proof}

\begin{lemma}\label{non-trivial-minus1}
    Let $\psi$ be a additive character defined above. For each integer $\ell\leq-2$, there exists an element $x\in F$ such that
    $\val(x)=\ell$ and $\psi(x)\ne-1$.
\end{lemma}
\begin{proof}
    The proof is almost same to that of \cref{non-trivial-1}.
    We can take $x$ with $\val(x)=\ell\leq2e-1$ such that $\psi(x)\ne1$ by \cref{non-trivial-1}.
    If $\psi(y)=-1$ for all $y$ with $\val(y)=\ell-1$, then one has
    $-1=\psi(x+y)=\psi(x)\psi(y)\ne1\cdot-1$, which is a contradiction.
\end{proof}

\section{Restricting to maximal compact subgroups}
In this section, we give the irreducible decomposition of the Weil representation restricted to maximal open compact subgroups.
Recall that the maximal open compact subgroups $\Kr$ of $\Mp(W)$ are full-inverse images of those of $\Sp(W)$. 
\subsection{The filtration}
In this subsection, 
we will construct a filtration of the representation space $\SS(Y)$ and show the irreducibility of its bottom space.
 We see that the subgroup $K_r$ of $\Sp(W)$ is the subgroup which preserves the lattices
 \begin{align*}
   \mathcal{L}_r&=\mathrm{span}_\oo\{e_1,\dots, e_n,\varpi f_1,\dots,\varpi f_r,f_{r+1},\dots,f_n\};\\
   \mathcal{L}_r^*&=\mathrm{span}_\oo\{\varpi^{-1} e_1,\dots, \varpi^{-1} e_r,e_{r+1},\dots,e_n,f_1,\dots,f_n\}
 \end{align*}
with respect to the natural action. For the above lattices $\mathcal{L}_r$ and $\mathcal{L}_r^*$, define 
$$
L_r=\mathcal{L}_r\cap Y\cong\varpi\oo^r \times\oo^{n-r}\mathrm{\quad and\quad} L=\mathcal{L}_r^*\cap Y\cong\oo^n.
$$
Then, we have a chain of lattices
\begin{equation}\label{chainlattice}
\cdots\subset\varpi L_r\subset \varpi L\subset L_r\subset L
\subset \varpi^{-1}L_r\subset \varpi^{-1}L\subset \cdots.
\end{equation}
From this, we obtain the chain of lattice quotients:
\begin{equation}\label{chainlatticequotient}
    L/2L_r\subset\varpi^{-1}L_r/2\varpi L\subset \varpi^{-1} L/2\varpi L_r\subset \varpi^{-2}L_r/2\varpi^2 L\cdots.
\end{equation}
We define function spaces on these lattices:
\begin{align*}
    \Slow{r}{m}&=\SS(\varpi^{-m}L/2\varpi^m L_r),\\
    \Shigh{r}{m}&=\SS(\varpi^{-m}L_r/2\varpi^m L).
\end{align*}
Then the sequence (\ref{chainlatticequotient}) lifts to a filtration  of function space $\SS(Y)$:
\begin{equation*}
    \Slow{r}{0}\subset\Shigh{r}{1}\subset\Slow{r}{1}\subset\Shigh{r}{2}\subset\cdots.
\end{equation*}
Note that the space $\Slow{r}{m}$ (resp.\ $\Shigh{r}{m}$) makes sense for $m\geq0$ (resp.\ $m\geq1$). Since one has $\SS(Y)=\lim_{m\rightarrow\infty}\Slow{r}{m}=\lim_{m\rightarrow\infty}\Shigh{r}{m}$, we obtain
\begin{equation}\label{K-decomposition}
    \begin{aligned}
        \SS(Y)&=\Slow{r}{0}\oplus\bigoplus_{m\geq1}\left(\Shigh{r}{m}/\Slow{r}{m-1}\oplus\Slow{r}{m}/\Shigh{r}{m}\right)\\
    &=\Slow{r}{0}^+\oplus\Slow{r}{0}^-\oplus\bigoplus_{\substack{m\geq1\\\eps\in\{+,-\}}}
     \left(\Shigh{r}{m}^\eps/\Slow{r}{m-1}^\eps\oplus\Slow{r}{m}^\eps/\Shigh{r}{m}^\eps\right).
    \end{aligned}
\end{equation}
Here, the second equality follows from \cref{plus-quotient}. Our goal in this section is to show that the above decomposition is the irreducible decomposition of the Weil representation restricted to maximal open compact subgroup $\Kr$:
\begin{theorem}\label{K-main-theorem}
    The irreducible decomposition of $\omega_\psi|_{\Kr}$ is given by 
    \begin{equation*}
        \SS(Y)\cong\Slow{r}{0}^+\oplus\Slow{r}{0}^-
        \oplus\bigoplus_{m\geq1}(\E_{r,m}^+\oplus\E_{r,m}^-\oplus
        \F_{r,m}^+\oplus\F_{r,m}^-),
    \end{equation*}
    where: 
    \begin{align*}
        \E_{r,m}=\Shigh{r}{m}/\Slow{r}{m-1},\quad
        \F_{r,m}=\Slow{r}{m}/\Shigh{r}{m}.
    \end{align*}
    Each dimension of the irreducible component is 
    \begin{align*}
    \dim \Slow{r}{0}^\pm&=\frac{1}{2}q^{en}(q^r\pm1);\\
        \dim\E_{r,m}^\pm&=\frac{1}{2}q^{(2m+e-2)n+r}
        (q^{2(n-r)}-1);\\
        \dim\F_{r,m}^\pm&=\frac{1}{2}q^{(2m+e)n-r}(q^{2r}-1).
    \end{align*}
    Moreover, the representation $\omega_\psi|_{\Kr}$ is multiplicity-free. 
\end{theorem}

As a preliminary step, we show that the representation is well-defined on the quotient space:
\begin{proposition}[{\cite{savin2015}}]
    The maximal compact subgroup $\Kr$ preserves the filtration
    \begin{equation*}
        \Slow{r}{0}\subset\Shigh{r}{1}\subset\Slow{r}{1}\subset\Shigh{r}{2}\subset\Slow{r}{2}\subset\cdots.
    \end{equation*}
\end{proposition}
\begin{proof}
This is already proved in \cite{savin2015}. However, to fit our notation, we offer a different proof here. We recall that $\Kr$ is generated by $\chit{a}$, $\hht{b}$ and $\etat_r\wt$, which is stated in \cref{max-cpt-subgp-generator}.
    We first show a proof for $\chit{a}$ with $a\in\M$.
    By the definition of the action, we see that the support does not change, so we only have to check that $\psi( \tp yay)$ is $2\varpi^mL$ (or $2\varpi^mL_r$)-invariant.
    Take $z$ from $L$ or $L_r$. As the conductor of $\psi$ is $4\oo$, we have
    \begin{equation*}
        \psi\left( \tp(y+2\varpi^mz)a(y+2\varpi^mz)\right)=
        \psi( \tp yay+4\varpi^m\tp yaz+4\varpi^{2m} \tp zaz)=\psi( \tp yay).
    \end{equation*}
The element $\hht{b}$ acts as $\hht{b}\varphi(y)=\varphi(\tp b^{-1}y)$. As a map $y\mapsto \tp b^{-1}y$ preserves the chain of lattices (\ref{chainlatticequotient}), $\hht{b}$ preserves the function spaces.
We see that the Weyl element $\etat_r\wt$ preserves the filtration  
by \cref{Fourier-transform}.
\end{proof}
Second, we show the irreducibility of representation on the bottom space:
\begin{proposition}[{\cite{savin2015}}]\label{irred-bottom-space}
For $0\leq r\leq n$, the spaces $\Slow{r}{0}^+$ and $\Slow{r}{0}^-$ are irreducible.
\end{proposition}
Although this fact is proved in \cite{savin2015},
we give our own proof here for the sake of completeness.
We note that the space $\Slow{0}{0}^-$ is zero when $r=0$. The element $\chit{a}$ plays an important role in establishing irreducibility. For the proof, we first show the following fact:
    \begin{lemma}\label{n>1mod2o}
    Let $x,y\in L=\oo^n$. If $\psi(\tp xax)=\psi(\tp yay)$ for all $a\in\M$, then 
    $x\equiv y$ or $x\equiv-y$ mod $2L_r=2\varpi\oo^r\times2\oo^{n-r}$.
    \end{lemma}
    \begin{proof}
        For $1\leq i\leq n$, we define
        \begin{align*}
            \delta(i)=
          \begin{cases}
              1\quad1\leq i\leq r;\\
              0\quad r+1\leq i\leq n.
          \end{cases}
        \end{align*}
    Then, this lemma is equivalent to the following statement:
    one has $x_i\equiv y_i$ or $y_i\equiv-y_i$ mod $2\varpi^{\delta(i)}\oo$ for all $i$, and the sign is independent of the index.
    
    Taking $a=tE _{ii}$, we have 
  $\psi(tx_i^2)=\psi(ty_i^2)$ for all $t\in \varpi^{-\delta(i)}\oo.$
  Hence by the first assertion of \Cref{n=1mod2o}, we obtain 
  $x_i\equiv y_i $ or $x_i\equiv -y_i$ mod $2\varpi^{\delta(i)}\oo$
  for all $i$. So it remains to verify that the sign is independent of the index.
  If not, there exist indices $i,j$ such that 
  $x_i \equiv y_i$ and $x_j \equiv  -y_j$ mod $2\varpi^{\delta(i)}\oo.$
Thus, we can write $y_i$ and $y_j$ as 
\begin{equation}\label{index}
  y_i=x_i+2\varpi^{\delta(i)}z_1, \quad \quad y_j=-x_j+2\varpi^{\delta(j)}z_2
\end{equation}
for some $z_1, z_2\in \oo$.
If $x_i\equiv -x_i\modd 2\varpi^{\delta(i)}\oo$, then these signs agree, so we may assume that
$\val(x_i)\leq\delta(i)-1$ and $\val(x_j)\leq \delta(j)-1$.
On the other hand, by taking $a=t(E_{ij}+E_{ji})$ for $t\in\varpi^{-\delta(i)\delta(j)}\oo$, we have $\psi(2tx_ix_j)=\psi(2ty_iy_j)$
for all $t\in\varpi^{-\delta(i)\delta(j)}\oo$. As $t$ is arbitrary, this implies that $x_ix_j-y_iy_j\in 2\varpi^{\delta(i)\delta(j)}\oo$.
By applying equation (\ref{index}) to this, we compute the element 
$x_ix_j-y_iy_j$ as follows:
\begin{align*}
  x_ix_j-y_iy_j&= x_ix_j-(x_i+2\varpi^{\delta(i)}z_1)(-x_j+2\varpi^{\delta(j)}z_2)\\
               &= 2(x_ix_j+(\varpi^{\delta(i)}x_jz_1-\varpi^{\delta(j)}x_iz_2)-2\varpi^{\delta(i)+\delta(j)}z_1z_2)\\
               &=2((x_i+\varpi^{\delta(j)}z_1)(x_j-\varpi^{\delta(i)}z_2)-\varpi^{\delta(i)+\delta(j)}z_1z_2)\in 2\varpi^{\delta(i)\delta(j)}\oo.
\end{align*}
Since $\delta(i)\delta(j)\leq\delta(i)+\delta(j)$, one has 
\begin{equation}\label{hello}
  (x_i+\varpi^{\delta(i)}z_1)(x_j-\varpi^{\delta(j)}z_2)\in \varpi^{\delta(i)\delta(j)}\oo.
\end{equation}
Since we assume that  $\val(x_i)\leq\delta(i)-1$ and $\val(x_j)\leq\delta(j)-1$, we have
$\val((x_i+\varpi^{\delta(j)}z_1)(x_j-\varpi^{\delta(i)}z_2))\leq\delta(i)+\delta(j)-2<\delta(i)\delta(j)$, which contradicts (\ref{hello}).
    \end{proof}

\begin{corollary}\label{K-separate-bottom}
    Let $\{\pphi_x\}_x$ (resp.\ $\{\mphi_x\}_x$) be a basis
    of $\Slow{r}{0}^+$ (resp.\ $\Slow{r}{m}^-$), and $\mathcal{U}$ be a nonzero irreducible component of $\Slow{r}{0}^+$ (resp.\ $\Slow{r}{0}^-$). Then, $\{\pphi_x\}_x$ (resp.\ $\{\mphi_x\}_x$) lies in $\mathcal{U}$ for some $x$.
\end{corollary}
\begin{proof}
We show the two cases at the same time. Each basis of $\Slow{r}{0}^+$ or $\Slow{r}{0}^-$ is given by
\begin{equation}\label{hello-1101}
\begin{alignedat}{2}
  \pphi_x
  &=
  \left\{
  \begin{aligned}
    &\textbf{1}_{x+2L_r}\\
    &\textbf{1}_{x+2L_r}+\textbf{1}_{-x+2L_r}
  \end{aligned}
  \right.
  &\quad&
  \begin{aligned}
    &x\in L_r;\\
    &x\notin L_r;
  \end{aligned}
  \\
  \mphi_x
  &=\textbf{1}_{x+2L_r}-\textbf{1}_{-x+2L_r}
  &\quad&x\notin L_r.
\end{alignedat}
\end{equation}
We note that the vector $\mphi_x$ is zero if $x\in L_r$.
The irreducible component $\mathcal{U}$ contains a nonzero vector $f$ with support $S$.
Let $\varphi_x$ denote the corresponding basis element ($\pphi_x$ or $\mphi_x$). Then, the vector $f$ can be written as a finite liner combination
\begin{equation*}
  f = \sum_{\substack{z\in S/\sim}} f(z) \varphi_z,
\end{equation*}
where $z\sim z'$ is defined by $\varphi_z=\varphi_{z'}$. 
Fix $x\in S$ and take an arbitrarily $y\in S\smallsetminus\{x\}$.
By \Cref{n>1mod2o},
we can find an element $a\in\Sym_n(\varpi^{-\delta}\oo)$ such that 
$\psi (\tp xax)\ne\psi (\tp yay)$.
Then by computing
\begin{equation}\label{important-calc}
\chit{a}f-\psi(\tp yay)f=
    \sum_{\substack{z\in S/\sim}} \psi(\tp zaz)f(z) \varphi_z
    -\sum_{\substack{z\in S/\sim}} \psi(\tp yay)f(z) \varphi_z,
\end{equation}
we can eliminate $\varphi_{y}$ without removing $\varphi_{x}$.
Repeating this procedure, we have $\varphi_{x}\in\mathcal{U}$.
We will use the calculation (\ref{important-calc}) repeatedly in many of the proofs that follow.
\end{proof}

\begin{proof}[Proof of \cref{irred-bottom-space}]
Each basis of $\Slow{r}{0}^+$ or $\Slow{r}{0}^-$ is given by (\ref{hello-1101}).
Let $\mathcal{U}$ be an irreducible component of $\Slow{r}{0}^+$ or $\Slow{r}{0}^-$ and $f\in\mathcal{U}$.
We will prove that $\mathcal{U}$ coincides with $\Slow{r}{0}^\pm$.
By \cref{K-separate-bottom}, we can assume that $\varphi_x\in\mathcal{U}$ for some $x\in S=\supp f$.

If $\varphi_x=\pphi_x$, we have 
\begin{align*}
\etat_r\wt\pphi_x&=
\begin{cases}
    \displaystyle c\sum_{\substack{t\in L/\sim}}\psi(2  \braket{x,\addpi{t}})\textbf{1}_{t+2L_r} &x\in L_r\\
    \displaystyle c\sum_{\substack{t\in L/\sim}}\psi(2 \braket{x,\addpi{t}})\textbf{1}_{t+2L_r}+\displaystyle c\sum_{\substack{t\in L/\sim}}\psi(-2 \braket{x,\addpi{t}})\textbf{1}_{t+2L_r} &x\notin L_r
\end{cases}\\
&=
\begin{cases}
\displaystyle c\sum_{\substack{t\in L}/\sim}\psi(2 \braket{x,\addpi{t}})\pphi_x &x\in L_r\\
\displaystyle c\sum_{\substack{t\in L/\sim}}[\psi(2 \braket{x,\addpi{t}})+\psi(-2 \braket{x,\addpi{t}})]\pphi_x&x\notin L_r\quad ,
\end{cases}
\end{align*}
where $c=\gamma_1\beta_{\mathbf{m}_r(\varpi)^{-1}}$. From this, 
$\etat_r\wt\pphi_x$ is not zero on $x=0$,
so we obtain $\varphi_0\in \mathcal{U}$ by the action of $\chit{a}$. 
By computing 
\begin{equation*}
  \etat_r\wt\pphi_0 = c\sum_{\substack{t\in L/\sim}}\pphi_t,
\end{equation*}
we see that $\pphi_t\in \mathcal{U}$ for all $t\in L/2L_r$ by \Cref{n>1mod2o}.

Now we consider the case of $\varphi=\mphi$. If $r=n$, then $\Kr$ includes the element
$\hht{b}$ with $b\in\GL_n(\oo)$. Since $L\smallsetminus L_n=\oo^n\smallsetminus\varpi\oo^n$, the element $\hht{b}$ acts transitively on the basis
$\{\mphi_t:t\in L\text{ and }t\notin L_n\}$. Thus, we have $\mphi_t\in \mathcal{U}$ for all $t\in L\smallsetminus L_n$.
If $r<n$, we compute the action of $\etat_r\wt\mphi_x$:
\begin{align*}
    \etat_r\wt\mphi_x(t)&=c\left(\psi(2 \braket{x,\addpi{t}})-\psi(-2 \braket{x,\addpi{t}})\right)\\
    &=c\psi(-2 \braket{x,\addpi{t}})(\psi(4 \braket{x,\addpi{t}})-1)\\
    &=c\psi(-2 \braket{x,\addpi{t}})[\psi(4 \braket{x^{(1)},t^{(1)}})\psi(4\varpi^{-1} \braket{x^{(2)},t^{(2)}})-1].
\end{align*}
It follows from \cref{non-trivial-1} that we can find an element $t^{(1)}$ such that $\psi(4\varpi^{-1} \braket{x^{(1)},t^{(1)}})-1\ne0$. So we fix $t^{(1)}$ and let
$t_0=\vect{t^{(1)}}{0^{(2)}}$ with $0^{(2)}=\tp(\underbrace{0,\ldots,0}_{n-r})$.
We see that $\mphi_{t_0}\in\mathcal{U}$ because of 
$\etat_r\wt\mphi_{x}(t_0)\ne 0$. We have 
\begin{equation*}
    \etat_r\wt\mphi_{t_0}\vect{x^{(1)}}{z^{(2)}}
    =c\psi(-2\varpi^{-1} \braket{x^{(1)},t^{(1)}})[\psi(4\varpi^{-1} \braket{x^{(1)},t^{(1)}})-1]\ne0
\end{equation*}
for all $z^{(2)}\in\oo^{n-r}$.
Thus $\mphi_z$ lies in $\mathcal{U}$ for all $z=\{x^{(1)}\}\times\oo^{n-r}$.
For each $z^{(2)}\in\oo^{n-r}$, we have $\mphi_z\in\mathcal{U}$ for 
$z\in\ (\oo^r\smallsetminus\varpi\oo^r)\times \{z^{(2)}\}$ because
$\GL_r(\oo)\cdot z^{(1)}=\oo^r\smallsetminus\varpi\oo^r$.
This implies $\mphi_z\in\mathcal{U}$ for all $z\in L\smallsetminus L_r=(\oo^r\smallsetminus\varpi\oo^r)\times\oo^{n-r}$.
\end{proof}

\subsection{The proof of first case}
In this subsection, we give a proof of irreducibility of spaces $\F_{r,m}=\Slow{r}{m}^\pm/\Shigh{r}{m}^\pm$ with $m\geq1$. If $r=0$, the quotient space is trivial, so we suppose that $1\leq r\leq n$.
Set a basis of the space $\Slow{r}{m}^+$ or $\Slow{r}{m}^-$ as follows:
\begin{alignat*}{2}
  \pphi_x&=\left\{
  \begin{aligned}
      &\textbf{1}_{x+2\varpi^{m}L_r}\\
    &\textbf{1}_{x+2\varpi^{m}L_r}+\textbf{1}_{-x+2\varpi^{m}L_r} & 
  \end{aligned}\right.
  &\begin{aligned}
      x\in \varpi^{m}L_r;\\
      x\notin \varpi^{m}L_r;
  \end{aligned}
  \\
  \mphi_x&=\textbf{1}_{x+2\varpi^{m}L_r}-\textbf{1}_{-x+2\varpi^{m}L_r}
  &x\notin \varpi^{m}L_r.
\end{alignat*}
We note that $\pphi_x$ is equal to $\pphi_{-x}$, and 
$\mphi_x$ is the zero function when $x\in \varpi^{m}L_r$.
Since we frequently deal with both even and odd cases, we sometimes write $\varphi_x=\pmphi_x$.
The quotient spaces $\Slow{r}{m}^\pm/\Shigh{r}{m}^\pm$ are the function space
generated the above bases and equipped the relation 
\begin{equation*}
    f= 0\iff 
    \begin{aligned}
        \supp(f)\subset\varpi^{-m}L_r\textup{ and }
        f\textup{ is }2\varpi^mL\textup{-invariant}.
    \end{aligned}
\end{equation*}
As a preliminary step, we show the following lemma:

\begin{lemma}\label{m-mod-first-case}
    Let $x,y\in \varpi^{-m}L\smallsetminus\varpi^{-m}L_r$ and suppose that $\psi(\tp xax)=\psi(\tp yay)$ for all $a\in\M$.
    Then $x\equiv y$ or $x\equiv-y$ mod $2\varpi^mL_r=2\varpi^{m+1}\oo^r\times2\varpi^{m}\oo^{n-r}$.
\end{lemma}
\begin{proof}
By \Cref{n>1mod2o}, 
one has 
\begin{equation*}\label{mod2o}
  x_j\equiv y_j \text{ or } x_j\equiv -y_j\modd 2\varpi^{\delta(j)}\oo
\end{equation*}
for all $j$. 
Note that the sign does not depend on the index. As we can find an index $i$
such that $1\leq i\leq r$ and $x_i\notin \varpi^{-m+1} \oo$, we fix the element $x_i$ and
take $j\ne i$. 
By \Cref{n=1mod2pimo}, we also see that $x_i\equiv y_i$ mod $2\varpi^{m+\delta(i)}\oo$. So we write $y_i$ and $y_j$ as
\begin{equation}
\left\{ \,
    \begin{aligned}\label{rewrite}
    & y_i =\pm x_i+2\varpi^{m+\delta(i)}z_1;  \\
    & y_j =\pm x_j+2\varpi^{\delta(j)}z_2,
    \end{aligned}
\right.
\end{equation}
where $z_1,z_2$ are elements of $\oo$. We will show that $z_2$ is in $\varpi ^{m}\oo$. 
Taking $a=t(E_{ij}+E_{ji})$ for $t\in \varpi^{-\delta(i)\delta(j)}\oo$, we have
\begin{equation*}
  \psi(2tx_ix_j)=\psi(2ty_iy_j) \textup{\quad for all\quad}  t\in \varpi^{-\delta(i)\delta(j)}\oo.
\end{equation*}
This implies
$x_ix_j-y_iy_j\in 2\varpi^{\delta(i)\delta(j)}\oo$.
By applying (\ref{rewrite}) to this, the left-hand side is computed as follows:
\begin{align*}
  x_ix_j-y_iy_j&= x_ix_j-(\pm x_i+2\varpi^{\delta(i)}z_1)(\pm x_j+2\varpi^{m+\delta(j)}z_2)\\
               &= -2(\pm \varpi^{\delta(j)}x_iz_2\pm\varpi^{m+\delta(i)}x_jz_1+2\varpi^{m+\delta(i)+\delta(j)}z_1z_2)\in 2\varpi^{\delta(i)\delta(j)}\oo.
\end{align*}
This shows that
$\varpi^{\delta(j)}x_iz_2+\varpi^{m+\delta(i)}x_jz_1\in \varpi^{\delta(i)\delta(j)}\oo$.
In this case, $\delta(i)=1$ by hypothesis, and hence we obtain $\varpi^{\delta(j)}x_iz_2+\varpi^{m+1}x_jz_1\in \varpi^{\delta(j)}\oo.$
One has $\val(x_j)\geq-m$, so $\varpi^{m+1}x_jz_1\in\varpi^{\delta(j)}\oo$.
Therefore, $x_iz_2\in\oo$, so we have $z_2\in \varpi^{m}\oo$. 
\end{proof}

\begin{remark}\label{K-x(a)-separate-first}
    Let $\mathcal{U}$ be a nonzero irreducible component of  
    $\F_{r,m}^\pm$ and $f\in\mathcal{U}$. 
    As in the proof of \cref{K-separate-bottom}, we see that \cref{m-mod-first-case} provides the following:
    if $x\in (\varpi^{-m}L\smallsetminus\varpi^{-m}L_r )\cap \supp f$, the basis vector
    $\varphi_x$ is in $\mathcal{U}$ by the action of $\chit{a}$.
\end{remark}

\begin{lemma}\label{h-transitive-first}
    Let $\mathcal{U}$ be an irreducible component of $\F_{r,m}^+=\Slow{r}{m}^+/\Shigh{r}{m}^+$(resp.\ $\F_{r,m}^-=\Slow{r}{m}^-/\Shigh{r}{m}^-$). If a basis vector $\pphi_x$ (resp.\ $\mphi_x$) lies in $\mathcal{U}$ for some $x\in\varpi^{-m}L\smallsetminus\varpi^{-m}L_r$, then $\pphi_x$ (resp.\ $\mphi_x$) lies in $\mathcal{U}$ for all $x\in\varpi^{-m}L\smallsetminus\varpi^{-m}L_r$.
\end{lemma}
\begin{proof}
This proof is almost same as that 
of the case $\varphi_x=\varphi_x^-$ in $\cref{irred-bottom-space}$. 
    We write $\varphi_x=\pmphi_x$ and prove both cases simultaneously.  As we have taken $x$ from $\varpi^{-m}L\smallsetminus\varpi^{-m}L_r$, we can assume that $\pmphi_x=\textbf{1}_x\pm\textbf{1}_{-x}$.
    By the action of $\etat_r\wt$, one has
    \begin{align*}
    \etat_r\wt\pmphi_x(t)&=c\left(\psi(2 \braket{x,\addpi{t}})\pm \psi(-2 \braket{x,\addpi{t}})\right)\\
    &=c\psi(- 2\braket{x,\addpi{t}})(\psi(4 \braket{x,\addpi{t}})\pm1)\\
    &=c\psi(-2 \braket{x,\addpi{t}})[\psi(4\varpi^{-1} \braket{x^{(1)},t^{(1)}})\psi(4 \braket{x^{(2)},t^{(2)}})\pm1],
\end{align*}
where $c=\gamma_1\beta_{\mathbf{m}_r(\varpi)^{-1}}$.
    Here, we recall $\addpi{t}=\tp(\varpi^{-1}t_1,\ldots,\varpi^{-1}t_r,t_{r+1},\ldots,t_n)$, see 
    \cref{Schrodinger-model}.
    Since the valuation of $x^{(1)}$ is equal to $-m\leq-1$, by \cref{non-trivial-minus1} (resp.\ \cref{non-trivial-1}), we can find an element $t^{(1)}\in\varpi^{-m}\oo^r\smallsetminus\varpi^{-m+1}\oo^r$ such that
    $\psi(4\varpi^{-1} \braket{x^{(1)},t^{(1)}})$ is not equal to $1$ (resp.\ $-1$). We fix the element $t_0=\vect{t^{(1)}}{0^{(2)}}$.
    Then $\etat_r\wt\varphi_x(t_0)$ takes a nonzero value, hence $\varphi_{t_0}\in\mathcal{U}$.
   Also, we have
    \begin{align*}
        \etat_r\wt\pmphi_{t_0}\vect{x^{(1)}}{z^{(2)}}
        &=c[\psi(2\varpi^{-1}\braket{t^{(1)},x^{(1)}})\pm\psi(-2\varpi^{-1}\braket{t^{(1)},x^{(1)}})]\ne0,
    \end{align*}
    so we know $\varphi_{z}\in\mathcal{U}$ for all
    $z\in\{x^{(1)}\}\times\varpi^{-m}\oo^{n-r}$.
    For each $z^{(2)}\in\varpi^{-m}\oo^{n-r}$, we see that 
    $\varphi_{z}\in\mathcal{U}$ for all
    $z\in(\varpi^{-m}\oo^r\smallsetminus\varpi^{-m+1}\oo^r)\times\{z^{(2)}\}$. This implies that $\varphi_z\in\mathcal{U}$ for all 
    $z\in\varpi^{-m}L\smallsetminus\varpi^{-m}L_r$.
\end{proof}

\begin{proposition}\label{separable-first}
    Let $\mathcal{U}$ be an irreducible component of $\F_{r,m}^\pm=\Slow{r}{m}^\pm/\Shigh{r}{m}^\pm$.
    Then $\pmphi_x\in\mathcal{U}$ for some $x\in \supp f$.
\end{proposition}
\begin{proof}

Note that this cannot be proved using only the action of 
$\chit{a}$, as was done in the proof of \cref{K-separate-bottom}.
We show the two cases simultaneously. Let $S=\supp f$. In the case where $S\not\subset \varpi^{-m}L_r$, take $x\in S\smallsetminus \varpi^{-m}L_r$.
Then, 
we see that $\varphi_x=\pmphi_x$ lies in $\mathcal{U}$ as we seen in \cref{K-x(a)-separate-first}.

In the case where $S\subset \varpi^{-m}L_r$, take $x\in S$ arbitrarily, and we will show $\varphi_x\in\mathcal{U}$.
We write the function $f$ as an element of vector space:
\begin{equation*}
    f = \sum_{\substack{y\in S/\sim }}f(y)\varphi_y,
\end{equation*}
  where $y\sim y'$ means $\varphi_y=\varphi_{y'}$. We compute the action of Weyl element:
  \begin{align*}
  \etat_r\wt f(t)&=\etat_r\wt\sum_{y\in S/\sim}f(y)\varphi_x(t)
  =c\left[\sum_{y\in S/\sim}f(y)\phihat_x\right](\addpi{t}),\\
   \etat_r\wt f&= c\sum_{\substack{t\in \varpi^{-m}L/\sim\\ }}\left[\sum_{\substack{y\in S/\sim\\ }}f(y)\phihat_y(\addpi{t})\right]\varphi_t,
  \end{align*}
  where $\addpi{t}=\tp(\varpi^{-1}t_1,\ldots,\varpi^{-1}t_r,t_{r+1},\ldots, t_n )$ as in (\ref{addpi-def}) and $c=\gamma_1\beta_{\mathbf{m}_r(\varpi)^{-1}}$.
  \begin{lemma}\label{zero-condition}
      For $y\in S\subset\varpi^{-m}L_r$, the function $t\mapsto\phihat_y(\addpi{t})$ is $2\varpi^mL$ -invariant.
  \end{lemma}
  \begin{proof}
      Since $\phihat_y^\pm(\addpi{t})=c\left(\psi(2 \braket{y,\addpi{t}})\pm \psi(-2\braket{y,\addpi{t}}_r)\right)$, it suffice to check that $t\mapsto\psi(\pm2\braket{y,\addpi{t}})$ is $2\varpi^mL$-invariant.
      Define 
      \begin{align*}
      \delta(i)=
          \begin{cases}
              1\quad1\leq i\leq r;\\
              0\quad r+1\leq i\leq n.
          \end{cases}
      \end{align*}
    Then, the $i$-th component of $\addpi{t}$ is $\varpi^{-\delta(i)}t_i$. For $z\in\oo^n$, one has
\begin{align*}
    \psi(\pm2 \braket{y,\addpi{t+2\varpi^mz}})
    &=\prod_{i=1}^n\psi(\pm2y_i(\varpi^{-\delta(i)}t_i+2\varpi^{m-\delta(i)} z_i))\\
    &=\prod_{i=1}^n\psi(\pm2\varpi^{-\delta(i)}y_it_i)
    =\psi(\pm2\braket{y,\addpi{t}}).
\end{align*}

  \end{proof}
  
  The lemma implies that the vector $\phihat_y(\addpi{t})$ lies in $\Slow{r}{m}^\pm$ for all $y\in S$ and $t\in \varpi^{-m}L_r$. Equivalently, since these vectors vanish in the space $\Shigh{r}{m}^\pm/\Slow{r}{m}^\pm
  $, we obtain
  \begin{equation}\label{hello0724-1}
  \begin{aligned}
  \etat_r\wt f&= c\sum_{\substack{t\in \varpi^{-m}L/\sim\\ }}\left[\sum_{\substack{y\in S/\sim\\ }}f(y)\phihat_y(\addpi{t})\right]\varphi_t\\
  &= c\sum_{\substack{t\in (\varpi^{-m}L\smallsetminus\varpi^{-m}L_r)/\sim\\ }}\left[\sum_{\substack{y\in S/\sim\\ }}f(y)\phihat_y(\addpi{t})\right]\varphi_t
  +c\sum_{\substack{t\in \varpi^{-m}L_r/\sim\\ }}\left[\sum_{\substack{y\in S/\sim\\ }}f(y)\phihat_y(\addpi{t})\right]\varphi_t\\
  &=c\sum_{\substack{t\in (\varpi^{-m}L\smallsetminus\varpi^{-m}L_r)/\sim\\ }}\left[\sum_{\substack{y\in S/\sim\\ }}f(y)\phihat_y(\addpi{t})\right]\varphi_t,
  \end{aligned}
  \end{equation}
  where $c=\gamma_1\beta_{\mathbf{m}_r(\varpi)^{-1}}$. By \cref{K-x(a)-separate-first}, we have $\varphi_{t_0}\in\mathcal{U}$ for some $t_0\in \varpi^{-m}L\smallsetminus\varpi^{-m}L_r$. It follows from \cref{h-transitive-first} that $\varphi_{t}\in\mathcal{U}$ for all $t\in \varpi^{-m}L\smallsetminus\varpi^{-m}L_r$.
  Therefore, for $x\in S$, we have 
  \begin{align*}
      \etat_r\wt\varphi_x&= 
      c\sum_{\substack{t\in \varpi^{-m}L/\sim\\ }}\left[\phihat_x(\addpi{t})\right]\varphi_t\\
          &=c\sum_{\substack{t\in (\varpi^{-m}L\smallsetminus\varpi^{-m}L_r)/\sim\\ }}\left[\phihat_x(\addpi{t})\right]\varphi_t\\
          &\in\bigoplus_{\substack{t\in (\varpi^{-m}L\smallsetminus\varpi^{-m}L_r)/\sim\\ }}\mathbb{C}\varphi_t\subset \mathcal{U}.
  \end{align*}
Since $(\etat_r\wt)^8=1$, the vector $\varphi_x$ lies in $\mathcal{U}$.
\end{proof}

\begin{proposition}\label{irreducible-first}
    For $0<r\leq n$, the two spaces $\F_{r,m}^\pm=\Slow{r}{m}^\pm/\Shigh{r}{m}^\pm$ 
    are irreducible.
\end{proposition}
Before proving this, we first show the following lemma:
\begin{lemma}\label{nonzero-first}
    For $x\in\varpi^{-m}L$, suppose that the function
    $t\mapsto\phihat_x(\addpi{t})$ on $\varpi^{-m}L_r$ is $2\varpi^mL$ - invariant. Then, the element $x$ also lies in $\varpi^{-m}L_r$.
\end{lemma}
\begin{proof}
We first prove the case where $\varphi_x=\pphi_x$. Suppose that $t\mapsto\hat{\varphi}^+_x(\addpi{t})$ is invariant under $2\varpi^m L$. If $x\in\varpi^mL$, then $x$ trivially lies in $\varpi^{-m}L_r$, so we assume that $x\notin\varpi^mL$, that is, $\pphi_x=\one_{x+2\varpi^mL_r}+\one_{-x+2\varpi^mL_r}$.
In this case, the hypothesis is equivalent to 
\begin{equation}\label{hello1030}
    \psi(2 \braket{x,\addpi{t}+2\varpi^m\addpi{z}})+\psi(-2 \braket{x,\addpi{t}+2\varpi^m\addpi{z}})
    =\psi(2\braket{x,\addpi{t}})+\psi(-2 \braket{x,\addpi{t}})
\end{equation}
for all $t\in \varpi^{-m}L_r$ and $z\in\oo^n$. Applying $t=0$ to (\ref{hello1030}), we have $\psi(4\varpi^m \braket{x,\addpi{z}})+\psi(-4\varpi^m \braket{x,\addpi{z}})=1+1=2$, which implies 
\begin{equation}\label{hello0725-2}
   \psi(4\varpi^m \braket{x,\addpi{z}})=\psi(-4\varpi^m \braket{x,\addpi{z}})=1
\end{equation}
for all $z\in L=\oo^n$. As $\addpi{z}=(\varpi^{-\delta(i)}z_i)_i$, we have $x_i\in\varpi^{-m+\delta(i)}\oo$ for all $i$. This implies that $x\in\varpi^{-m}L_r$.

Second, we show the case of $\varphi_x=\mphi_x$. Suppose that 
$t\mapsto\phihat^-_x(\addpi{t})$ is $2\varpi^mL$-invariant. Then we have 
\begin{equation}\label{hello3}
    \psi(2 \braket{x,\addpi{t}+2\varpi^m\addpi{z}})-\psi(-2\braket{x,\addpi{t}+2\varpi^m\addpi{z}})
    =\psi(2 \braket{x,\addpi{t}})-\psi(-2 \braket{x,\addpi{t}})
\end{equation}
for all $t\in\varpi^{-m}L_r$ and $z\in\oo^n$. Applying $t=0$,
we obtain 
\begin{equation*}
    \psi(4\varpi^m\braket{x,\addpi{z}})-\psi(-4\varpi^m \braket{x,\addpi{z}})=0.
\end{equation*}
This implies that $\psi(4\varpi^m \braket{x,\addpi{z}})$ equals to $1$ or $-1$ for all $z\in\oo^n$. We assume that $x\notin\varpi^{-m}L_r$. Then there exists an index $1\leq j\leq r$ with $\val(x_j)=-m$. So we can find $z'= \tp(0,..,0,z_j,0,\cdots,0 )$ satisfying $\psi(4\varpi^m \braket{x,\addpi{z'}})=\psi(4\varpi^{m-1}x_jz_j)=-1$ because the conductor of $\psi$ is $4\oo$.
Applying to $z=z'$ to (\ref{hello3}), we have
\begin{equation*}
    -\psi(2 \braket{x,\addpi{t}})+\psi(-2\braket{x,\addpi{t}})
    =\psi(2\braket{x,\addpi{t}})-\psi(-2\braket{x,\addpi{t}}).
\end{equation*}
This is equivalent to $\psi(4\braket{x,\addpi{t}})=1$ for all $t$. In particular, one has $\psi(4\varpi^{-1}x_jt_j)=1$ for all $t_j\in\varpi^{-m}\oo$, which is a contradiction.
\end{proof}

\begin{proof}[Proof of \cref{irreducible-first}]
    Let $\mathcal{U}$ be an irreducible component of $\F_{r,m}^\pm=\Slow{r}{m}^\pm/\Shigh{r}{m}^\pm$. By \cref{separable-first}, we can assume that a basis vector $\varphi_x=\pmphi_x$ lies in $\mathcal{U}$.
    We first show the case of $x\in \varpi^{-m}L_r$. Since this quotient space has the relation
    \begin{equation*}
        \sum_{\substack{t\in\varpi^{-m}L_r}/\sim}\varphi_t=0,
    \end{equation*}
    one has
    \begin{equation*}
        \varphi_x= -\sum_{\substack{t\in(\varpi^{-m}L_r\smallsetminus\{x,-x\}})/\sim}\varphi_t\in \mathcal{U}.
    \end{equation*}
    Recall that $t\sim t'$ means $\varphi_t=\varphi_{t'}$.
    By using \Cref{separable-first} repeatedly, it follows that 
   $\varphi_t\in \mathcal{U}$ for all $t\in \varpi^{-m}L_r$. It remains to show that $\varphi_t$ lies in $\mathcal{U}$ for all $t\in \varpi^{-m}L\smallsetminus\varpi^{-m}L_r$. By the action of $\etat_r\wt$, we have
   \begin{equation*}
       \etat_r\wt\varphi_x=
       c\sum_{\substack{t\in \varpi^{-m}L/\sim\\ }}\left[\phihat_x(\addpi{t})\right]\varphi_t
          =c\sum_{\substack{t\in (\varpi^{-m}L\smallsetminus\varpi^{-m}L_r)/\sim\\ }}\left[\phihat_x(\addpi{t})\right]\varphi_t\in \mathcal{U},
   \end{equation*}
   where $c=\gamma_1\beta_{\mathbf{m}_r(\varpi)^{-1}}$ and the second equality follows from \Cref{zero-condition}.
   This sum is not zero because the vector $\varphi_x$ is non-zero and $\etat_r\wt$ is an automorphism.
So by \Cref{separable-first}, we have $\varphi_t\in\mathcal{U}$ for some $t\in\varpi^{-m}L\smallsetminus\varpi^{-m}L_r$.
By the action of $\tilde{h}(b)$, we obtain $\varphi_t\in\mathcal{U}$ for all $t\in\varpi^{-m}L\smallsetminus\varpi^{-m}L_r$ from \cref{h-transitive-first}.

Next, we prove the case where $x\notin \varpi^{-m}L_r$.
In this case, one has $\varphi_t\in\mathcal{U}$ for all $t\in\varpi^{-m}L\smallsetminus\varpi^{-m}L_r$ by \cref{h-transitive-first}.
We again use the action of the Weyl element:
\begin{equation}\label{weyl-sum}
\begin{aligned}
       \etat_r\wt\varphi_x&=
       c\sum_{\substack{t\in \varpi^{-m}L/\sim\\ }}\left[\phihat_x(\addpi{t})\right]\varphi_t\\
          &=c\sum_{\substack{t\in (\varpi^{-m}L\smallsetminus\varpi^{-m}L_r)/\sim\\ }}\left[\phihat_x(\addpi{t})\right]\varphi_t
          +  c\sum_{\substack{t\in\varpi^{-m}L_r/\sim\\ }}\left[\phihat_x(\addpi{t})\right]\varphi_t.
   \end{aligned}
   \end{equation}
   Since we are supposing $x\notin\varpi^{-m}L_r$, it follows from \cref{nonzero-first} that the second sum of (\ref{weyl-sum}) is not zero mod $\Shigh{r}{m}$. So by repeating \Cref{separable-first}, we can find a vector 
$\varphi_{t_0}\in \mathcal{U}$ for some $t_0\in \varpi^{-m}L_r$.
Hence, we reduce to the former case ($x\in\varpi^{-m}L_r$). Therefore the proof is complete.
\end{proof}

\subsection{The proof of second case}
In this subsection, we give a proof of the irreducibility of spaces $\E_{r,m}^\pm=\Shigh{r}{m}^\pm/\Slow{r}{m-1}^\pm$. If $r=n$, the quotient space is trivial, so we suppose that $0\leq r< n$.
\begin{lemma}\label{m-mod-second-case}
  Fix $x, y\in \varpi^{-m}L_r \smallsetminus\varpi^{-m+1} L$ with $m\geq1$.
  If $\psi (\tp xax)=\psi (\tp yay)$ for all $a\in \M$, 
  then $x\equiv y$ or $x\equiv -y$ mod $2\varpi^{m}L=2\varpi^m\oo^n$. 
\end{lemma}
\begin{proof}
    This proof is analogous to that of \cref{m-mod-first-case}.
    Since we can find an index $r<i\leq n$ such that $\val(x_i)=-m$, we fix $i$ for the rest of the proof.
    Following a similar argument as in \cref{m-mod-first-case}, we have an analogue of equation (\ref{rewrite}), namely
    \begin{equation}
\left\{ \,
    \begin{aligned}
    & y_i =\pm x_i+2\varpi^{m+\delta(i)}z_1;  \\
    & y_j =\pm x_j+2\varpi^{\delta(j)}z_2
    \end{aligned}
\right.
\end{equation}
for all $j\ne i$. Note that the sign does not depend on the index.
We show that $z_2$ is in $\varpi^{m-\delta(j)}\oo$.
Applying $a=t(E_{ij}+E_{ji})$, we have
$\varpi^{\delta(j)}x_iz_2+\varpi^{m+\delta(i)}x_jz_1\in\varpi^{\delta(i)\delta(j)}\oo$.
In this case, $\delta(i)=0$ by hypothesis, and hence we obtain 
$\varpi^{\delta(j)}x_iz_2+\varpi^{m}x_jz_1\in\oo$.
As $\val(x_j)\geq-m$, we obtain $\varpi^{\delta(j)}x_iz_2\in\oo$. We conclude that $z_2\in\varpi^{m-\delta(j)}\oo$.
\end{proof}

\begin{remark}\label{K-x(a)-separate-second}
We have the fact analogous to \cref{K-x(a)-separate-first} from \cref{m-mod-second-case}. 
    Let $\mathcal{U}$ be a nonzero irreducible component of  
    $\E_{r,m}^\pm$ and $f\in\mathcal{U}$. 
    If $x\in (\varpi^{-m}L_r\smallsetminus\varpi^{-m+1}L )\cap \supp f$, one has 
    $\varphi_x\in\mathcal{U}$ via the action of $\chit{a}$.
\end{remark}

\begin{lemma}\label{h-transitive-second}
    Let $\mathcal{U}$ be an irreducible component of $\E_{r,m}^\pm=\Shigh{r}{m}^\pm/\Slow{r}{m-1}^\pm$. If a basis vector $\varphi_x$ lies in $\mathcal{U}$ for some $x\in\varpi^{-m}L_r\smallsetminus\varpi^{-m+1}L$, then the basis vector $\varphi_x$ is in $\mathcal{U}$ for all $x\in\varpi^{-m}L_r\smallsetminus\varpi^{-m+1}L$.
\end{lemma}
\begin{proof}
    We write the basis vector $\varphi_x=\pmphi_x$. This is an analogue of \cref{h-transitive-first}, so we omit some details.
    For $x=\vect{x^{(1)}}{x^{(2)}}$, one has
    \begin{equation}
        \etat_r\wt\pmphi_x(t)=
        c\psi(-2\braket{x,\addpi{t}})[\psi(4\varpi^{-1} \braket{x^{(1)},t^{(1)}})\psi(4 \braket{x^{(2)},t^{(2)}})\pm1].
    \end{equation}
    Since $\val(x^{(2)})=-m\leq-1$ by hypothesis, 
    \cref{non-trivial-minus1} (resp.\ \cref{non-trivial-1}) implies that we can find an element $t^{(2)}\in\varpi^{-m}\oo^{n-r}\smallsetminus\varpi^{-m+1}\oo^{n-r}$ such that
    $\psi(4\varpi^{-1} \braket{x^{(2)},t^{(2)}})$ is not equal to $1$ (resp.\ $-1$). Fixing the element $t^{(2)}$, we see that $\etat_r\wt \varphi_x(t_0)\ne0$ for $t_0=\vect{0^{(1)}}{t^{(2)}}$ .
    Hence, the vector $\varphi_{t_0}$ lies in $\mathcal{U}$ because of \cref{K-x(a)-separate-second}. Moreover, we have 
    \begin{equation*}
        \etat_r\wt\varphi_{t_0}
        \vect{z^{(1)}}{x^{(2)}}
        =c\left(\psi(2\varpi^{-1} \braket{t^{(2)},x^{(2)}})\pm
        \psi(-2\varpi^{-1} \braket{t^{(2)},x^{(2)}})\right)\ne0
    \end{equation*}
    for all $z^{(1)}$, so the vector $\varphi_{z}$ is in $\mathcal{U}$ for all 
    $z\in(\varpi^{-m+1}\oo^{n-r})\times\{x^{(2)}\}$.
    Since one has $\GL_{n-r}(\oo)\cdot x^{(2)}=\varpi^{-m}\oo^{n-r}\smallsetminus\varpi^{-m+1}\oo^{n-r}$, the action of $\hht{b}$ implies that $\varphi_z\in\mathcal{U}$ for all $z\in \varpi^{-m}L_r\smallsetminus\varpi^{-m+1}L$.
\end{proof}
\begin{proposition}\label{separable-second}
    Let $\mathcal{U}$ be a nonzero irreducible component of $\E_{r,m}^\pm=\Shigh{r}{m}^\pm/\Slow{r}{m-1}^\pm$ and $f\in\mathcal{U}$. Then $\pmphi_x$ lies in $\mathcal{U}$ for some $x\in\supp f$.
\end{proposition}
\begin{proof}
    Let $S=\supp f$ and set the basis vector $\varphi_x=\pmphi_x$. In the case where $S\not\subset\varpi^{-m+1}L$, \cref{K-x(a)-separate-second} implies that a vector $\varphi_x$ lies in $\mathcal{U}$ by the actions of $\chit{a}$.

    In the case where $S\subset\varpi^{-m+1}L$, take $x\in S$ arbitrarily, and we will show that the vector $\varphi_x$ lies in $\mathcal{U}$. 
    We compute the action of $\etat_r\wt$:
    \begin{equation*}
        \etat_r\wt f=c\sum_{t\in\varpi^{-m}L_r/\sim}
        \left[
        \sum_{y\in S/\sim}f(y)\phihat_x(\addpi{t})
        \right]\varphi_t.
    \end{equation*}
Here, $\addpi{t}=(\varpi^{-1}t_1,\ldots,\varpi^{-1}t_r,t_{r+1},\ldots,t_n)$ and $c=\gamma_1\beta_{\mathbf{m}_r(\varpi)^{-1}}$.
Recall that $y\sim y'$ is defined by $\varphi_z=\varphi_{z'}$.
For $y\in S$, we see that the function $t\mapsto\phihat_x(\addpi{t})$ is $2\varpi^mL_r$-invariant.
Hence we obtain the same equation as (\ref{hello0724-1}):
\begin{equation*}
    \etat_r\wt f=c
    \sum_{t\in(\varpi^{-m}L_r\smallsetminus\varpi^{-m+1}L)/\sim}
    \left[
    \sum_{y\in S/\sim}f(y)\phihat_x(\addpi{t})
    \right]\varphi_t.
\end{equation*}
By \cref{m-mod-second-case}, we have $\phihat_{t_0}\in\mathcal{U}$ for some $t_0\in\varpi^{-m}L_r\smallsetminus\varpi^{-m+1}L$.
It follows from \cref{h-transitive-second} that $\varphi_x\in\mathcal{U}$ for all $t\in\varpi^{-m}L_r\smallsetminus\varpi^{-m+1}L$. Therefore, for $x\in S$, we have
\begin{equation*}
    \etat_r\wt\varphi_x=c\sum_{t\in\varpi^{-m}L_r/\sim}\phihat_x
    (\addpi{t})\varphi_t
    =c\sum_{t\in(\varpi^{-m}L_r\smallsetminus\varpi^{-m+1}L)/\sim}\phihat_x
    (\addpi{t})\varphi_t
    \quad\in\mathcal{U}.
\end{equation*}
Since $(\etat_r\wt)^8=1$, the basis vector $\varphi_x$ lies in $\mathcal{U}$.
\end{proof}

\begin{proposition}\label{irred-second}
    For $0\leq r<n$, two spaces $\E_{r,m}^\pm=\Shigh{r}{m}^\pm/\Slow{r}{m-1}^\pm$ are irreducible.
\end{proposition}

Before proving this, we first show the following lemma:
    \begin{lemma}\label{nonzero-second}
        For $x\in\varpi^{-m}L_r$, suppose that the function  $t\mapsto \phihat_x(\addpi{t})$ on $\varpi^{-m+1}L$ is $2\varpi^{m-1}L_r$-invariant. Then the element $x$ also lies in $\varpi^{-m+1}L$.
    \end{lemma}
    
    \begin{proof}
        This proof is the same to that of \cref{nonzero-first}, so we follow the proof.
        If $\varphi=\pphi$, 
        one has $\phihat_x\left(\addpi{t+2\varpi^{m-1}z}\right)=\phihat_x(\addpi{t})$ for $z\in L_r$. Hence,
        we have a result analogous to (\ref{hello0725-2}), which is 
        \begin{equation}
        \psi(4\varpi^{m-1}\braket{x,z})=\psi(-4\varpi^{m-1}\braket{x,z})=1
\end{equation}
for all $z\in\oo^n$. For each index $i$, we apply a value to $z$ so that $z_i=1$ and $z_j=0$ for $j\ne i$, and we have $x_i\in\varpi^{-m+1}\oo$. This implies that $x\in\varpi^{-m+1}L$.

If $\varphi=\mphi$, one has 
\begin{equation}\label{hello0725-3}
\begin{aligned}
    \psi\left(2\braket{x,\addpi{t+2\varpi^{m-1}z}}\right)
    -\psi\left(-2\braket{x,\addpi{t+2\varpi^{m-1}z}}\right)
    =\psi(2\braket{x,\addpi{t}})-\psi(-2\braket{x,\addpi{t}})
\end{aligned}
\end{equation}
for all $z$. Applying $t=0$, we have
$$
    \psi\left(4\varpi^{m-1} \braket{x,\addpi{z}}\right)-\psi\left(-4\varpi^{m-1} \braket{x,\addpi{z}}\right)=0
$$
and hence $\psi(4\varpi^{-m+1}\braket{x,z})$ equals $1$ or $-1$ for each $z$.
Suppose by contradiction that $x\notin\varpi^{-m+1}L$. Then there exists an index $r+1\leq j\leq n$ such that $\val(x_j)=-m$. Therefore, there exists 
an element $z'=\tp(0,\ldots,0,z_j,0,\ldots,0)$ satisfying 
$\psi(4\varpi^{m-1}\braket{x,z'})=\psi(4\varpi^{m-1}x_jz_j)=-1$ because the conductor of $\psi$ is $4\oo$.
Applying $z=z'$ to (\ref{hello0725-3}), we have 
\begin{equation*}
    -\psi(2\braket{x,\addpi{t}})+\psi(-2\braket{x,\addpi{t}})
    =  \psi(2\braket{x,\addpi{t}})
    -\psi(-2\braket{x,\addpi{t}})
\end{equation*}
for all $t$, which is equivalent to $\psi(4\braket{x,\addpi{t}})=1$.
In particular, one has $\psi(4\varpi ^{-1}x_jt_j)=1$ for all $t_j\in\varpi^{-m}\oo^\times$, which is a contradiction.
    \end{proof}

\begin{proof}[Proof of \cref{irred-second}]
Let $\mathcal{U}$ be an irreducible component of $\Shigh{r}{m}^+/\Slow{r}{m-1}^+$ or $\Shigh{r}{m}^-/\Slow{r}{m-1}^-$.
By \cref{separable-second}, we can assume that a vector $\varphi_x=\pmphi_x$ is in $\mathcal{U}$.

If $x\in\varpi^{-m+1}L$, the proof is completely analogous to the first case of \cref{irreducible-first}.
We use \cref{h-transitive-second} and \cref{separable-second}
instead of \cref{h-transitive-first} and \cref{separable-first}.

We show the case where $x\notin\varpi^{-m+1}L$.
In this case, one has $\varphi_t\in\mathcal{U}$ for all $t\in\varpi^{-m}L_r\smallsetminus\varpi^{-m+1}L$ by \cref{h-transitive-second}.
We compute the action of $\etat_r\wt$:
\begin{align}\label{hello0725-1}
    \etat_r\wt\varphi_x
    =c\sum_{t\in(\varpi^{-m}L_r\smallsetminus\varpi^{-m+1}L)/\sim}[\phihat_t(\addpi{t})]\varphi_t
    +c\sum_{t\in\varpi^{-m+1}L/\sim}[\phihat_t(\addpi{t})]\varphi_t.
\end{align}
    Here, $c=\gamma_1\beta_{\mathbf{m}_r(\varpi)^{-1}}$ and $t\sim t'$ means that $\varphi_t=\varphi_{t'}$. As $x$ is not in $\varpi^{-m+1}L$, \cref{nonzero-second} shows that the vector 
    \begin{equation*}
        \sum_{t\in\varpi^{-m+1}L/\sim}\phihat_t(\addpi{t})\varphi_t
    \end{equation*}
    is non-zero mod $\Slow{r}{m-1}$. Thus \cref{separable-second} implies that we can find a vector $\varphi_{t_0}\in\mathcal{U}$ for some $t_0\in\varpi^{-m+1}L$.
    Then, one has
\begin{equation*}
        \varphi_{t_0}= -\sum_{\substack{t\in(\varpi^{-m+1}L\smallsetminus\{t_0,-t_0\})/\sim}}\varphi_t,
\end{equation*}
so we have $\varphi_t\in \mathcal{U}$ for all $t\in \varpi^{-m+1}L$.
Therefore, the vector $\varphi_t$ lies in $\mathcal{U}$ for all $t\in\varpi^{-m}L_r$.
The proof is complete.
\end{proof}
By the above propositions, we have the irreducible decomposition of $\omega_\psi|_{\Kr}$:
\begin{proof}[Proof of \cref{K-main-theorem}]
    By \cref{irred-bottom-space,irreducible-first,irred-second}, each component in the decomposition (\ref{K-decomposition}) is irreducible.
    
    We now verify that $\omega_\psi|_{\Kr}$ is multiplicity-free.
    Each dimension of the irreducible component is 
    \begin{align*}
    \dim \Slow{r}{0}^\pm&=\frac{1}{2}q^{en}(q^r\pm1);\\
        \dim\E_{r,m}^\pm&=\frac{1}{2}q^{(2m+e-2)n+r}
        (q^{2(n-r)}-1);\\
        \dim\F_{r,m}^\pm&=\frac{1}{2}q^{(2m+e)n-r}(q^{2r}-1).
    \end{align*}
We will see that no two of these dimensions coincide.
Suppose that there is a pair of equal dimensions. By analyzing their prime factors, we see that the terms such as $q^r\pm1$,
$q^{2r}-1$ and $q^{2(n-r)}-1$ must much.

First, it is trivial that $q^r+1\ne q^r-1$. If $q^r+1=q^{2r}-1$, we have 
$q^r=2$, so $q=2$ and $r=1$. Then, matching the power of $q$, we obtain $en=(2m+e)n-r$, which is equivalent to 
$2m-1=0$, which is a contradiction since $m\in\mathbb{Z}$.
If $q^r+1=q^{2(n-r)}-1$, we have $q=2$ and $r=1$. Then, we get 
$2(m-1)n=-1$ which is a contradiction.
If $q^r-1=q^{2r}-1$, then we have $r=0$. In this case, 
we have $\dim \Slow{0}{0}^-=\dim\F_{0,m}^-=0$. 
In the case where $q^r-1=q^{2(n-r)}-1$, we have $3r=2n$. Then equality of the powers of $q$ gives $2mn-2n+r=0$, which is equivalent to 
$r(3m-2)=0$. Since $m$ is an integer, this forces $r=0$ and $n=0$.
If $q^{2r}-1=q^{2(n-r)}-1$, we have $2r=n$. 
Matching the powers of $q$ yields $n=r$. Therefore, we obtain 
$r=n=0$, a contradiction.

The remaining possibility is the cases of $\E_{r,m}^+\cong\E_{r,m}^-$ and $\F_{r,m}^+\cong\F_{r,m}^-$, whose dimensions coincide. Hence, we show that these are not equivalent by the character of representation such as $\chi_{\E_{r,m}}(g)=\Tr (g)$.
If $\E_{r,m}^+\cong\E_{r,m}^-$, then their characters must coincide.
Let $\{\xi_i^+\}_i$ (resp.\ $\{\xi_i^-\}_i$) be a basis of the space for $\E_{r,m}^+$ (resp.\ $\E_{r,m}^-$). Since these basis vectors are even (resp.\ odd) functions, one has $\hht{-1}\cdot\xi_i^\pm=\pm\beta_{-1}\xi_i^\pm$ and hence 
\begin{align*}
    \chi_{\E_{r,m}^+}(\hht{-1})&=\sum_i\beta_{-1}=\beta_{-1}\dim\E_{r,m}^+,\\
    \chi_{\E_{r,m}^-}(\hht{-1})&=\sum_i-\beta_{-1}=-\beta_{-1}\dim\E_{r,m}^-.
\end{align*}
Therefore, we have $\chi_{\E_{r,m}^+}\ne\chi_{\E_{r,m}^-}$, which is a contradiction. We can similarly show $\F_{r,m}^+\not\cong\F_{r,m}^-$. The proof is complete.
\end{proof}

\section{Restricting to the Iwahori subgroup}
In this section, we provide an irreducible decomposition of the Weil representation restricted to the Iwahori subgroup $\tilde{I}$. 
\subsection{The filtration}
Recall that the Iwahori subgroup is the intersection of all maximal compact subgroups $\Kr$. Since each $\Kr$ preserves the filtration
\begin{equation*}
    \Slow{r}{0}\subset\Shigh{r}{1}\subset\Slow{r}{1}\subset\Shigh{r}{2}\subset\cdots,
\end{equation*}
the Iwahori subgroup preserves all these subspaces:

\begin{equation}\label{iwahori-filtration}
\begin{array}{ccccccccccc}
\omega_\psi|_{\Kzero}:&&\Slow{0}{0} &\subset& \Shigh{0}{1} &=& \Slow{0}{1} &\subset&\cdots&&\\
&&\cap && \cup && \cap &&\\
\omega_\psi|_{\tilde{K}_1}:&&\Slow{1}{0} &\subset& \Shigh{1}{1} &\subset& \Slow{1}{1}&\subset&\cdots&&\\
&&\cap && \cup && \cap &&\\
&&\vdots && \vdots && \vdots &&\vdots\\
&&\cap && \cup && \cap &&\cup\\
\omega_\psi|_{\tilde{K}_{n-1}}:&&\Slow{n-1}{0} &\subset& \Shigh{n-1}{1} &\subset& \Slow{n-1}{1}&\subset&\Shigh{n-1}{2}&\subset&\cdots\\
&&\cap && \cup && \cap &&\cup\\
\omega_\psi|_{\Kn}:&&\Slow{n}{0} &=& \Shigh{n}{1} &\subset& \Slow{n}{1}&=&\Shigh{n}{2}&\subset&\cdots\rlap{\quad.}\\
\end{array}
\end{equation}
There are vertical inclusion relations such as $\Slow{r}{m}\subset\Slow{r+1}{m}$ and $\Shigh{r}{m}\supset\Shigh{r-1}{m}$. Hence, we obtain a finer filtration as a representation of $\tilde{I}$.
To make the structure of these subspaces clearer through the vertical inclusions, we write $\lat{i}{j}=\varpi^{j}\oo/2\varpi^{i}\oo$ for integers $i,j$. 
Then, we rewrite the spaces $\Slow{r}{m}$ and $\Shigh{r}{m}$ as follows:
\begin{alignat*}{2}
    \Slow{r}{m}&=\suni{m}^{\otimes r}\otimes\tuni{m}^{\otimes n-r},&
    \qquad\tuni{m}&=\SS(\Lt),\\
    \Shigh{r}{m}&=\suni{m-1}^{\otimes r}\otimes\tuni{m}^{\otimes n-r},&
    \suni{m}&=\SS(\Ls).
\end{alignat*}
We recall that the Iwahori subgroup $\tilde{I}$ is generated by 
$\chit{a}$ for $a\in\Sym_n(\oo)$, $\hht{b}$ for $b\in(\oo^\times)^n$, and $\wt\chit{\varpi a}\wt$ for $a\in\Sym_n(\oo)$.
We have previously used the action of $\chit{a}$ to show that a vector lies in a certain irreducible component $\mathcal{U}$ in the case of $\Kr$. We state here an analogous result for the element $\chit{a}\in\tilde{I}$, which is the case $r=0$ of \cref{n>1mod2o} and \cref{K-separate-bottom}:
\begin{corollary}\label{iwahori-separate-bottom}
    Let $x,y\in F$. If $\psi(\tp xax)=\psi(\tp yay)$ for all $a\in\Sym_n(\oo)$, then $x\equiv y$ or $x\equiv-y$ mod $2\oo^{n}$. Therefore, for a nonzero irreducible component $\mathcal{U}$ of $\tuni{0}^{\otimes n}$ and $f\in\mathcal{U}$, one has $\one_{x+2\oo^n}\in\mathcal{U}$ for some $x\in\supp f$.
\end{corollary}
Moreover, we sometimes use the action of $\wt\chit{\varpi a}\wt$ to show that the vectors $\hat{\textbf{1}}_{x+2\oo^n}$ can be separated. For this purpose, we prove the following lemma:
\begin{lemma}\label{iwahori-separate-bottom-pi}
    Let $x,y\in \oo/2\oo$. If $\psi(\tp x\varpi ax)=\psi(\tp y\varpi ay)$ for all $a\in\Sym_n(\oo)$, then $x\equiv y$ mod $2\oo^n$. Moreover, let $\mathcal{U}$ be a nonzero irreducible component
    of $\tuni{0}^{\otimes n}$ and $f\in\mathcal{U}$.
    Then, $\hat{\one}_{x+2\oo^n}\in\mathcal{U}$ for some $x\in\oo/2\oo$.
\end{lemma}
\begin{proof}
    Taking $a=tE_{ii}$, we have $\psi(\varpi tx_i^2)=\psi(\varpi ty_i^2)$ for all $t\in\oo$.
    In the proof of \cref{n=1mod2o}, we extend $\delta$ to $\delta=-1$. Then, the inequality (\ref{val-ineq}) becomes
    \begin{equation*}
        2\max \{x+y,x-y\}\geq2e-1.
    \end{equation*}
    Hence, we have $\val(x+y)\geq1$ or $\val(x-y)\geq1$. Furthermore, since $x$ and $y$ are taken from $\oo/2\oo$ and $y\equiv-y$ mod $2\oo$, we have $x_i\equiv y_i$ for all $i$, yielding $x\equiv y$ mod $2\oo^n$.

    We show the second part. The vector $f\in\mathcal{U}$ is written as the finite linear combination of $\{\hat{1}_z\}_{z\in\oo/2\oo}$:
    \begin{equation*}
        f=\sum_{z\in \supp \hat{f}}\hat{f}(z)\hat{\one}_z.
    \end{equation*}
    Here, $\one_z\coloneq\one_{z+2\oo^n}$.
    Fix $x\in\supp f$ and take $y\in \supp f\smallsetminus\{x\}$. \cref{iwahori-separate-bottom-pi} implies that we can find $a\in\Sym_n(\oo)$ so that  $\psi(\tp x\varpi ax)\ne\psi(\tp y\varpi ay)$. By computing
    \begin{align*}
        \gamma_1^{-2}\wt\chit{\varpi a}\wt f-\psi(\tp y\varpi ay)f
        =\sum_z\psi(\tp z\varpi az)\hat{f}(z)\hat{\one}_z
        -\sum_z\psi(\tp y\varpi ay)\hat{f}(z)\hat{\one}_z,
    \end{align*}
    We can eliminate $\hat{\one}_y$ without removing $\hat{\one}_x$. 
\end{proof}

\begin{proposition}\label{iwahori-bottom}
    The space 
    $\Slow{0}{0}=\tuni{0}^{\otimes n}$
    is irreducible as a representation of $\tilde{I}$.
\end{proposition}

\begin{proof}
This proposition is stated in \cite[ Section 2.2, p.1108]{takeda2018}, where it is claimed to have been proven in \cite{savin2015}. However, to the best of our knowledge, no proof is given in \cite{savin2015}. It is our understanding that the proof for the case $n=1$ was originally outlined in an email from Takeda to Wood \cite{mail}; we modify it here to present our own proof.
Let $\mathcal{U}$ be an irreducible component of $\tuni{0}^{\otimes n}$. By \cref{iwahori-separate-bottom}, we can assume that the vector 
$\textbf{1}_x=\textbf{1}_{x+2\oo^n}$ lies in $\mathcal{U}$ for some $x\in\oo$.
Taking $a\in\Sym_n(\oo)$ so that its $(1,1)$-entry is $\varpi$ and all the other entries are 0, one has 
\begin{align*}
\wt\chit{a}\wt\textbf{1}_x&=\wt\chit{a}\cdot \gamma_1\sum_{t\in\oo^n/2\oo^n}
    \psi(2\braket{x,t})\textbf{1}_t\\
    &=\gamma_1^2\sum_{t\in\oo^n/2\oo^n}
    \psi(\varpi t_1^2)\psi(2\braket{x,t})\hat{\textbf{1}}_t\\
    &=\gamma_1^2\sum_{t\in\oo^n/2\oo^n}(\psi(\varpi t_1^2)-1)\psi(2\braket{x,t})\hat{\textbf{1}}_t
    +\gamma_1^2\sum_{t\in\oo^n/2\oo^n}\psi(2\braket{x,t})\hat{\textbf{1}}_t\\
    &=\gamma_1^2\sum_{t\in\oo^n/2\oo^n}(\psi(\varpi t_1^2)-1)\psi(2\braket{x,t})\hat{\textbf{1}}_t+\gamma_1^2\textbf{1}_x.
\end{align*}
Hence, we obtain $\sum_{t\in\oo^n/2\oo^n}(\psi(\varpi t_1^2)-1)\psi(2\braket{x,t})\hat{\textbf{1}}_t\in\mathcal{U}$.
Since the conductor of $\psi$ is $2e$, we can find an element $t_1$ such that $\psi(\varpi t_1^2)-1\ne0$ by \cref{non-trivial-1}. 
By the action $\wt\chit{\varpi a}\wt$, it follows from \cref{iwahori-separate-bottom-pi} that $\hat{\one}_t\in\mathcal{U}$ for some $t$.
Moreover, since $\hat{\textbf{1}}_t=\sum_{s\in\oo^n/2\oo^n}\psi(2\braket{t,s})\textbf{1}_s$,
the vector $\hat{\textbf{1}}_t$ contains all the basis vectors of $\tuni{0}^{\otimes n}$. By \cref{iwahori-separate-bottom}, we have $\textbf{1}_s\in\mathcal{U}$ for all $s\in\oo^n/2\oo^n$, the proof is completed. 
\end{proof}

Before proving irreducibility of the quotient spaces, we introduce some preliminary observations.
First, we obtain the analogue to \cref{K-x(a)-separate-first}.
\begin{corollary}\label{iwahori-separate-normal}
    Let $\mathcal{U}$ be a nonzero irreducible component of $\left(\suni{m}^{\otimes r}\otimes(\tuni{m+1}/\suni{m})^{\otimes n-r}\right)^\pm$ for $0\leq r\leq n-1$ and $f\in\mathcal{U}$.
    Then, one has $\pmphi_x\in\mathcal{U}$ for all $x\in\supp f$ via the action of $\chit{a}\in\tilde{I}$.
\end{corollary}

\begin{proof}
For simplicity, we write $\one_x=\one_{x+2\varpi^m\oo^n}$. The space is generated by $\pmphi_x=\one_x\pm\one_{-x}$, where $x\in(\Ls)^r\times (\lat{m+1}{-m-1}\smallsetminus\Ls)^{n-r}$. 
For $\pmphi_x$ and $\pmphi_y$ with $x\ne y$, we have
$\val(x)=\val(y)=-m-1$.
Thus, by the case $r=n$ of \cref{m-mod-first-case},
there exists $a\in\Sym_n(\oo)$ such that
$\psi(\tp xax)\ne\psi(\tp yay)$.
Hence the calculation in (\ref{important-calc}) is also valid for this space.
\end{proof}

The quotient space 
\begin{equation*}
    \suni{m}/\tuni{m}= \SS(\Ls)/\SS(\Lt)=\SS(\varpi^{-m}\oo/2\varpi^{m+1}\oo)/\SS(\varpi^{-m}\oo/2\varpi^{m}\oo)
\end{equation*}
arises in the case $n=1$. By contrast, the image of this space by $\wt$ has a clearer description of its basis:
\begin{equation*}
    \wt(\suni{m}/\tuni{m})=\wt(\suni{m})/\wt(\tuni{m})=
    \SS(\lat{m}{-m-1})/\SS(\Lt)=\mathrm{span}_{\mathbb{C}}
    \{\textbf{1}_{x+2\varpi^m\oo}:\val(x)=-m-1\}.
\end{equation*}
Therefore, it is often advantageous to work with the representation on 
     this space. For more general subspace $V$ of $S$, the following diagram is commutative:
    \begin{equation*}
        \begin{tikzcd}
        V \arrow[r,"\wt"]\arrow[d,"g\in\tilde{I}"] & \hat{V} \arrow[d,"\wt g\wt^{-1}"]\\
        V \arrow [r,"\wt"]& \hat{V}\rlap{\quad,}
        \end{tikzcd}
    \end{equation*}
where $\hat{V}$ is the image of $V$ by $\wt$. 
Thus, we are led to consider the representation of $\wt\tilde{I}\wt^{-1}$.
We see that this representation includes the action of $\wt\tilde{I}\wt$, \cref{wIw-rep}:
\begin{lemma}\label{wIw-rep}
    Let $V$ be a subspace of $S=\SS(Y)$, and $V^+$ (resp.\ $V^-$) be the space of even (resp.\ odd) functions. Then, some elements of $\wt\tilde{I}\wt^{-1}$ act as follows:
    \begin{align*}
        \wt\chit{a}\wt^{-1}\cdot\pmphi_x&=\pm\gamma_1^6\wt\chit{a}\wt\cdot\pmphi_x;\\
        \wt\hht{b}\wt^{-1}\cdot\pmphi_x&=\hht{b}\cdot\pmphi_x;\\
        \wt\wt\chit{\varpi a}\wt^{-1}\wt\cdot\pmphi_x&=\pm\gamma_1^2\chit{\varpi a}\cdot\pmphi_x.
    \end{align*}
\end{lemma}
\begin{proof}
    Using the fact that $\wt^8=1$ and $\wt^2\cdot\varphi(y)=\gamma_1^2\cdot\varphi(-y)$. We have
    \begin{align*}
        \wt\chit{a}\wt^{-1}\cdot\pmphi_x(y)
        &=\gamma_1^6\wt\chit{a}\wt\cdot\pmphi_x(-y)\\
        &=\pm\gamma_1^6\wt\chit{a}\wt\cdot\pmphi_x(y).
    \end{align*}
    Since $\hht{b}$ lies in the center of $\Mp_{2n}(W)$, the second equality follows immediately. The remaining relation is obtained from the computation
    \begin{align*}
        \wt\wt\chit{\varpi a}\wt^{-1}\wt\cdot\pmphi_x
        &=\wt^2\cdot\psi(\tp x\varpi ax)\pmphi_x\\
        &=\pm\gamma_1^2\psi(\tp x\varpi ax)\pmphi_x.
    \end{align*}
\end{proof}
Consequently, we frequently examine the actions of $\wt\chit{a}\wt$, $\hht{b}$, and $\chit{\varpi a}$. We note, however, that these three elements do not generate the subgroup $\wt\tilde{I}\wt^{-1}$. A fact similar to \cref{iwahori-separate-normal} holds for the action of $\chit{\varpi a}$:
\begin{corollary}\label{iwahori-separate-singular}
    Fix $x, y\in \varpi^{-m-1}\oo^n \smallsetminus \varpi^{-m} \oo^n$ with $m\geq0$.
  If $\psi (\tp x\varpi ax)=\psi (\tp y\varpi ay)$ for all $a\in \Sym_n(\oo)$, 
  then $x\equiv y$ or $x\equiv -y$ mod $2\varpi^{m}\oo^n$.
  Therefore, for an irreducible component $\mathcal{U}$ of 
  $\left(\tuni{m}^{\otimes r}\otimes(\sunihat{m}/\tuni{m})^{\otimes n-r}\right)^\pm$ with $0\leq r\leq n-1$ and $f\in\mathcal{U}$, a basis vector 
  $\pmphi_x$ is in $\mathcal{U}$ via the action of $\chit{\varpi a}\in\wt\tilde{I}\wt^{-1}$.
\end{corollary}
\begin{proof}
Since there exists an index $i$ such that $\val(x_i)=-m-1$, we fix $i$ and choose $j\ne i$.
The proof is to follow the lines of \cref{m-mod-first-case} by applying $\delta(i)=\delta(j)=0$, replacing $a$ with $\varpi a$, and shifting $m$ to $ m+1$. 
Using \cref{iwahori-separate-bottom-pi} in place of \cref{n>1mod2o} and applying \cref{n=1mod2pimo} with $\delta=-1$, we can express $x_i$ and $x_j$ as
\begin{equation}
\left\{ \,
    \begin{aligned}\label{rewrite2}
    & y_i =\pm x_i+2\varpi^{m}z_1;  \\
    & y_j =\pm x_j+2z_2.
    \end{aligned}
\right.
\end{equation}
We now show $z_2\in\varpi ^m\oo$.
Following the argument of \cref{m-mod-first-case}, we have $\varpi(x_iz_2+\varpi^mx_jz_1)\in\oo$. As $\varpi^m x_jz_1$ lies in $\varpi^{-1}\oo$, we obtain $x_iz_2\in \varpi^{-1}\oo$. Since $\val(x_i)=-m-1$, it follows that $z_2\in\varpi^m\oo$.

The space $\left(\tuni{m}^{\otimes r}\otimes(\sunihat{m}/\tuni{m})^{\otimes n-r}\right)^\pm$ is generated by
$\pmphi_x=\one_{x+2\varpi^m\oo^n}\pm\one_{-x+2\varpi^m\oo^n}$, where $x\in(\Lt)^r\times (\lat{m}{-m-1}\smallsetminus\Lt)^{n-r}$.
For $\pmphi_x$ and $\pmphi_y$ with $x\ne y$, we know that $\val(x)=\val(y)=-m-1$. Hence the first statement implies that 
there exists $a\in\Sym_n(\oo)$ so that 
$\psi (\tp x\varpi ax)\ne\psi (\tp y\varpi ay)$.
Hence, the calculation in (\ref{important-calc}) is also valid for this case.
\end{proof}

\subsection{The case of $n=1$}\label{iwahori1}
In this case,
    the filtration (\ref{iwahori-filtration}) of $\SS(Y)$ is 
\begin{equation*}
    \tuni{0}\subset\suni{0}\subset\tuni{1}\subset\suni{1}\subset\cdots
    \tuni{m}\subset\suni{m}\subset\tuni{m+1}\subset\cdots,
\end{equation*}
which yields the orthogonal decomposition
\begin{align*}
    \SS(Y) = \tuni{0}\oplus\bigoplus_{m\geq0}(\suni{m}/\tuni{m}\oplus\tuni{m+1}/\suni{m}).
\end{align*}
Thus, we have only to consider the irreducibility of two components
$\suni{m}/\tuni{m}$ and $\tuni{m+1}/\suni{m}$. More generally, we show the irreducibility of the $n$-times tensor product of these spaces.
\begin{proposition}\label{iwahori-quotient}
    Let $m\geq0$. The four spaces
      $((\tuni{m+1}/\suni{m})^{\otimes n})^\pm$ and $((\suni{m}/\tuni{m})^{\otimes n})^\pm$ are irreducible 
     as representations of $\tilde{I}=\tilde{I}^{(n)}$.
\end{proposition}

\begin{proof}
We first show the irreducibility of $((\tuni{m+1}/\suni{m})^{\otimes n})^\pm$.
    Let $\mathcal{U}$ be an irreducible component of the space.
    By \cref{iwahori-separate-normal}, we can assume that 
    $\varphi_x=\pmphi_x\in\mathcal{U}$ for some $x\in(\varpi^{-m-1}\oo^\times)^n$. 
    Let $b=\diag(b_1,\ldots,b_n)$ with $b_i\in\oo^\times$. Then 
    $\hht{b}$ acts on $\varphi_x$ by 
    $$
    \hht{b}\varphi_x=\beta_{b}\varphi_{\tp b^{-1}x}=\beta_{b}\varphi_{b^{-1}x}=\beta_{b}\varphi_{\tp (b_1^{-1}x_1,\ldots,b_n^{-1}x_n)}
    $$
    and one has $\diag(\oo^\times)\cdot x=(\varpi^{-m-1}\oo^\times)^n$. 
    Therefore, $\varphi_t$ lies in $\mathcal{U}$ for all $t\in(\varpi^{-m-1}\oo^\times)^n$.

    Next, we consider the case of $((\suni{m}/\tuni{m})^{\otimes n})^\pm$.
    We analyze the irreducibility on the space $\wt(\suni{m}/\tuni{m})=\sunihat{m}/\tuni{m}$ with actions of $\wt\tilde{I}\wt^{-1}$.
    Note that 
    \begin{align*}
        ((\sunihat{m}/\tuni{m})^{\otimes n})^\pm=\mathrm{span}_{\mathbb{C}}\{\textbf{1}_{x+2\varpi^{m}\oo^n}\pm\textbf{1}_{-x+2\varpi^{m}\oo^n}:x\in(\varpi^{-m-1}\oo^\times)^n\}.
    \end{align*}
    Let $\mathcal{U}$ be an irreducible component of $((\sunihat{m}/\tuni{m})^{\otimes n})^\pm$.
    Recall that the elements $\chit{\varpi a}$ and $\hht{b}$ act on the space $\wt(\suni{m}/\tuni{m})$ (see \cref{wIw-rep}).
    By \cref{iwahori-separate-singular}, we can assume that a vector $\pmphi_x$ lies in $\mathcal{U}$ by the action of $\chit{\varpi a}$.
    Since
    \begin{equation*}
        \{\diag(b_1,\ldots,b_n):b_i\in\oo^\times\}\cdot x=
        \{(y_1,\ldots,y_n):y_i\in\varpi^{-m-1}\oo^\times\},
    \end{equation*}
    the action $\hht{b}\cdot\pmphi_x=\pmphi_{b^{-1}x}$ is transitive.
    Hence, $\pmphi_t\in\mathcal{U}$ for all $t\in(\varpi^{-m-1}\oo^\times)^n$.
    \end{proof}

    Combining the above two propositions, we have the following result:
    \begin{proposition}\label{iwahori-n=1-theorem}
        The direct sum decomposition 
        \begin{align*}
    S(Y) = \tuni{0}\oplus\bigoplus_{m\geq0}[(\suni{m}/\tuni{m})^+\oplus(\suni{m}/\tuni{m})^-\oplus(\tuni{m+1}/\suni{m})^+\oplus(\tuni{m+1}/\suni{m})^-]
\end{align*}
provides an irreducible decomposition of the Weil representation restricted to the $2\times2$ Iwahori subgroup $\tilde{I}=\tilde{I}^{(1)}$.
\end{proposition}
    
\subsection{The case of $n=2$}\label{iwahori2}
In the case of $n=2$, the filtration (\ref{iwahori-filtration}) of $S(Y)$ is 
\begin{align*}
&\tuni{0}\otimes\tuni{0}
\subset
\suni{0}\otimes\tuni{0}
\subset
\tuni{1}\otimes\tuni{1}
\subset\cdots\\
&\cdots\subset\tuni{m}\otimes\tuni{m}
\subset
\suni{m}\otimes\tuni{m}
\subset
\suni{m}\otimes\suni{m}
\subset
\suni{m}\otimes\tuni{m+1}
\subset
\tuni{m+1}\otimes\tuni{m+1}
\subset\cdots.
\end{align*}
Hence, we consider four cases of quotient spaces for $m\geq0$:
\begin{enumerate}[label={\textup{(\arabic*)}}]
\item[(1)] $(\suni{m}\otimes\tuni{m})/(\tuni{m}\otimes\tuni{m})\cong
\suni{m}/\tuni{m}\otimes\tuni{m}$.

\item[(1')] $(\suni{m}\otimes\suni{m})/(\suni{m}\otimes\tuni{m})\cong
\suni{m}\otimes\suni{m}/\tuni{m}$.

\item[(2)] $(\suni{m}\otimes\tuni{m+1})/(\suni{m}\otimes\suni{m})\cong
\suni{m}\otimes\tuni{m+1}/\suni{m}$.

\item[(2')] $(\tuni{m+1}\otimes\tuni{m+1})/(\suni{m}\otimes\tuni{m+1})\cong
\tuni{m+1}/\suni{m}\otimes\tuni{m+1}$.
\end{enumerate}
Since $\suni{m}\cong\tuni{m}\oplus\suni{m}/\tuni{m}$ and $\tuni{m+1}\cong\suni{m}\oplus\tuni{m+1}/\suni{m}$, we can decompose these spaces of (1') and (2') as follows:
\begin{equation}\label{hello2}
\begin{aligned}
    \suni{m}\otimes\suni{m}/\tuni{m}&\cong    (\suni{m}/\tuni{m})^{\otimes2}\oplus(\tuni{m}\otimes\suni{m}/\tuni{m}),
    \\ 
    \tuni{m+1}/\suni{m}\otimes\tuni{m+1}&\cong
    (\tuni{m+1}/\suni{m})^{\otimes2}\oplus(\tuni{m+1}/\suni{m}\otimes\suni{m}).
\end{aligned}
\end{equation}
By \cref{iwahori-quotient}, 
we have obtained the irreducible decomposition of the spaces of 
$(\suni{m}/\tuni{m})^{\otimes2}$ and $(\tuni{m+1}/\suni{m})^{\otimes2}$. 
So the remaining cases are (1), (2), and the right components of (\ref{hello2}). But we see that we have only to show the irreducibility of (1) and (2), \cref{swapping-lemma}. We show this for the more general case $\tilde{I}=\tilde{I}^{(n)}$:
\begin{lemma}\label{swapping-lemma}
    Let $\mathfrak{S}_n$ be the symmetric group of degree $n$, and 
    $V\subset\SS(Y)$ be a finite dimensional subspace of the form     $V=\bigotimes_{i=1}^nV_i$.
    Take $\pi\in\mathfrak{S}_n$ and let $P$ be the permutation matrix corresponding to $\pi$.
    Then, the following diagram
    \begin{equation*}
        \begin{tikzcd}
        \bigotimes_{i=1}^nV_i\arrow[r,"\pi"]\arrow[d,"g"]&\bigotimes_{i=1}^nV_{\pi (i)}\arrow[d, "P^*(g)"]\\
        \bigotimes_{i=1}^nV_i\arrow[r, "\pi"]&\bigotimes_{i=1}^nV_{\pi (i)}
        \end{tikzcd}
    \end{equation*}
    is commutative with isomorphism $P^*\colon\tilde{I}\rightarrow\tilde{I}$, $g\mapsto \hht{\tp P^{-1}}g\hht{\tp P}$. Hence, the space $(\bigotimes_{i=1}^nV_i)^\pm$ is irreducible if and only if $(\bigotimes_{i=1}^nV_{\pi (i)})^\pm$ is irreducible.
\end{lemma}
\begin{proof}
Take a basis $\otimes_i\textbf{1}_{x_i}\in\bigotimes_{i=1}^nV_i$.
Since one has 
$\pi\cdot\otimes_i\textbf{1}_{x_i}
 =\otimes_i\textbf{1}_{x_{\pi(i)}}
=\textbf{1}_{Px}$, the map $\pi$ agrees with the action of $\hht{\tp P^{-1}}$ in $\SS(Y)$. Hence, we see that the diagram is commutative.

We also have to show that the image of $P^*$ is in $\tilde{I}$.
We recall that $I=
\begin{bmatrix}
    A&B\\
    C&D
\end{bmatrix}$ takes the form
\begin{align*}\label{hello0728-1}
A=D=
\begin{bmatrix}
\oo^\times & \oo & \cdots & \oo\\
\varpi\oo & \oo^\times & \cdots & \oo\\
\vdots & \vdots & \ddots & \vdots\\
\varpi\oo & \varpi\oo & \cdots & \oo^\times
\end{bmatrix},\quad B\in\Mat_{n\times n}(\oo),\quad C\in\Mat_{n\times n}(\varpi\oo).
\end{align*}
Since the permutation matrix $P$ is orthogonal, one has 
$\tp P=P^{-1}$. Hence, we have
\begin{align*}
    P^*\left(
    \begin{bmatrix}
    A&B\\
    C&D
    \end{bmatrix}
    \right)
    &=
    \begin{bmatrix}
    \tp P^{-1}&0\\
    0&P
\end{bmatrix}
\begin{bmatrix}
    A&B\\
    C&D
\end{bmatrix}
\begin{bmatrix}
    \tp P&0\\
    0&P^{-1}
\end{bmatrix}\\
&=
    \begin{bmatrix}
        P&0\\
        0&P
    \end{bmatrix}
    \begin{bmatrix}
    A&B\\
    C&D
\end{bmatrix}
\begin{bmatrix}
        P^{-1}&0\\
        0&P^{-1}
    \end{bmatrix}
    =\begin{bmatrix}
    PAP^{-1}&PBP^{-1}\\
    PCP^{-1}&PDP^{-1}
\end{bmatrix}.
\end{align*}
As the conjugation of $P$ only permutes the rows and columns at the same time via $\pi$, the image of $P^*$ lies in the Iwahori subgroup $\tilde{I}$.
\end{proof}
\begin{proposition}\label{iwahori-quotient-n=2-normal}
    Let $m\geq 0$. The two components
     $(\suni{m}\otimes\tuni{m+1}/\suni{m})^{\pm}$
     are irreducible as representations of $\tilde{I}=\tilde{I}^{(2)}$.
\end{proposition}
\begin{proof}
    We write the space explicitly:
    \begin{equation*}
        (\suni{m}\otimes\tuni{m+1}/\suni{m})^{\pm}
        =\{\pmphi_x=\textbf{1}_x\pm\textbf{1}_{-x}:x_1\in\Ls, x_2\in \lat{m+1}{-m-1}\smallsetminus\Ls\},
    \end{equation*}
    where $\one_x\coloneq \one_{x+2\varpi^{m+1}\oo^n}$.
    Let $\mathcal{U}$ be an irreducible component of 
    $(\suni{m}\otimes\tuni{m+1}/\suni{m})^{\pm}$. 
    It follows from \Cref{iwahori-separate-normal} that we can assume that a vector $\pmphi_x$ lies in $\mathcal{U}$. As $\chit{a}$ can separate all vectors, it remains to show that the vector $\pmphi_x$ generates the entire space.
    For 
    $a=
    \begin{bmatrix}
        0&u \\
         u&0
    \end{bmatrix}
    $ with $u\in\oo^\times$, we define $S^{\pm x}(y)=[\wt\chit{\varpi a}\wt\one_{\pm x}](y)$. Then
    \begin{align*}
        S^{\pm x}(y)&=[\wt\chit{a}\wt\one_{\pm x}](y)\\
        &=\gamma_1\wt\chit{\varpi a}\sum_{t\in\lat{m}{-m-1}\times\lat{m+1}{-m-1}}\psi(\pm2\braket{x,t})\one_t(y)\\
        &=\gamma_1^2\sum_{t\in\lat{m}{-m-1}\times\lat{m+1}{-m-1}}\psi(\pm2\braket{x,t})\psi(\tp t\varpi at)\hat{\one}_t(y)\\
        &=\gamma_1^2\sum_{t\in\lat{m}{-m-1}\times\lat{m+1}{-m-1}}\psi(\pm2\braket{x,t})\psi(2\tp t\varpi at)\psi(2\braket{t,y})\\
        &=\gamma_1^2\sum_{t\in\lat{m}{-m-1}\times\lat{m+1}{-m-1}}
        \psi(2\varpi ut_1t_2)\psi(2\braket{x\pm y,t}).
    \end{align*}
    
    Setting $y=\vect{y_1}{x_2}$, we have
    \begin{align*}
        \gamma_1^{-2}S^{-x}\vect{y_1}{x_2}&=
        \sum_{t\in\lat{m}{-m-1}\times\lat{m+1}{-m-1}}\psi(2u\varpi t_1t_2)\psi(2(x_1-y_1)t_1)\\
        &=\sum_{t_1\in\lat{m}{-m-1}}\psi(2(x_1-y_1)t_1)
        \sum_{t_2\in\lat{m+1}{-m-1}}\psi(2u\varpi t_1t_2)\\
        &=|\lat{m+1}{-m-1}|\sum_{\substack{t_1\in\lat{m}{-m-1}\text{ and}\\t_1\in2\varpi^m\oo}}\psi(2(x_1-y_1)t_1)\\
        &=|\lat{m+1}{-m-1}|\psi(0)=q^{2m+2+e}.
    \end{align*}
    On the other hand, the summation $\gamma_1^{-2}S^{+x}\vect{y_1}{x_2}$ is computed as follows:
    \begin{align*}
        \gamma_1^{-2}S^{+x}\vect{y_1}{x_2}&=
        \sum_{t\in\lat{m}{-m-1}\times\lat{m+1}{-m-1}}\psi(2u\varpi t_1t_2)\psi(2(x_1+y_1)t_1+4x_2t_2)\\
        &=\sum_{t_1\in\lat{m}{-m-1}}\psi(2(x_1+y_1)t_1)
        \sum_{t_2\in\lat{m+1}{-m-1}}\psi(2(u\varpi t_1+2x_2)t_2)\\
        &=|\lat{m+1}{-m-1}/\lat{m}{-m-1}|^{-1}\sum_{t_1\in\lat{m+1}{-m-1}}\psi(2(x_1+y_1)t_1)
        \sum_{t_2\in\lat{m+1}{-m-1}}\psi(2(u\varpi t_1+2x_2)t_2).
    \end{align*}
    The sum $\sum_{t_2\in\lat{m+1}{-m-1}}\psi(2(u\varpi t_1+2x_2)t_2)$ is nonzero if and only if $t_1$ satisfies
    $u\varpi t_1+2x_2\in\varpi^{m+1}\oo$.
    If $p\ne 2$, we have $\val(2x_2)=-m-1$ and hence $\val(\varpi t_1+2x_2)=-m-1$. So the condition is not satisfied for all $t_1$. Hence, the sum is zero and so is $S^{+x}\vect{y_1}{x_2}$.
    If $p=2$, we note that $\val(2)=e\geq1$.
    Then the condition is equivalent to 
    $t_1\in-2u^{-1}\varpi^{-1}x_2+2\varpi^{m}\oo$.
    So when $p=2$ we have
    \begin{align*}
    \gamma_1^{-2}S^{+x}\vect{y_1}{x_2}&=
         q^{-1}|\lat{m+1}{-m-1}|
         \sum_{\substack{t_1\in\lat{m+1}{-m-1}\text{ and}\\u\varpi t_1+2x_2\in2\varpi^{m+1}\oo }}\psi(2(x_1+y_1)t_1)\\
         &=
         q^{2m+1+e}
         \sum_{z\in2\varpi^m\oo/2\varpi^{m+1}\oo}
         \psi(2(x_1+y_1)(-2u^{-1}\varpi^{-1}x_2+z))\\
         &=q^{2m+2+e}\psi(-4u^{-1}\varpi^{-1}(x_1+y_1)x_2).
    \end{align*}
    Therefore, we obtain
    \begin{equation}\label{hello0712-2}
    \begin{aligned}
        \wt\chit{\varpi a}\wt\pmphi_x\vect{y_1}{x_2}&=cS^{+x}\vect{y_1}{x_2}\pm cS^{-x}\vect{y_1}{x_2}\\
        &=
        \begin{cases}
        \pm cq^{2m+2+e}&p\ne2;\\
        cq^{2m+2+e}(1\pm\psi(-4u^{-1}\varpi^{-1}(x_1+y_1)x_2)) &p=2
        \end{cases}
    \end{aligned}
    \end{equation}
    for $c=\gamma_1^{2}$.
    If $p\ne2$, for each $x_2$ one has $\wt\chit{\varpi a}\wt\pmphi_x\vect{y_1}{x_2}\ne0$ for all $y_1$. Hence \cref{iwahori-separate-normal} implies $\pmphi_z\in\mathcal{U}$ for all $z\in\lat{m+1}{-m}\times\{x_2\}$. As $x_2$ only takes values of valuation $-m-1$, we have $\pmphi_z\in\mathcal{U}$ for all $z\in\lat{m+1}{-m}\times(\lat{m+1}{-m-1}\smallsetminus\lat{m+1}{-m})$ by the action of $\hht{b}$. The proof is complete.
    
    It remains to consider the case of $p=2$. If $\val(x_1)>-m$, we take $y_1\in\varpi^{-m}\oo^\times$. Then $\val(\varpi^{-1}(x_1+y_1)x_2)\leq-2$, we have $\wt\chit{a}\wt\pmphi_x\vect{y_1}{x_2}\ne0$ for some $u\in\oo^\times$ by \cref{non-trivial-1,non-trivial-minus1}. So $\pmphi_{\tp (y_1,x_2)}\in\mathcal{U}$ by \cref{iwahori-separate-normal} and we can assume that $\val(x_1)=-m$. By the action $\hht{b}\cdot\pmphi_x$, we have $\pmphi_z\in\mathcal{U}$ for all $z\in\varpi^{-m}\oo^\times\times \{x_2\}$. For  $y_1\in\varpi^{-m+1}\oo$, we can find an element $a=
    \begin{bmatrix}
        0&u\\
         u&0
    \end{bmatrix}$ 
    such that $\wt\chit{\varpi a}\wt\pmphi_x\vect{y_1}{x_2}\ne0$ by \cref{non-trivial-1} or \cref{non-trivial-minus1}. Hence we also have $\pmphi_z\in\mathcal{U}$ for $z\in\varpi^{-m}\oo\times\{x_2\}$ by \cref{iwahori-separate-normal}.
    By the action of $\hht{b}$, we have $\pmphi_z$ for all $z\in\varpi^{-m}\oo\times\varpi^{-m-1}\oo^\times$.
    The proof is complete.
\end{proof}

\begin{proposition}\label{iwahori-quotient-n=2-singular}
    Let $m\geq0$. The two components
    $(\tuni{m}\otimes\suni{m}/\tuni{m})^\pm$
are irreducible as representations of $\tilde{I}=\tilde{I}^{(2)}$.
\end{proposition}

\begin{proof}
First, we assume that the space is not equal to $(\tuni{0}\otimes \suni{0}/\tuni{0})^+$.
We consider the representation of $w\tilde{I}w^{-1}$ on $(\tuni{m}\otimes\sunihat{m}/\tuni{m})^\pm$.
We write the space explicitly:
\begin{equation*}
    (\tuni{m}\otimes\sunihat{m}/\tuni{m})^\pm
    =\mathrm{span}_{\mathbb{C}}\{\pmphi_x=\textbf{1}_x\pm\textbf{1}_{-x}:x_1\in\Lt, x_2\in \lat{m}{-m-1}\smallsetminus\Lt \},
\end{equation*}
where $\one_x\coloneq \one_{x+2\varpi^m\oo^n}$.
Let $\mathcal{U}$ be an irreducible component of $(\tuni{m}\otimes\sunihat{m}/\tuni{m})^\pm$.
By \cref{iwahori-separate-singular}, we can assume that a vector
$\pmphi_x=\textbf{1}_x\pm\textbf{1}_{-x}$ lies in $\mathcal{U}$. We note that $x_2\in\lat{m}{m-1}\smallsetminus\Lt$, i.e, $\val(x_2)=-m-1$.
It remains to show that the vector $\pphi_x$ generates the entire space. 
We compute the action of $\wt\chit{a}\wt$:
\begin{align*}
    \wt\pmphi_x &=\gamma_1\sum_{t\in\Lt\times\Ls}\psi(2\braket{x,t})\mathbf{1}_t
    \pm \gamma_1\sum_{t\in\Lt\times\Ls}\psi(-2\braket{x,t})\mathbf{1}_t,\\
     \wt\chit{a}\wt\pmphi_x &
    =\gamma_1^2\sum_{t\in\Lt\times\Ls}\psi(\tp tat)\psi(2\braket{x,t})\hat{\mathbf{1}}_t
    \pm \gamma_1^2\sum_{t\in\Lt\times\Ls}\psi(\tp tat)\psi(-2\braket{x,t})\hat{\mathbf{1}}_t.
\end{align*}
Since $\hat{\mathbf{1}}_t(y)=\psi(2\braket{t,y})$, we have
\begin{align*}
    \wt\chit{a}\wt\pmphi_x(y) 
    =\gamma_1^2\sum_{t\in\Lt\times\Ls}\psi(\tp tat+2\braket{x+y,t})
    \pm \gamma_1^2\sum_{t\in\Lt\times\Ls}\psi(\tp tat+2\braket{x-y,t}).
\end{align*}
We show that for each $y$, we can find $a\in\Sym_n(\oo)$ such that $\wt\chit{a}\wt\pphi_x\vect{y_1}{x_2}\ne0$.
We take $a=\
\begin{bmatrix}
    0 & u\\
    u & 0
\end{bmatrix}$, where $u\in\oo$. 
We denote by $S^{+x}(y)$ (resp.\ $S^{-x}(y)$) the first (resp.\ second) summation. Since the map $t_2\mapsto \psi(2t_1t_2)$ is well-defined modulo $2\varpi^m\oo$, the summation $\gamma_1^{-2}S^{+x}(y)$ is computed as follows:
\begin{align*}
    \gamma_1^{-2}S^{+x}(y)&=\sum_{t\in\Lt\times\Ls}\psi(\tp tat+2\braket{x+y,t})\\&=
    |\Ls/\Lt|\sum_{t_1\in\Lt}\psi(2(x_1+y_1)t_1)\sum_{t_2\in\Lt}(2(ut_1+2x_2)t_2).\\
\end{align*}
Here, since the summation of $t_2$ is nonzero if and only if $ut_1+2x_2\in2\varpi^m\oo$,
we have 
\begin{equation}\label{hello0728-3}
    \begin{aligned}
        \gamma_1^{-2}S^{+x}(y)&=|\Ls/\Lt|\cdot|\Lt|\sum_{\substack{t_1\in\Lt\\ut_1+2x_2\in2\varpi^m\oo}}\psi(2(x_1+y_1)t_1)\\
    &=q^{2m+1+e}\sum_{\substack{t1\in\Lt\\t_1\in2\varpi^m\oo}}\psi(2(x_1+y_1)(t_1-u^{-1}x_2))\\
    &=q^{2m+1+e}\psi(-4u^{-1}x_2(x_1+y_1)).
    \end{aligned}
\end{equation}
On the other hand, the latter summation $S^{-x}\vect{y_1}{x_2}$ is computed as follows:
\begin{align*}
    \gamma_1^{-2}S^{-x}\vect{y_1}{x_2}&=\sum_{t_1\in\Lt}\psi(2(x_1-y_1)t_1)\sum_{\Ls}\psi(2ut_1t_2)\\
    &=|\Ls|\sum_{\substack{t_1\in\Lt\\t_1\in2\varpi^m\oo}}\psi(2(x_1-y_1)t_1)\\
    &=q^{2m+1+e}.
\end{align*}
So we obtain
\begin{align}\label{hello0713-3}
    \wt\chit{a}\wt\pmphi_x(y)=\gamma_1^2q^{2m+1+e}(\psi(-4u^{-1}x_2(x_1+y_1))\pm1).
\end{align}
If $\val(x_1)>-m$, take $y\in\Lt$ such that $\val(y_1)=-m$.
Then $\val(u^{-1}x_2(x_1+y_1))=(-m-1)-m=-2m-1$, so we have $\wt\chit{a}\wt\pmphi_x(y)\ne0$ for some $u\in\oo^\times$ by \cref{non-trivial-1,non-trivial-minus1}. Hence $\pmphi_y$ is in $\mathcal{U}$ by \cref{iwahori-separate-singular} and we can assume that $\val(x_1)=-m$.
By the action of $\hht{b}$, we have $\pmphi_y\in\mathcal{U}$ for all $\val(y_1)=-m$ and $\val(y_2)=-m-1$.
For arbitrary $y=\vect{y_1}{y_2}$ with $\val(y_1)>-m$, we can find an element $a$ such that 
$\wt\chit{a}\wt\pmphi_x(y)\ne0$ by
the above two lemmas. Hence $\pmphi_y$ lies in $\mathcal{U}$ for all $y$ by \cref{iwahori-separate-singular}.

Finally, we show the irreducibility of $(\tuni{0}\otimes\suni{0}/\suni{0})^+$.
One has $(\tuni{0})^+=\tuni{0}$ because $x_1\equiv-x_1$ (mod $2\oo$).
Hence, we have
$(\tuni{0}\otimes\suni{0}/\tuni{0})^+=\tuni{0}\otimes(\suni{0}/\tuni{0})^+.
$
On the other hand, the $4\times4$-size Iwahori $\tilde{I}=\tilde{I}^{(2)}$ contains the product of the $2\times2$-size Iwahori subgroups $\tilde{I}^{(1)}\times\tilde{I}^{(1)}$, which acts on each tensor factor.
By \cref{iwahori-quotient} with $m=0$ and \cref{iwahori-bottom} with $n=1$, the spaces $\tuni{0}$ and $(\suni{0}/\tuni{0})^+$ are irreducible as representations of $\tilde{I}^{(1)}$.
Therefore, we see that the subgroup $\tilde{I}^{(2)}$ acts irreducibly on the space $\tuni{0}\otimes(\suni{0}/\tuni{0})^+$.
\end{proof}

The above two propositions gives an irreducible decomposition of $\omega_\psi|_{\tilde{I}}$ with $\tilde{I}=\tilde{I}^{(2)}$.

\subsection{The case of $n\geq3$}\label{iwahori3}
We recall that the filtration (\ref{iwahori-filtration}) of $S=\SS(Y)$:
\begin{equation*}
\begin{array}{ccccccc}
\Slow{0}{0} &\subset& \Shigh{0}{1} &=& \Slow{0}{1} &&\\
\cap && \cup && \cap &&\\
\Slow{1}{0} &\subset& \Shigh{1}{1} &\subset& \Slow{1}{1}&&\\
\cap && \cup && \cap &&\\
\vdots && \vdots && \vdots &&\vdots\\
\cap && \cup && \cap &&\cup\\
\Slow{n-1}{0} &\subset& \Shigh{n-1}{1} &\subset& \Slow{n-1}{1}&\subset&\Shigh{n-1}{2}\\
\cap && \cup && \cap &&\cup\\
\Slow{n}{0} &=& \Shigh{n}{1} &\subset& \Slow{n}{1}&=&\Shigh{n}{2}\rlap{\quad.}\\
\end{array}
\end{equation*}
From this, we can make a finer filtration from vertical inclusions:
\begin{align*}
    \Slow{0}{0}\subset\Slow{1}{0}\subset\cdots\subset\Slow{n}{0}=\Shigh{n}{1}\subset\Shigh{n-1}{1}\subset\cdots\subset\Shigh{0}{1}
    =\Slow{0}{1}\subset\Slow{1}{1}\subset\cdots.
\end{align*}
Thus, the quotients of these spaces consists of two kinds of spaces:
\begin{enumerate}[label={\textup{(\arabic*)}}]
    \item[(1)] $\Slow{r}{m}/\Slow{r-1}{m}
    \cong
    \suni{m}^{\otimes r-1}\otimes\suni{m}/\tuni{m}\otimes\tuni{m}^{\otimes n-r}$ for $1\leq r\leq n$.
    \item [(2)]$\Shigh{r}{m}/\Shigh{r+1}{m}
    \cong\suni{m}^{\otimes r}\otimes\tuni{m+1}/\suni{m}\otimes\tuni{m+1}^{\otimes n-r-1}$
    for $0\leq r\leq n-1$.
\end{enumerate}
In (1), as seen in the case of $n=2$, we must decompose further:
\begin{equation*}
    \suni{m}^{\otimes r-1}\otimes\suni{m}/\tuni{m}\otimes\tuni{m}^{\otimes n-r}\cong
    (\tuni{m}\oplus\suni{m}/\tuni{m})^{\otimes r-1}\otimes\suni{m}/\tuni{m}\otimes\tuni{m}^{\otimes n-r}.
\end{equation*}
Similarly, we also decompose (2):
\begin{equation*}
    \suni{m}^{\otimes r}\otimes\tuni{m+1}/\suni{m}\otimes\tuni{m+1}^{\otimes n-r-1}
    \cong\suni{m}^{\otimes r}\otimes\tuni{m+1}/\suni{m}\otimes(\suni{m}\oplus\tuni{m+1}/\suni{m})^{\otimes n-r-1}.
\end{equation*}
Hence, the irreducible components are obtained by decomposing each component of the above decomposition into its even and odd parts, which is another main result of this paper:
\begin{theorem}\label{iwahori-main-theorem}
The irreducible decomposition of $\omega_\psi|_{\tilde I}$ is given by
\begin{equation}\label{iwahori-main-decomposition}
    \SS(Y)\cong\tuni{0}^{\otimes n}
    \oplus\bigoplus_{m\geq0}
    \left(
        \X_m^+\oplus\X_m^-
        \oplus\Y_m^+\oplus\Y_m^-
    \right),
\end{equation}
where $\X_m$ and $\Y_m$ are defined by
\begin{align*}
    \X_m\oplus\tuni{m}^{\otimes n}
    &=
    \left(\tuni{m}\oplus\suni{m}/\tuni{m}\right)^{\otimes n},\\
    \Y_m\oplus\suni{m}^{\otimes n}
    &=
    \left(\suni{m}\oplus\tuni{m+1}/\suni{m}\right)^{\otimes n}.
\end{align*}
The dimension of each component is written explicitly.
Moreover, $\omega_\psi|_{\tilde I}$ is multiplicity-free.
\end{theorem}

We note that this result is also valid when $n=2$ and $n=1$. We will prove the above main theorem. As seen in \cref{swapping-lemma}, permuting tensor factors does not affect irreducibility. As we have shown that $\tuni{0}^{\otimes n}$, $((\tuni{m+1}/\suni{m})^{\otimes n})^\pm$ and $((\suni{m}/\tuni{m})^{\otimes n})^\pm$ are irreducible in 
\cref{iwahori-bottom,iwahori-quotient}, it suffices to check the irreducibility of these two types for $1\leq r\leq n-1$ and $m\geq0$:
\begin{enumerate}[label={\textup{(\roman*)}}]
    \item $(\tuni{m}^{\otimes r}\otimes
    (\suni{m/}\tuni{m})^{\otimes n-r})^+$ and $(\tuni{m}^{\otimes r}\otimes
    (\suni{m/}\tuni{m})^{\otimes n-r})^-$, 
    \item  $(\suni{m}^{\otimes r}\otimes(\tuni{m+1}/\suni{m})^{\otimes n-r})^+$ and $(\suni{m}^{\otimes r}\otimes(\tuni{m+1}/\suni{m})^{\otimes n-r})^-$.
\end{enumerate}
To show the irreducibility of these spaces, we use the idea of applying the case of $n=2$.
As seen in \cref{iwahori-subrep}, the Iwahori $\tilde{I}^{(N)}$ with $N=\{i,j\}$ acts only on the $i$-th and $j$-th factors of the tensor product. We frequently use the fact for $n=2$ as $N=\{i,j\}\subset\{1,\ldots,n\}$.
\begin{proposition}\label{iwahori-quotient-n>3-normal}
    For $1\leq r\leq n-1$ and $m\geq0$, the two spaces $(\suni{m}^{\otimes r}\otimes(\tuni{m+1}/\suni{m})^{\otimes n-r})^+$ and $(\suni{m}^{\otimes r}\otimes(\tuni{m+1}/\suni{m})^{\otimes n-r})^-$
    are irreducible.
\end{proposition}
\begin{proof}
    The proof is almost the same as that of \cref{iwahori-quotient-n=2-normal}.
    We 
    write $\one_x=\one_{x+2\varpi^{m+1}\oo^n}$ and 
    use the fact that $\wt\textbf{1}_x=\hat{\textbf{1}}_x=\otimes_i\hat{\textbf{1}}_{x_i}$.
    Let $\mathcal{U}$ be an irreducible component of $(\suni{m}^{\otimes r}\otimes(\tuni{m+1}/\suni{m})^{\otimes n-r})^\pm$.
    We can assume that $\pmphi_x\in\mathcal{U}$ by \cref{iwahori-separate-normal}.

    First, we fix $x_2,\cdots,x_n$ and set
    \begin{equation}\label{vect-denote}
        x^{(1)}=\threevect{x_1}{\vdots}{x_r},\quad
        x^{(2)}=\threevect{x_{r+1}}{\vdots}{x_n}\textup{ and }
        \xlow=\threevect{x_2}{\vdots}{x_{n-1}}.
    \end{equation}
    We choose $a\in\Sym_n(\oo)$ so that $\tp tat=2ut_1t_n$ for $u\in\oo^\times$. Then, the action of $\chit{\varpi a}$ may be regarded as acting only on the first and $n$-th tensor factors. 
    Moreover, the element $w_{1n}=w_{\eps_1}w_{\eps_n}$ also acts on only those tensor factors, and $\wt_{1n}\chit{\varpi a}\wt_{1n}$ belongs to the Iwahori 
    subgroup $\tilde{I}$ for $a$ as defined above.
    Thus we have
    \begin{equation}
    \begin{aligned}
        \wt_{1n}\chit{\varpi a}\wt_{1n}\cdot\textbf{1}_x&=
        \left(\wt_{1n}\chit{\varpi a}\wt_{1n}\cdot(\one_{x_1}\otimes\one_{x_n})\right)
        \otimes(\one_{x_2}\otimes\cdots\otimes\one_{x_{n-1}}).
    \end{aligned}
    \end{equation}
    Therefore, we obtain
    \begin{equation}\label{hello0712-1}
    \begin{aligned}
        \wt_{1n}\chit{\varpi a}\wt_{1n}\cdot\pmphi_x=
        (\wt_{1n}\chit{\varpi a}&\wt_{1n}\cdot\one_{\tp(x_1,x_n)})
        \otimes\one_{\xlow}\\&\pm
        (\wt_{1n}\chit{\varpi a}\wt_{1n}\cdot\one_{\tp(-x_1,-x_n)})
        \otimes\one_{-\xlow}.
    \end{aligned}
    \end{equation}    
    We now distinguish two cases: $\xlow\equiv -\xlow$ mod $2\varpi^{m+1}\oo^{n-2}$ or otherwise. In the former case, the equation (\ref{hello0712-1}) becomes
    \begin{align*}
        \wt_{1n}\chit{\varpi a}\wt_{1n}\cdot\pmphi_x
        =\wt_{1n}\chit{\varpi a}\wt_{1n}\cdot\pmphi_{\tp(x_1,x_n)}\otimes\one_{\xlow}.
    \end{align*}
    The calculation of $\wt_{1n}\chit{\varpi a}\wt_{1n}\cdot\pmphi_{\tp(x_1,x_n)}$ is exactly the same as that of $\wt\chit{\varpi a}\wt\pmphi_x$ in the proof of $n=2$, namely, (\ref{hello0712-2}). Hence we have
    \begin{equation*}
         \wt_{1n}\chit{\varpi a}\wt_{1n}\cdot\pmphi_x
         \threevect{y_1}{\xlow}{x_n}
         =cq^{2m+1}(\psi(-4u^{-1}x_n(x_1+y_1))\pm1)
    \end{equation*}
    for $u\in\oo^\times$ and an $8$-th root of unity $c$. By the same argument as in the case  $n=2$, we see that $\pmphi_{x}\in\mathcal{U}$ for all $x_1\in\Ls$.
    In latter
    case, we see that $\one_{-\xlow}(\xlow)=0$. So it follows 
    that $\wt_{1n}\chit{\varpi a}\wt_{1n}\one_x
    \threevect{y_1}{\xlow}{-x_n}$ vanishes. Hence
    we obtain 
    \begin{align*}
        \wt_{1n}\chit{\varpi a}\wt_{1n}\cdot\pmphi_x
        \threevect{y_1}{\xlow}{-x_n}
        &=cS^{+\vect{x_1}{x_n}}\vect{y_1}{-x_n}\\
        &=cS^{-\vect{x_1}{x_n}}\vect{-y_1}{x_n}=cq^{2m+2+e}.
    \end{align*}
    This implies that $\pmphi_{z}\in\mathcal{U}$ for $z=\threevect{y_1}{x^\circ}{-x_n}$, and for  $z=\threevect{y_1}{x^\circ}{x_n}$ by the action of $\hht{b}$ with $b=\diag(1,\ldots,1,-1)$. From these two cases, we get
    $\pmphi_{x}\in\mathcal{U}$ for all $x_1\in\Ls$.

    Next, we fix $x_1$ and $x_3,\ldots,x_n$, and take $y_2$ arbitrarily. Choose $a\in\Sym_n(\oo)$ such that $\tp tat=2ut_2t_n$.
    Using the above argument again, we see that $\pmphi_{x}\in\mathcal{U}$ for all $x_2$. As $x_1$ has been chosen arbitrarily, 
    we have $\pmphi_{x}\in\mathcal{U}$ for all $x_1,x_2$.
    
    Repeating this argument, this gives that 
    $\pmphi_{x}\in\mathcal{U}$
    for all $x^{(1)}\in(\Ls)^r$.
    Finally, for each $x^{(1)}$, we have $\pmphi_{x}\in\mathcal{U}$ for all 
    $x^{(2)}\in(\lat{m+1}{-m-1}\smallsetminus\Ls)^{n-r}$ by the action of $\hht{\diag(1,\ldots,1,b_{r+1},\ldots,b_n)}$, which implies that $\mathcal{U}$ coincides with the entire space.
    \end{proof}

\begin{proposition}\label{iwahori-quotient-n>3-singular}
    For $1\leq r\leq n-1$ and $m\geq0$,
    $(\tuni{m}^{\otimes r}\otimes(\suni{m}/\tuni{m})^{\otimes n-r})^+$ and $(\tuni{m}^{\otimes r}\otimes(\suni{m}/\tuni{m})^{\otimes n-r})^-$
    are irreducible.
\end{proposition}
   
   \begin{proof}
   We use the idea of the proof of \cref{iwahori-quotient-n=2-singular}.
   Thus, we again exclude the case $(\tuni{0}^{\otimes r}\otimes(\suni{0}/\tuni{0})^{\otimes n-r})^+$. Also, we use the representation of $\wt\tilde{I}\wt^{-1}$ on $(\tuni{m}^{\otimes r}\otimes(\sunihat{m}/\tuni{m})^{\otimes n-r})^\pm$.
   
       Let $\mathcal{U}$ be an irreducible component of this space. We can assume that $\pmphi_x=\one_x\pm\one_{-x}\in\mathcal{U}$ by \cref{iwahori-separate-singular}, where $\one_x\coloneq \one_{x+2\varpi^m\oo^n}$.
       We set $x^{(1)}$, $x^{(2)}$ and $\xlow$ as defined in 
       (\ref{vect-denote}).
       We remark that $\val(x_{r+1})=\cdots=\val(x_n)=-m-1$.
       Choose $a\in\Sym_n(\oo)$ so that $\tp tat=2ut_1t_n$ for $u\in\oo^\times$. Then, the element $\chit{a}$ can be regarded as an element $(\chit{\azero},1)$ in the subgroup $\tilde{I}^{(\{1,n\})}\times\tilde{I}^{(\{2,\ldots,n-1\})}$
       with $\azero=
       \begin{bmatrix}
           1&u\\
           u&1
       \end{bmatrix}$.
       Moreover, the Weyl element $\wt=\wt_1\cdots\wt_n$ lies in this subgroup as $(\wt_{1n},\wt_2\dots\wt_{n-1})$, where $\wt_{1n}=\wt_1\wt_n$.
       Hence, the action of $\wt\chit{a}\wt=(\wt_{1n}\chit{\azero}\wt_{1n},(\wt_2\dots\wt_{n-1})^2)$ splits into the actions on $\one_{\tp(x_1,x_n)}$ and $\one_{\xlow}$. Therefore, we have
       \begin{align*}
           \wt\chit{a}\wt\one_x
           &=(\wt_{1n}\chit{\azero}\wt_{1n})\one_{\tp(x_1,x_n)}
           \otimes(\wt_2\dots\wt_{n-1})^2\one_{\xlow}\\
           &=c(\wt_{1n}\chit{\azero}\wt_{1n})\one_{\tp(x_1,x_n)}\otimes
           \one_{-\xlow}
       \end{align*}
       for an $8$-th root of unity $c$. Hence, we have
       \begin{equation}\label{hello0713}
       \begin{aligned}
           \wt\chit{a}\wt\pmphi_x
           =c(\wt_{1n}\chit{\azero}\wt_{1n})&\one_{\tp(x_1,x_n)}\otimes
           \one_{-\xlow}\\
           &\pm c(\wt_{1n}\chit{\azero}\wt_{1n})\one_{\tp(-x_1,-x_n)}\otimes
           \one_{\xlow}.
       \end{aligned}
       \end{equation}
       If $-\xlow\equiv\xlow$ mod $2\varpi^m\oo^{n-2}$, the equation (\ref{hello0713}) is 
       \begin{equation*}
           \wt\chit{a}\wt\pmphi_x=c
           (\wt_{1n}\chit{\azero}\wt_{1n})\pmphi_{\tp(x_1,x_n)}\otimes
           \one_{\xlow}.
       \end{equation*}
       The calculation of $(\wt_{1n}\chit{\azero}\wt_{1n})
       \pmphi_{\tp(x_1,x_n)}$
       is the same as that of $\wt\chit{a}\wt\pmphi_x$ in the proof of $n=2$, namely, (\ref{hello0713-3}). Hence we have
       \begin{equation*}
           \wt\chit{a}\wt\pmphi_x
           \threevect{y_1}{\xlow}{x_n}
           =
           cq^{2m+1}(\psi(-4u^{-1}x_n(x_1+y_1))\pm1)
       \end{equation*}
       for $u\in\oo^\times$.
       By the same argument in proof of \cref{iwahori-quotient-n=2-singular}, we see that 
       $\pmphi_{x}\in\mathcal{U}$ for all $x_1\in\Lt$. If 
       $-\xlow\not\equiv\xlow$ mod $2\varpi^m\oo^{n-2}$, one has $\one_{-\xlow}(\xlow)=0$, and hence 
       $$
       [\wt_{1n}\chit{\azero}\wt_{1n}\one_{\tp(x_1,x_n)}\otimes
           \one_{-\xlow}]
           \threevect{y_1}{\xlow}{x_n}
           =0.
       $$
       Hence we obtain $\wt\chit{a}\wt\pmphi_x\threevect{y_1}{\xlow}{x_n}
            =\pm cS^{-\vect{x_1}{x_n}}\vect{y_1}{x_n}=cq^{2m+2+e}$.
        This implies that $\pmphi_{x}$ is in $\mathcal{U}$ for all $x_1\in\Lt$ by \cref{iwahori-separate-singular}. 
        
        Next, we fix $x_1$ and $x_3,\ldots,x_{n-1}$, and take $y_2$ arbitrarily. Choose $a\in\Sym_n(\oo)$ so that $\tp tat=2ut_2t_n$. Using the above argument again, we see that  $\pmphi_{x}\in\mathcal{U}$ for all $x_2\in\Lt$.
        Moreover, since $x_1$ has been chosen arbitrarily, we have $\pmphi_x\in\mathcal{U}$ for all $x_1$ and $x_2$.
        
        Repeating this argument yields $\pmphi_{x}\in\mathcal{U}$ for all 
        $x^{(1)}\in(\Lt)^r$. Finally, for each $x^{(1)}$, we have $\pmphi_{x}\in\mathcal{U}$  for all $x^{(2)}\in(\lat{m}{-m-1}\smallsetminus\Lt)^{n-r}$ by the action of $\hht{\diag(1,\ldots,1,b_{r+1},\ldots,b_n)}$, which implies that $\mathcal{U}$ agrees with the entire space.

        The remaining case is 
    $(\tuni{0}^{\otimes r}\otimes(\suni{0}/\tuni{0})^{\otimes n-r})^+$.
    Recalling that $(\tuni{0})^+=\tuni{0}$ by definition, we have
    $(\tuni{0}^{\otimes r}\otimes(\suni{0}/\tuni{0})^{\otimes n-r})^+
    =\tuni{0}^{\otimes r}\otimes((\suni{0}/\tuni{0})^{\otimes n-r})^+$.
    On the other hand, $\tilde{I}=\tilde{I}^{(n)}$ contains the product of a smaller Iwahori subgroup $\tilde{I}^{(r)}\times\tilde{I}^{(n-r)}$ and it acts on each tensor factor.
    By \cref{iwahori-quotient} with $m=0$ and \cref{iwahori-bottom}, the spaces $\tuni{0}^{\otimes r}\otimes((\suni{0}/\tuni{0})^{\otimes n-r})^+$ are irreducible as representations of $\tilde{I}^{(r)}\times
    \tilde{I}^{(n-r)}$.
    Therefore, the subgroup $\tilde{I}$ acts irreducibly on the space $\tuni{0}^{\otimes r}\otimes((\suni{0}/\tuni{0})^{\otimes n-r})^+$.
    \end{proof}
    
    By the above propositions, we have an irreducible 
    decomposition of $\omega_\psi|_{\tilde{I}}$.

    \begin{proof}[Proof of \cref{iwahori-main-theorem}]
    By \cref{iwahori-quotient-n>3-normal,iwahori-quotient-n>3-singular,iwahori-bottom,iwahori-quotient}, 
    each component of the decomposition (\ref{iwahori-main-decomposition}) is irreducible.

    We show that the representation $\omega_\psi|_{\tilde{I}}$ is multiplicity-free. 
    We write the spaces $\X_m$ and $\Y_m$ as follows:
    \begin{align*}
        \X_m=\bigoplus_{j=1}^{2^n-1}X_{m,j},\quad
        \Y_m=\bigoplus_{j=1}^{2^n-1}Y_{m,j}.
    \end{align*}
    The dimension of each irreducible component is given by 
    \begin{align*}
        \dim \tuni{0}^{\otimes n}&=q^{en};\\
        \dim X_{m,j}^\pm&=\frac{1}{2}q^{(2m+e)n}(q-1)^{\ell(X_{m,j})};\\
        \dim Y_{m,j}^\pm&=\frac{1}{2}q^{(2m+e+1)n}(q-1)^{\ell(Y_{m,j})},
    \end{align*}
    where $\ell(X_{m,j})$ (resp.\ $\ell(Y_{m,j})$) is the number 
    of tensor factors equal to $\suni{m}/\tuni{m}$ 
    (resp.\ $\tuni{m+1}/\suni{m}$).
    Since $q$ and $q-1$ are coprime, it is trivial that $\dim X_{m,j}^\pm\ne\dim Y_{m,j}^\pm$.

    First, we show that $\tuni{0}^{\otimes n}\not\cong\dim Y_{m,j}^\pm$. Suppose that these are equivalent, so that these dimensions must coincide.
    In the case where $q=p^k$ with $p\geq3$, one has $en=(2n+e+1)n$, which implies $n=0$, which is a contradiction.
    In the case of $q=2^k$, it must hold that $q=2$ and hence $k=1$. Moreover, it follows from $(2m+e+1)n-1=en$ that $m=0$ and $n=1$.
    Then, the representation spaces are $\tuni{0}=\SS(\oo/2\oo)$ and $\tuni{1}/\suni{0}=\SS(\varpi^{-1}\oo/2\varpi\oo)/\SS(\oo/2\varpi\oo)$. We show these are not equivalent by checking the character of representations such as $\chi_{\tuni{0}^{\otimes n}}(g)=\Tr(g)$. 
    For $a=4\in\Sym_1(\oo)$, one has
    \begin{align*}
        \chi_{\tuni{0}^{\otimes n}}(\chit{a})&=\sum_{x\in\oo/2\oo}\psi(4x^2)=2^e;\\
        \chi_{Y_{m,j}^\pm}(\chit{a})&=\sum_{\substack{y\in(\varpi^{-1}\oo/2\varpi\oo)/\sim\\
        \val(y)=-1}}\psi(4y^2),
    \end{align*}
    where $y\sim y'$ means $\varphi_y=\varphi_{y'}$.
    Since $\psi(4y^2)$ is non-trivial, we have a bound of real part:
    \begin{equation}\label{hello0830}
    \Re\left(\chi_{Y_{m,j}^\pm}(\chit{a})\right)=\sum_{\substack{y\in(\varpi^{-1}\oo/2\varpi\oo)/\sim\\
        \val(y)=-1}}\Re\left(\psi(4y^2)\right)
        <\sum_{\substack{x\in(\varpi^{-1}\oo/2\varpi\oo)/\sim\\
        \val(y)=-1}}1=2^e.
    \end{equation}
    Therefore, we have $\chi_{\tuni{0}^{\otimes n}}\ne\chi_{Y_{m,j}^\pm}$ and hence 
    $\tuni{0}^{\otimes n}\not\cong Y_{m,j}^\pm$.

    Next, we suppose that $\tuni{0}^{\otimes n}\cong X_{m,j}^\pm$.
    If $q=2^k$, it must hold that $q=2$ and hence $k=1$. Therefore, it follows from $en=(2m+e)n-1$ that $1=2mn$, which is a contradiction. If $q=p^k$ with $p\geq3$, it must hold that $q=3$
    because $(q-1)^{\ell( Y_{m,j})}/2=1$. We also have $m=0$ from $en=(2m+e)n$.
    In this case, these representation spaces are 
    $\tuni{0}^{\otimes n}$ and $\tuni{0}^{\otimes i-1}\otimes (\suni{0}/\tuni{0})^{\pm}\otimes \tuni{0}^{\otimes n-i}$ for $1\leq i\leq n$.
    For simplicity, we assume that $i=n$ (other cases can be shown similarly).
    If $\tuni{0}^{\otimes n}\cong X_{m,j}^\pm$, there exists an intertwining operator 
    $T:\tuni{0}^{\otimes n}\rightarrow\tuni{0}^{\otimes n-1}\otimes (\suni{0}/\tuni{0})^\pm$. Then, the following diagram is commutative:
    \begin{equation}\label{hello0830-2}
        \begin{tikzcd}
            \tuni{0}^{\otimes n}\arrow[r,"\wt^{-1}"]\arrow[d, "\wt g\wt^{-1}"]
            &\tuni{0}^{\otimes n}\arrow[r,"T"]\arrow[d, "g"]
            &\tuni{0}^{\otimes n-1}\otimes(\suni{0}/\tuni{0})^\pm\arrow[r,"\wt"]\arrow[d, "g"]
            &\tuni{0}^{\otimes n-1}\otimes(\sunihat{0}/\tuni{0})^\pm\arrow[d,"\wt g\wt^{-1}"]\\
            \tuni{0}^{\otimes n}\arrow[r,"\wt^{-1}"]
            &\tuni{0}^{\otimes n}\arrow[r,"T"]
            &\tuni{0}^{\otimes n-1}\otimes(\suni{0}/\tuni{0})^\pm\arrow[r,"\wt"]
            &\tuni{0}^{\otimes n-1}\otimes(\sunihat{0}/\tuni{0})^\pm\rlap{\,.}
        \end{tikzcd}
    \end{equation}
    Hence, $\tuni{0}^{\otimes n}$ and $\tuni{0}^{\otimes n-1}\otimes(\sunihat{0}/\tuni{0})^\pm$ are equivalent 
    as the representation of $\wt\tilde{I}\wt^{-1}$. 
    Since $e=0$, the above two spaces are respectively generated by one element $\one_{0+2\oo^n}^{\otimes n}$ and $\one_{0+2\oo^{n-1}}^{\otimes n-1}\otimes \pmphi_y$ with $\val(y)=-1$. Since $\wt^2\chit{a}\in\wt\tilde{I}\wt^{-1}$, the traces of $\wt^2\chit{\varpi }$ are
    \begin{align*}
        \chi_{\tuni{0}^{\otimes n}}(\wt^2\chit{\varpi })&=\pm\gamma_1^2\psi(0)=\pm\gamma_1^2,\\
        \chi_{ X_{m,j}^\pm}(\wt^2\chit{\varpi })&=\pm\gamma_1^2\psi(\varpi\braket{(0,\ldots,0,y),(0,\ldots,0,y)})=\pm\gamma_1^2\psi(\varpi y^2 )
    \end{align*}
    by \cref{wIw-rep}.
    Since $\psi(\varpi y^2)$ is non-trivial, these do not agree, which is a contradiction.

    It follows that $ X_{m,j}^+\not \cong X_{m,j}^-$ from 
    $\chi_{ X_{m,j}^+}(\hht{-1})=\sum\beta_{-1}>0$
    and 
    $\chi_{ X_{m,j}^-}(\hht{-1})=\sum-\beta_{-1}<0$.
    We can also show that $Y_{m,j}^+\not\cong Y_{m,j}^-$.

    We suppose that $Y_{m,j}^\pm\cong Y_{m,j'}^\pm$ with $\ell( Y_{m,j})=\ell( Y_{m,j'})$. Then, 
    there exists an index $i$ such that 
    the $i$-th tensor factors of $ Y_{m,j}$ and $ Y_{m,j'}$ are 
    $\suni{m}$ and $\tuni{m+1}/\suni{m}$ respectively.
    Setting $a=4\varpi^{2m}E_{ii}$, one has
    \begin{align*}
        \chi_{Y_{m,j}^\pm}(\chit{a})&=
        \sum_{x}\psi(\tp xax)=   \sum_{x}\psi(4\varpi^{2m}x_i^2)=\dim Y_{m,j}^\pm,\\
        \chi_{Y_{m,j'}^\pm}(\chit{a})&=
        \sum_{y}\psi(\tp yay)=
        \sum_{y}\psi(4\varpi^{2m}y_j^2).
    \end{align*}
    Since $\psi(4\varpi^{2m}y_j^2)$ is non-trivial, the second sum never coincides with the first one (by the same argument of (\ref{hello0830})). So we have $Y_{m,j}^\pm\not\cong Y_{m,j'}^\pm$.

    Finally, we show that $ X_{m,j}^{\pm}\not\cong X_{m,j'}^{\pm}$.
    If these are equivalent, similarly to (\ref{hello0830-2}), the following diagram is commutative:

    \begin{equation}
        \begin{tikzcd}
            \hat{X}_{m,j}^\pm\arrow[r,"\wt^{-1}"]\arrow[d, "\wt g\wt^{-1}"]
            & X_{m,j}^\pm\arrow[r,"T"]\arrow[d, "g"]
            &X_{m,j'}^\pm\arrow[r,"\wt"]\arrow[d, "g"]
            &\hat{X}_{m,j'}^\pm\arrow[d,"\wt g\wt^{-1}"]\\
            \hat{X}_{m,j}^\pm\arrow[r,"\wt^{-1}"]
            & X_{m,j}^\pm\arrow[r,"T"]
            &X_{m,j'}^\pm\arrow[r,"\wt"]
            &\hat{X}_{m,j'}^\pm\rlap{\quad,}
        \end{tikzcd}
    \end{equation}
    where $\hat{X}_{m,j}=\wt( X_{m,j})$. Hence $\hat{X}_{m,j}\cong\hat{X}_{m,j'}$. Since we can find 
    $i$ such that $i$-th tensor component of 
    $X_{m,j}$ (resp.\ $X_{m,j'}$)
    is $\tuni{m}$ (resp.\ $\suni{m}/\tuni{m}$). For $a=4\varpi^{2m}E_{ii}$, one has
    \begin{align*}
        \chi_{\hat{X}_{m,j}^{\pm}}(\wt^2\chit{a})
        &=\sum_{x}\pm\gamma_1^2\psi(\tp xax)=
        \pm\gamma_1^2\sum_{x}\psi(4\varpi^{2m}x_i^2)=\pm\gamma_1^2\dim X_{m,j}^{\pm},\\
        \chi_{\hat{X}_{m,j'}^{\pm}}(\wt^2\chit{a})
        &=\sum_{y}\pm\gamma_1^2\psi(\tp yay)=\pm\gamma_1^2
        \sum_{y}\psi(4\varpi^{2m}y_i^2).
    \end{align*}
Since $\psi(4\varpi^{2m}y_i^2)$ is non-trivial, the second sum never coincides with $\chi_{ X_{m,j}^{\pm}}(\wt^2\chit{a})$, which contradicts to $\hat{X}_{m,j}\cong\hat{X}_{m,j'}$.
   \end{proof}

\subsection*{Tool Use}
ChatGPT (OpenAI) and Google Gemini were used to check
English grammar and improve the wording of the manuscript.
All mathematical statements and proofs were independently checked by the author, who takes full responsibility for the contents of the paper.

   \nocite{*}
   
  \printbibliography
    
\end{document}